\documentclass[leqno,12pt]{amsart} 
\usepackage{amsmath,amsopn,amssymb,amsthm,multicol}
\usepackage{hyperref}
\usepackage{color}
\usepackage{graphicx}

\DeclareMathOperator{\diag}{diag}
\DeclareMathOperator{\Ad}{Ad}
\DeclareMathOperator{\ad}{ad}
\DeclareMathOperator{\Aut}{Aut}

\DeclareMathOperator{\tr}{tr}

\DeclareMathOperator{\Ric}{Ric}

\DeclareMathOperator{\Span}{span}

\newcommand{\fr}{\mathfrak}
\newcommand{\al}{\alpha}

\newcommand{\bb}{\mathbb}

\DeclareMathOperator{\SO}{SO}
\DeclareMathOperator{\s}{S}

\DeclareMathOperator{\Sp}{Sp}
\DeclareMathOperator{\SU}{SU}
\DeclareMathOperator{\U}{U}

 \newtheorem{lemma} {Lemma} [section]
\newtheorem{theorem}[lemma]{Theorem} 
\newtheorem{remark}[lemma] {Remark} 
\newtheorem{prop} [lemma]{Proposition}

\begin{document}

\title{Invariant Einstein metrics on   $\SU(N)$ and  \\ complex Stiefel manifolds} 
\author{Andreas Arvanitoyeorgos, Yusuke Sakane  and Marina Statha}
\address{University of Patras, Department of Mathematics, GR-26500 Rion, Greece}
%\ead{arvanito@math.upatras.gr}
\email{arvanito@math.upatras.gr}
 \address{Osaka University, Department of Pure and Applied Mathematics, Graduate School of Information Science and Technology, Suita, 
Osaka 565-0871, Japan}
%\ead{sakane@math.sci.osaka-u.ac.jp}
 \email{sakane@math.sci.osaka-u.ac.jp}
\address{University of Patras, Department of Mathematics, GR-26500 Rion, Greece}
%\ead{arvanito@math.upatras.gr}
\email{statha@master.math.upatras.gr} 
\medskip
%\noindent
 %\thanks{The second  author  was supported by JSPS KAKENHI Grant Number 25400071.}

   \begin{abstract}
 We study existence of invariant Einstein metrics on complex Stiefel manifolds $G/K = \SU(\ell+m+n)/\SU(n) $ 
 and the special unitary groups $G = \SU(\ell+m+n)$. We decompose the Lie algebra $\frak g$ of $G$ and the tangent space $\frak p$ of $G/K$, by using the generalized flag manifolds $G/H = \SU(\ell+m+n)/\s(\U(\ell)\times\U(m)\times\U(n))$. 
 We parametrize scalar products on the 2-dimensional center of the Lie algebra of $H$, and we consider $G$-invariant and left invariant metrics determined by $\Ad(\s(\U(\ell)\times\U(m)\times\U(n))$-invariant scalar products on $\frak g$ and $\frak p$ respectively.   Then we compute their Ricci tensor for such metrics. 
 We prove existence of $\Ad(\s(\U(1)\times\U(2)\times\U(2))$-invariant Einstein metrics  on $V_3\bb{C}^{5}=\SU(5)/\SU(2)$, $\Ad(\s(\U(2)\times\U(2)\times\U(2))$-invariant Einstein metrics  on $V_4\bb{C}^{6}=\SU(6)/\SU(2)$, and $\Ad(\s(\U(m)\times\U(m)\times\U(n))$-invariant Einstein metrics  on $V_{2m}\bb{C}^{2m+n}=\SU(2m+n)/\SU(n)$. We also prove existence of $\Ad(\s(\U(1)\times\U(2)\times\U(2))$-invariant Einstein metrics
on the compact Lie group $\SU(5)$, which are not naturally reductive. 
The Lie group  $\SU(5)$  is the special unitary group of smallest rank known for the moment, admitting non naturally reductive Einstein metrics. Finally, we show that the compact Lie group
$\SU(4+n)$ admits two non naturally reductive $\Ad(\s(\U(2)\times\U(2)\times\U(n)))$-invariant  Einstein metrics  for $ 2 \leq n \leq 25$,  and four   non naturally reductive Einstein metrics for $n\ge 26$.  This extends previous results of K.~ Mori  about  non naturally reductive Einstein metrics  on $\SU(4+n)$  ($n \geq 2$).

   \end{abstract}

 \footnote[0]{ 


The first author was supported by a Grant from the Empirikion Foundation in Athens and the second author by JSPS KAKENHI Grant Number 16K05130.  }
  
\maketitle
 
%\today 

 \section{Introduction}
\markboth{Andreas Arvanitoyeorgos, Yusuke Sakane and Marina Statha}{Einstein metrics on   $\SU(N)$ and complex Stiefel manifolds}

A Riemannian manifold $(M, g)$ is called Einstein if it has constant Ricci curvature, i.e. $\Ric_{g}=\lambda\cdot g$ for some $\lambda\in\bb{R}$. 
  Besides the detailed exposition on Einstein manifolds  in \cite{Be}, we refer to  \cite{W1},  \cite{W2} for more recent results.
  General existence results are difficult to obtain and some methods are described
  in \cite{Bom}, \cite{BWZ} and \cite{WZ}.
  For the case of homogeneous spaces the problem of finding all invariant Einstein metrics becomes slightly more accessible, due to the possibility 
  of making symmetry assumptions, but  still it is not easy.  For example, the classification of invariant Einstein metrics for important classes of homogeneous spaces, such as the generalized Wallach spaces, was only recently achieved (\cite{CN}).
  Also,  
   for other classes of homogeneous spaces, such as the generalized flag manifolds, a complete classification of invariant Einstein metrics is still open.  We refer to \cite{A2} for more details.
 For Lie groups   the problem of determining all left-invariant Einstein metrics is also quite difficult, even if one makes geometric assumptions, such as the natural reductivity of the metric.

In the present paper we study left-invariant Einstein metrics on the compact Lie group $\SU(n)$ and $\SU(n)$-invariant Einstein metrics on the complex Stiefel manifolds $V_k\bb{C}^n=\SU(n)/\SU(n-k)$, of orthonormal $k$-frames in $\bb{C}^n$.
Two marginal cases are
   the sphere $\bb{S}^{2n-1}=\SU(n)/\SU(n-1)=V_1\bb{C}^n$ and the compact Lie group
   $\SU(n)=V_{n-1}\bb{C}^n$.
   The first is an irreducible symmetric space, therefore it admits up to scale a unique $\SU(n)$-invariant Einstein metric.

Left-invariant Einstein metrics on $\SU(n)$  have not been extensively studied.
We recall that in \cite{DZ} J.E. D'Atri and W. Ziller  found a large number of left-invariant Einstein metrics
on the compact Lie groups
  $\SU(n), \SO(n)$ and $\Sp(n)$, which are naturally reductive  
  and they posed the question whether there exist left-invariant Einstein metrics on compact Lie groups, which are not naturally reductive. 
  This is not an easy problem in general, especially when the rank of the  Lie group is small.
  For example, the number of left-invariant Einstein metrics on the Lie groups $\SU(3)$ and $\SU(2)\times\SU(2)$ is not known (however see recent progress by F. Belgum, V. Cort\'es, A.S. Haupt and D. Lindemann in \cite{Co}).
  
  In our recent work \cite{ASS2} we proved existence of left-invariant Einstein metrics on $\SO(n)$ ($n\ge 7$), which are not naturally reductive.
  The problem of finding  non naturally reductive left-invariant Einstein metrics on $\SU(n)$  
  was first considered in  the  unpublished work of K. Mori \cite{M}, 
where he proved existence  for $n\geq 6$.  
He considered $\SU(n)$ fibered over a generalized flag manifold and  
used the method  of  Riemannian submersions (cf. \cite{Be})  to compute the Ricci tensor and  to prove existence of left-invariant Einstein metrics.
However, he considered a special class of left-invariant  metrics on $\SU(n)$.
One of our main results in this paper is to prove   existence of non naturally reductive left-invariant Einstein metrics on $\SU(5)$.  We also extend Mori's result.

The first invariant Einstein metrics on the real Stiefel manifolds $V_k\bb{R}^n=\SO(n)/\SO(n-k)$ were obtained by A. Sagle in \cite{S}. Later, G. Jensen obtained additional Einstein metrics on   $V_k\bb{R}^n$ as well as on the quaternionic Stiefel manifolds $V_k\bb{H}^n=\Sp(n)/\Sp(n-k)$ (\cite{J2}). 
 In the works \cite{ADN1}, \cite{ADN2}, \cite{ADN3} the first author, V.V. Dzhepko and Yu. G. Nikonorov 
proved existence of new invariant Einstein metrics on $V_k\bb{R}^n$ and $V_k\bb{H}^n$, 
by making certain symmetry assumptions.
The method was extended by the authors in \cite{ASS1} and \cite{ASS3}
and obtained  additional invariant Einstein metrics on these spaces.

 Invariant Einstein metrics on complex Stiefel manifolds have not been studied before.
Since the isotropy representation of
   $V_k\bb{C}^n$ contains equivalent irreducible subrepresentations,
the  search for invariant Einstein metrics on such homogeneous spaces $G/H$, is quite difficult.
In fact, a  complete description of the set of all $G$-invariant metrics, and in turn the computation of the Ricci tensor of $G/H$ is  complicated. 
Some other works where the authors studied  invariant Einstein metrics for such type of homogeneous spaces,  are \cite{Ke} by M. Kerr, and \cite{Ni1}, \cite{Ni2} by Yu.G. Nikonorov.
Also,  in the previous mentioned works  \cite{ADN1}, \cite{ADN2}, \cite{ADN3}, \cite{ASS1} and \cite{ASS3}, the Einstein metrics  were obtained
by using the generalized Wallach spaces $G/H=\SO(\ell +m+n)/(\SO(\ell)\times\SO(m)\times\SO(n))$ or
$\Sp(\ell +m+n)/(\Sp(\ell)\times\Sp(m)\times\Sp(n))$, where the dimension of the center of the Lie algebra of $H$ is  at most 1.
For the complex Stiefel manifolds $G/K=\SU(\ell +m+n)/\SU(n)$ we find $\SU(\ell +m+n)$-invariant Einstein metrics by using the generalized Wallach space $G/H=\SU(\ell +m+n)/\s(\U(\ell)\times\U(m)\times\U(n))$ (a generalized flag manifold).
In this case the dimension of the center of the Lie algebra of $H$ is 2, which makes the description of invariant metrics more complicated.

In the present work we give a unified treatment for finding left-invariant Einstein metrics on the Lie group $G=\SU(\ell+m+n)$, which are not naturally reductive, as well as $\SU(\ell+m+n)$-invariant Einstein metrics on the Stiefel manifold $G/K=\SU(\ell+m+n)/\SU(n)$.

   Our approach is the following:
  We consider the generalized flag manifold $G/H=\SU(\ell+m+n)/\s(\U(\ell)\times\U(m)\times\U(n))$ whose tangent space decomposes into a direct sum of irreducible and inequivalent submodules $\fr{m}=\fr{m}_{12}\oplus\fr{m}_{13}\oplus\fr{m}_{23}$.  
  We decompose the Lie algebra of $H$ into its center $\frak{h}_0$ and simple ideals $\frak{h}_1, \frak{h}_2, \frak{h}_3$. Then the Lie algebra of $G$ decomposes into a direct sum $\frak{g}=\fr{h}_0\oplus\frak{h}_1\oplus \frak{h}_2\oplus\frak{h}_3\oplus\fr{m}_{12}\oplus\fr{m}_{13}\oplus\fr{m}_{23}$. Also the tangent space of the
  Stiefel manifold $G/K$ decomposes as 
 $\frak{p}=\fr{h}_0\oplus\frak{h}_1\oplus \frak{h}_2\oplus\fr{m}_{12}\oplus\fr{m}_{13}\oplus\fr{m}_{23}$.
 Then we parametrize all scalar products in the center $\fr{h}_0$ by further decomposing $\fr{h}_0=\fr{h}_4\oplus\fr{h}_5$ into one-dimensional ideals, and  then consider appropriate $\Ad(\s(\U(\ell)\times\U(m)\times\U(n))$-invariant scalar products on $\fr{g}$ and $\fr{p}$. These scalar products  determine left-invariant metrics on $G$, and $G$-invariant metrics on $G/H$ respectively.
  
 Next, we pursue with the computation of the Ricci tensor for such metrics, which
  consists of a non diagonal part at the center  $\frak{h}_0$  and a diagonal part
 at $\frak{h}_1\oplus \frak{h}_2\oplus\frak{h}_3\oplus\fr{m}_{12}\oplus\fr{m}_{13}\oplus\fr{m}_{23}$.
 We introduce the numbers $\big\lbrace
\begin{smallmatrix}
i\\
jk
\end{smallmatrix}
\big\rbrace$, which generalize the well known numbers 
$\left[
\begin{smallmatrix}
i\\
jk
\end{smallmatrix}
\right]$ introduced by M. Wang and W. Ziller in \cite{WZ}.  
 As a result, we obtain explicit expressions for the Ricci tensor in terms of the variables of the metric and $\ell, m, n$, so 
the Einstein equation reduces to an algebraic system of equations $r_0=r_1=\cdots = r_8=0$ with parameters $\ell, m, n$.
By making a suitable choice of a basis for the center of the Lie algebra of $\s(\U(\ell)\times \U(m)\times\U(n))$,
some of the equations become linear with respect to some variables (cf. Subsection 5.2).
 We also take $\ell=1, m=2$,
and then we use Gr\"obner bases methods and arguments using the resultant of polynomials, to obtain explicit solutions, or prove existence of positive solutions for such systems.

 For the case of the  complex Stiefel manifold  $\SU(p+n)/\SU(n)$ some of the $\SU(p+n)$-invariant Einstein metrics are  obtained from solutions of quadratic equations.  We call these Einstein metrics of {\it Jensen's type}, 
because they are of the form $g=B|_{\fr{m}}+s^2B|_{\fr{h}_0}+t^2B|_{\fr{su}(p)}$, on the total space of fibrations 
  $\SU(p+n)/\SU(n) \to \SU(p+n)/\s(\U(p)\times \U(n))$, where $\fr{m}$ is the orthogonal complement of $\fr{s}(\fr{u}(p)+\fr{u}(n)) $ in $\fr{su}(p+n)$, $\fr{h}_0$ is the center of the Lie algebra of $\fr{s}(\fr{u}(p)+\fr{u}(n))$ and $B$ is the negative of the Killing form of $\fr{g}$ (cf. \cite{J2}).

 \smallskip
Our results for the special unitary group are the following:

\begin{theorem}
 The compact Lie group $\SU(5)$ admits left-invariant Einstein metrics which are not naturally reductive.
   \end{theorem}
   
 %  We also extend Mori's result as follows:
   
   \begin{theorem} The compact Lie group $\SU(4+n)$ admits at least two non naturally reductive left-invariant Einstein metrics for $2\le n\le 25$ and  four
    non naturally reductive left-invariant Einstein metrics for $n\ge 26$.  
     \end{theorem}

%   It is possible to show that  the compact Lie group $\SU(n+3)$ admits {\it two} left-invariant non naturally reductive Einstein metrics, which correspond to  $\Ad(\s(\U(1)\times\U(2)\times\U(n)))$-invariant inner products of the form {\rm (\ref{metrics1})},  for $2\leq n\leq 12$. Also, we {\it conjecture} that for $n\geq 13$, $\SU(n+3)$ admits   {\it four} left-invariant non naturally reductive Einstein metrics.
  Our results for the  complex Stiefel manifold are the following:

\begin{theorem} {\rm 1)}  The complex Stiefel manifold $V_{2}\bb{C}^{4}=\SU(4)/\SU(2)$ admits 
two $\Ad($ $\s(\U(1)\times\U(1)\times\U(2))$-invariant Einstein metrics  which are of Jensen's type.

\noindent
{\rm 2)}  The complex Stiefel manifold $V_{3}\bb{C}^{5}=\SU(5)/\SU(2)$ admits 
four $\Ad(\s(\U(1)\times\U(2)\times\U(2))$-invariant Einstein metrics, two of these are of Jensen's type.

\noindent
{\rm 3)} The complex Stiefel manifold $V_{4}\bb{C}^{6}=\SU(6)/\SU(2)$ admits 
eight $\Ad(\s(\U(2)\times\U(2)\times\U(2))$-invariant Einstein metrics,  two of these are of Jensen's type.
\end{theorem}

\begin{theorem}
 The complex Stiefel manifolds $V_{2m}\bb{C}^{2m + n}$ ($m\geq 2$) admit at least two $\Ad(\s(\U(m)\times\U(m)\times\U(n)))$-invariant Einstein metrics which are not of Jensen's type,  for certain infinite values of $m$ and $n$.  
 \end{theorem}
 
% We  {\it conjecture}  that the Stiefel manifolds $V_2\mathbb{C}^{n+2}=\SU(n+2)/\SU(n)$  admit  precisely {\it two} invariant Einstein merics, which are of Jensen's type.  This is the analogue of the real Stiefel manifolds $V_2\mathbb{R}^{n+2}=\SO(n+2)/\SO(n)$, which had been studied before by other authors (eg. \cite{Ke}). 

\section{The Ricci tensor for reductive homogeneous spaces}

Let $G$ be a compact semisimple Lie group, $K$ a connected closed subgroup of $G$  and  
let  $\frak g$ and $\fr{k}$  be  the corresponding Lie algebras. 
The Killing form  of $\frak g$ is negative definite, so we can define an $\mbox{Ad}(G)$-invariant inner product $B$ on 
  $\frak g$, where $B$ is {\it the negative of the Killing form} of $\frak{g}$. 
Let $\frak g$ = $\frak k \oplus
\frak m$ be a reductive decomposition of $\frak g$ with respect to $B$ so that $\left[\,\frak k,\, \frak m\,\right] \subset \frak m$ and
$\frak m\cong T_o(G/K)$.

Any $G$-invariant metric $g$ on $G/K$ is determined by an $\Ad(K)$-invariant scalar product $\langle\ ,\ \rangle$ on $\frak m$.
Let $\{X_j\}$ be a $\langle\ ,\ \rangle$-orthonormal basis of $\frak m$.  Then the Ricci tensor  $r$ of the metric $g$ is given as follows (\cite[p. 381]{Be}):
\begin{eqnarray}\label{r}
& & r(X, Y) = -\frac12\sum_i\langle [X, X_i], [Y, X_i]\rangle +\frac12B(X, Y)
       +\frac14 \sum _{i, j}\langle [X_i, X_j], X\rangle\langle [X_i, X_j], Y\rangle.\label{ric1}
\end{eqnarray}

If the isotropy representation of $G/K$ 
is decomposed into a sum of non equivalent irreducible summands, then we will also use an alternative
expression for the Ricci tensor, which we describe next.  
Let
\begin{equation}\label{iso}
{\frak m} = {\frak m}_1 \oplus \cdots \oplus {\frak m}_q,
\end{equation} 
be a decomposition into mutually non equivalent irreducible $\mbox{Ad}(K)$-modules. 
Then any $G$-invariant metric on $G/K$ is determined by the scalar product  
\begin{eqnarray}
 \langle  \,\,\, , \,\,\, \rangle  =  
x_1   B|_{\mbox{\footnotesize$ \frak m$}_1} + \cdots + 
 x_q   B|_{\mbox{\footnotesize$ \frak m$}_q},  \label{eq2}
\end{eqnarray}
for positive real numbers $(x_1, \dots, x_q)\in\bb{R}^{q}_{+}$.  Note that  $G$-invariant symmetric covariant 2-tensors on $G/K$ are 
of the same form as the Riemannian metrics (although they  are not necessarily  positive definite).  
 In particular, the Ricci tensor $r$ of a $G$-invariant Riemannian metric on $G/K$ is of the same form as (\ref{eq2}), that is 
 \[
 r=z_1 B|_{\mbox{\footnotesize$ \frak m$}_1}  + \cdots + z_{q} B|_{\mbox{\footnotesize$ \frak m$}_q} ,
 \]
 for some real numbers $z_1, \ldots, z_q$.

Let $\lbrace e_{\alpha} \rbrace$ be a $B$-orthonormal basis 
adapted to the decomposition of $\frak m$,    i.e. 
$e_{\alpha} \in {\frak m}_i$ for some $i$, and
$\alpha < \beta$ if $i<j$. 
We put ${A^\gamma_{\alpha
\beta}}= B \left(\left[e_{\alpha},e_{\beta}\right],e_{\gamma}\right)$ so that
$\left[e_{\alpha},e_{\beta}\right]%_{\frak n}
= {\sum_{\gamma}
A^\gamma_{\alpha \beta} e_{\gamma}}$ and set 
${k \brack {ij}}=\sum (A^\gamma_{\alpha \beta})^2$, where the sum is
taken over all indices $\alpha, \beta, \gamma$ with $e_\alpha \in
{\frak m}_i, e_\beta \in {\frak m}_j, e_\gamma \in {\frak m}_k$ (cf. \cite{WZ}).  
Then the positive numbers $ {k \brack {ij}}$ are independent of the 
$B$-orthonormal bases chosen for ${\frak m}_i, {\frak m}_j, {\frak m}_k$,
and 
%\begin{equation}
${k \brack {ij}}={k \brack {ji}}= {j \brack {ki}}.  
 \label{eq3}
$
We call these numbers {\it $B$-structure constants.}

Let $ d_k= \dim{\frak m}_{k}$. Then we have the following:

\begin{lemma}\label{ric2}\textnormal{(\cite{PS})}
The components ${ r}_{1}, \dots, {r}_{q}$ 
of the Ricci tensor ${r}$ of the metric $ \langle  \,\,\, , \,\,\, \rangle $ of the
form {\em (\ref{eq2})} on $G/K$ are given by 
\begin{equation}
{r}_k = \frac{1}{2x_k}+\frac{1}{4d_k}\sum_{j,i}
\frac{x_k}{x_j x_i} {k \brack {ji}}
-\frac{1}{2d_k}\sum_{j,i}\frac{x_j}{x_k x_i} {j \brack {ki}}
 \quad (k= 1,\ \dots, q),    \label{eq51}
\end{equation}
where the sum is taken over $i, j =1,\dots, q$.
\end{lemma} 
Since by assumption the submodules $\fr{m}_{i}, \fr{m}_{j}$ in the decomposition (\ref{iso}) are mutually non equivalent for any $i\neq j$, it is $r(\fr{m}_{i}, \fr{m}_{j})=0$ whenever $i\neq j$.  Thus by Lemma \ref{ric2} it follows that    $G$-invariant Einstein metrics on $M=G/K$ are exactly the positive real solutions $g=(x_1, \ldots, x_q)\in\bb{R}^{q}_{+}$  of the  polynomial system $\{r_1=\lambda, \ r_2=\lambda, \ \ldots, \ r_{q}=\lambda\}$, where $\lambda\in \bb{R}_{+}$ is the Einstein constant.

\section{Invariant metrics on $\SU(\ell+m+n)$ and on $V_{\ell+m}\mathbb{C}^{\ell+m+n} = \SU(\ell+m+n)/\SU(n)$}
\subsection{Decomposition of tangent spaces}
We will describe decompositions of the tangent spaces of the Lie group $\SU(\ell+m+n)$ and the Stiefel manifold $\SU(\ell +m+n)/\SU(n)$ at corresponding identity elements, which will be convenient for our study.
We consider the homogeneous space $G/H=\SU(\ell+m+n)/\s (\U(\ell)\times \U(m)\times\U(n))$, which is a complex generalized flag manifold.
It is known that the  isotropy representation of $G/H$ is a direct sum of three non equivalent subrepresentations, hence
the tangent space $\frak m$ of $G/H$ at $eH$ decomposes into three non equivalent $\Ad(H)$-submodules
$
\frak m = \frak m_{12}\oplus\frak m_{13}\oplus\frak m_{23},
$
 given by

\begin{eqnarray*}
\fr{m}_{12} &=&
\Big\{  
\begin{pmatrix}
0 & A & 0\\
-\bar{A^t} & 0 & 0\\
0 & 0 & 0
\end{pmatrix}
: A\in M_{\ell , m}\mathbb{C}
\Big\},\quad
\fr{m}_{13} =
\Big\{  
\begin{pmatrix}
0 & 0 & B\\
 0 & 0 & 0\\
 -\bar{B^t} & 0 & 0
\end{pmatrix}
: B\in M_{\ell, n}\mathbb{C}
\Big\},\\
\fr{m}_{23} &=&
\Big\{  
\begin{pmatrix}
0 & 0 & 0\\
0 & 0 & C\\
0 & -\bar{C^t} & 0
\end{pmatrix}
: C\in M_{m, n}\mathbb{C}
\Big\},
\end{eqnarray*}
where $M_{\ell, m}\mathbb{C}$   denotes the set of all $\ell\times m$ complex matrices.
In fact, $\frak m$ is given by $\frak k ^\perp$ in $\frak g =\frak{su}(\ell+m+n)$ with respect to $B$. 

 Let
$\fr{h}=\fr{h}_0\oplus\fr{h}_1\oplus\fr{h}_2\oplus\fr{h}_3$ be the decomposition of $\fr{h}$, the Lie algebra of $H$, into its $2$-dimensional center $\fr{h}_0$ and simple ideals,   given by
\begin{eqnarray*}
\fr{h}_0&=&\Big\{  \sqrt{-1}
\begin{pmatrix}
\frac{a_1}{\ell}I_\ell & 0 & 0\\
0 & \frac{a_2}{m}I_m & 0\\
0 & 0 & \frac{a_3}{n}I_n
\end{pmatrix}
: a_1+a_2+a_3=0, \ a_1, a_2, a_3 \in\mathbb{R}
\Big\},\\
\fr{h}_1 &=&\Big\{  
\begin{pmatrix}
A_1 & 0 & 0\\
0 & 0 & 0\\
0 & 0 & 0
\end{pmatrix}
: A_1\in\fr{su}(\ell)
\Big\},\quad
\fr{h}_2 = \Big\{  
\begin{pmatrix}
0 & 0 & 0\\
0 & A_2 & 0\\
0 & 0 & 0
\end{pmatrix}
: A_2\in\fr{su}(m)
\Big\},\\
\fr{h}_3 &=&\Big\{  
\begin{pmatrix}
0 & 0 & 0\\
0 & 0 & 0\\
0 & 0 & A_3
\end{pmatrix}
:A_3\in\fr{su}(n)
\Big\}.
\end{eqnarray*}
Then the Lie algebra $\fr{g}$ splits into $\fr{h}$ and three $\Ad(H)$-irreducible modules as 
\begin{equation}\label{decg}
\fr{g}=\fr{h}\oplus\fr{m}=\fr{h}_0\oplus\fr{h}_1\oplus\fr{h}_2\oplus\fr{h}_3\oplus\fr{m}_{12}
\oplus\fr{m}_{13}\oplus\fr{m}_{23}.
\end{equation}

\noindent
This is an orthogonal decomposition with respect to $B$.

Let 
\begin{eqnarray*}
H_4=\sqrt{-1}
\begin{pmatrix}
\frac{b_1}{\ell+m}I_\ell & 0 & 0\\
0 & \frac{b_1}{\ell+m}I_m & 0\\
0 & 0 & -\frac{b_1}{n}I_n
\end{pmatrix}\ \ \mbox{and}\ \ 
H_5= \sqrt{-1}
\begin{pmatrix}
\frac{b_2}{\ell}I_\ell & 0 & 0\\
0 & -\frac{b_2}{m}I_m & 0\\
0 & 0 & 0
\end{pmatrix},
\end{eqnarray*}
where\ $\displaystyle b_1=a_1+a_2, b_2= (m a_1-\ell a_2)/(\ell+m)$, 
and
consider the $B$-orthogonal decomposition  
$
\fr{h}_0 = \fr{h}_4\oplus\fr{h}_5
$,
where
\begin{eqnarray*}
\fr{h}_4= \Span\{H_4\},\ 
\fr{h}_5 = \Span\{H_5\}.
\end{eqnarray*}

Then  decomposition (\ref{decg}) becomes
\begin{equation}\label{dec1}
\fr{g} = \fr{h}_1\oplus \fr{h}_{2} \oplus\fr{h}_3\oplus\fr{h}_4\oplus\fr{h}_5\oplus\fr{m}_{12}\oplus\fr{m}_{13}\oplus\fr{m}_{23}.
\end{equation}

We also consider the complex Stiefel manifold $G/K=\SU(\ell+m+n)/\SU(n)$ and the
$\Ad(K)$-invariant decomposition of its tangent space $\frak{p}$ at $eK$, given by
\begin{equation}\label{decp}
\fr{p} = \fr{h}_1\oplus \fr{h}_{2} \oplus\fr{h}_4\oplus\fr{h}_5\oplus\fr{m}_{12}\oplus\fr{m}_{13}\oplus\fr{m}_{23}.
\end{equation}

By a direct computation we obtain the following:

\begin{lemma}\label{brackets}  The submodules in the decompositions $(\ref{dec1})$, $(\ref{decp})$ satisfy the following bracket relations:
\begin{center}
\begin{tabular}{lll}
$[ \fr{h}_i, \fr{h}_{i}] \subset \fr{h}_{i}, (i=1,2,3)$  &  $[\fr{h}_0, \fr{h}_{i}] =  (0), (i=1, 2, 3)$,& $[ \fr{h}_i,\fr{h}_{j}] = (0), (i\ne j)$, \\
$[ \fr{h}_1, \fr{m}_{12}] \subset \fr{m}_{12},$  & $[ \fr{h}_1, \fr{m}_{13}] \subset \fr{m}_{13},$  &   $[\fr{h}_2, \fr{m}_{12}]\subset\fr{m}_{12},$    \\
$[\fr{h}_2, \fr{m}_{23}] \subset  \fr{m}_{23}$, &
 $[ \fr{h}_3, \fr{m}_{13}] \subset \fr{m}_{13},$  &$[\fr{h}_3, \fr{m}_{23}] \subset  \fr{m}_{23}$,\\
$[\fr{h}_k, \fr{m}_{ij}] = (0), (k=1,2, k\ne i, j)$ &
$[\fr{h}_4, \fr{m}_{ij}] \subset \fr{m}_{ij},$   &   $[\fr{h}_5, \fr{m}_{ij}]\subset\fr{m}_{ij},$  
\\ 
  $[\fr{m}_{12}, \fr{m}_{13}] \subset \frak m_{23},$ &
$[\fr{m}_{12}, \fr{m}_{23}] \subset \frak m_{13}$, &
$[\fr{m}_{13}, \fr{m}_{23}] \subset \frak m_{12}$, \\ 
$[\fr m_{12}, \fr m_{12}]\subset\fr h_{1}\oplus\fr h_2 \oplus \fr h_4\oplus \fr h _5$, &   $[\fr m_{13}, \fr m_{13}]\subset\fr h_{1}\oplus\fr h_3 \oplus \fr h_4\oplus \fr h _5$,\\ 
$[\fr m_{23}, \fr m_{23}]\subset\fr h_{2}\oplus\fr h_3 \oplus \fr h_4\oplus \fr h _5$. 
\end{tabular}
\end{center}  
  \end{lemma}
  
   Therefore,  we see that the only non zero $B$-structure constants (up to permutation of indices) are 
   
  \begin{equation}\label{triplets1}
{1 \brack {11}},  {2 \brack {22}},   {3 \brack {33}}, {(12) \brack {1(12)}},  {(13) \brack {1(13)}},  
   {(12) \brack {2(12)}},    {(23) \brack {2(23)}},  
   {(13) \brack {3(13)}}, {(23) \brack {3(23)}},
\end{equation}

  \begin{equation}\label{triplets2}
 {(12) \brack {4(12)}},  {(23) \brack {4(23)}},  
   {(13) \brack {4(13)}},    {(13) \brack {5(13)}},  {(23) \brack {5(23)}},  
   {(12) \brack {5(12)}},  {(13) \brack {(12)(23)}}.
\end{equation}

In order to compute the triplets ${i \brack {jk}}$ we need the following lemma from
 \cite{ADN1} adjusted to our case (for a more detailed proof see also \cite[Lemma 5.2]{ASS1}).

\begin{lemma}\label{lemma5.2aa}  Let $\fr{q}$ be a simple subalgebra of $\frak g=\fr{su}(N)$.  Consider an orthonormal basis $\{ f_j \}$  of $\fr{q}$ 
 with respect to $B$ {\em({\em negative of the Killing form of $\fr{su}(N)$})}, and denote by
 $B_{\fr{q}}$ the Killing form of $\fr{q}$.  Then,  for $i = 1, \ldots, \dim \fr{q}$, we have 
\begin{equation*}\label{eq14a}
\sum_{j, k = 1}^{\dim \fr{q}} \left(B([f_i, f_j], f_k \right)^2 = {\al}^{\fr{q}}_{\fr{su}(N)},
\end{equation*}
where ${\al}^{\fr{q}}_{\fr{su}(N)}$ is the constant determined by $B_{\fr{q}} = {\al}^{\fr{q}}_{\fr{su}(N)}\cdot B|_{\fr{q}}$. 
\end{lemma}

\begin{lemma}\label{ijk}
Let $N=\ell+m+n$.  Then the following expressions are valid:
{\small
\begin{equation*}\label{eq14}
\begin{array}{llll} 
\vspace{0.3cm}
 \displaystyle{{1 \brack {1 1}} = \frac{\ell (\ell^2 -1)}{N} },   &  \displaystyle{{2 \brack {2 2}} = 
 \frac{m (m^2 -1)}{N} }, &  \displaystyle{{3 \brack {3 3}} = \frac{n (n^2-1)}{N} },
  &  \displaystyle{{(12) \brack {1 (12)}} = \frac{m (\ell^2-1)}{N} },\\
  \vspace{0.3cm}
\displaystyle{{(13) \brack {1 (13)}} = \frac{n (\ell^2 -1)}{N} },   
&  \displaystyle{{(12) \brack {2 (12)}} = \frac{\ell (m^2-1)}{N} },
  &  \displaystyle{{(23) \brack {2 (23)}} = \frac{n (m^2-1)}{N} },
  &  \displaystyle{{(12) \brack {4 (12)}} = 0 }, \\ 
  \vspace{0.3cm}
    \displaystyle{{(13) \brack {4 (13)}} = \frac{\ell}{\ell+m} }, 
  & \displaystyle{{(23) \brack {4 (23)}} = \frac{m}{\ell+m} }, 
  &  \displaystyle{{(12) \brack {5 (12)}} = 
 \frac{\ell+m}{N} },
 &  \displaystyle{{(13) \brack {5 (13)}} = \frac{mn}{N(\ell+m)} }, \\
  \vspace{0.3cm}
    \displaystyle{{(23) \brack {5 (23)}} = \frac{\ell n}{N (\ell+m)} },
  & \displaystyle{{(23) \brack {(12) (13)}} = \frac{\ell mn}{N} },
  &  \displaystyle{{(13) \brack {3 (13)}} = \frac{\ell(n^2-1)}{N} },
  &  \displaystyle{{(23) \brack {3 (23)}} = \frac{m(n^2-1)}{N} }.
\end{array} 
\end{equation*}
}
\end{lemma} 
\begin{proof} Let $\fr g = \fr{su}(\ell+m+n)$,  $\fr q=\fr{su}(\ell)$ with corresponding  Killing forms  
$B_{\fr g}(X, Y)=2(\ell+m+n)\tr (XY)$,  $B_{\fr q}=2\ell\tr (XY)$ respectively.
Let $\{f_j\}$ be an orthonormal basis of $\fr q$ with respect to $-B_{\fr g}$ $(1\le j\le \dim\fr q)$.
Then $B_{\fr{q}} = {\al}^{\fr{q}}_{\fr{g}}\cdot \left.B_{\fr g}\right|_{\fr{q}}$ with
${\al}^{\fr{q}}_{\fr{g}}=\frac{\ell}{\ell+m+n}$. Then we have
$$
\displaystyle{{1 \brack {1 1}}}=\sum _{i=1}^{\ell^2-1}\sum _{j, k=1}^{\dim\fr q}B_\fr{g}([f_i, f_j], f_k)^2
=(\dim\fr q){\al}^{\fr{q}}_{\fr{g}}=\frac{\ell(\ell^2-1)}{N}.
$$
The triplets ${2 \brack {2 2}}$ and ${3 \brack {3 3}}$ can be computed in a similar manner, by choosing $\fr q=\fr{su}(m)$ and $\fr q=\fr{su}(n)$ respectively.

Now let $\fr q=\fr{su}(\ell+m)$ and  $\{f_j\}$ be an orthonormal basis of $\fr q$ with respect to $B_{\fr g}$.
Since $\fr{su}(\ell+m)=\fr{su}(\ell)\oplus\fr m_{12}\oplus\fr{su}(m)\oplus\fr{h}_5$, we adapt the basis $\{f_j\}$ to this decomposition as follows:
$\{f_1, \dots , f_{\ell^2-1}\}\in\fr{su}(\ell)$, $\{f_{\ell^2}, \dots , f_{(\ell^2-1)+2\ell m}\}\in\fr{m}_{12}$,
$\{f_{\ell^2+2\ell m}, \dots , f_{\ell^2+2\ell m+m^2-2}\}\in\fr{su}(m)$, $f_{(\ell+m)^2-1}\in\fr{h}_5$.
Then
$$
\sum _{j, k=1}^{(\ell+m)^2-1}B_\fr{g}([f_i, f_j], f_k)^2
={\al}^{\fr{su}(\ell+m)}_{\fr{g}}=\frac{\ell+m}{N}.
$$
For $\{f_i: i=1, \dots , \ell^2-1\}\in\fr{su}(\ell)$ we have
$$
\sum _{i=1}^{\ell^2-1}\left(
\sum _{j, k=1}^{(\ell+m)^2-1}B_\fr{g}([f_i, f_j], f_k)^2\right)=\frac{\ell+m}{N}(\ell^2-1),
$$
and for $\{f_i: i=1, \dots , \ell^2-1\}\in\fr{su}(\ell)$, $\{f_j: j=1, \dots , m^2-1\}\in\fr{su}(m)$ we have that
$$
[f_i, f_j]\in
\begin{cases}
\fr{su}(\ell), f_j\in\fr {su}(\ell)\\
\fr m_{12}, f_j\in\fr m_{12}\\
0, \ \ \ f_j\in\fr{su}(m).
\end{cases}
$$
Therefore,
$$
\displaystyle{{1 \brack {1 1}}}+\displaystyle{{1 \brack {(12) (12)}}}+0+0=\frac{\ell+m}{N}(\ell^2-1),
$$
from which it follows that
$$
\displaystyle{{1 \brack {(12) (12)}}}=\frac{m}{N}(\ell^2-1).
$$
The other $B$-structure constants can be computed in an analogous way.
\end{proof}  

\subsection{A parametrization of invariant metrics} %%page6%%%

We now consider left-invariant metrics on $\SU(\ell+m+n)$ determined by 
the
$\Ad(H)$-invariant
scalar products on $\frak g =\frak{su}(l+m+n)$.
Note that in the decomposition (\ref{decg}) the $\Ad(H)$-irreducible modules $\fr{h}_1, \fr{h}_2, \fr{h}_3, \fr{m}_{12},
\fr{m}_{13}$ and $\fr{m}_{23}$ are mutually non equivalent.
So any $\Ad(H)$-invariant scalar product $h$ on $\fr{su}(\ell+m+n)$ can be expressed in the form
\begin{equation}\label{inv_metrics}
h=\beta|_{\fr{h}_0} +u_1B|_{\frak h_1}+u_2B|_{\frak h_2}+u_3B|_{\frak h_3}+\sum^{}_{i < j} x_{ij}B|_{\frak m_{ij}},\quad u_i, x_{ij}>0,
\end{equation}
where $\beta|_{\fr{h}_0}$ is a scalar product on $\fr{h}_0$.

Let $\langle\langle\cdot, \cdot\rangle\rangle$ be an arbitrary scalar product on $\fr{h}_0$.
Then the matrix of $\langle\langle\cdot, \cdot\rangle\rangle$ with respect to the basis $\{H_4, H_5\}$
can be orthogonally diagonalized $\diag\{v _4, v_5\}$, for some positive numbers $v_4, v_5$.
Then there is an orthonormal basis $\{V_4, V_5\}$ with respect to some scalar product
$\langle\cdot, \cdot\rangle$ of $\fr{h}_0$, such that
\begin{equation}\label{prod}
\langle\langle\cdot, \cdot\rangle\rangle = v_4\left.\langle\cdot ,\cdot \rangle\right|_{\tilde{\frak h}_4} +v_5\left.\langle\cdot ,\cdot \rangle\right|_{\tilde{\frak h}_5}, \quad v_4, v_5 >0,
\end{equation}  
where  $\tilde{\fr{h}}_4 = \Span\{V_4\}$ and $\tilde{\fr{h}}_5 = \Span\{V_5\}$.
The basis $\{V_4, V_5\}$ is related to the basis $\{H_4, H_5\}$ by 
$$
(V_4, V_5) = (H_4, H_5) \begin{pmatrix}
p & q\\
r & s
\end{pmatrix},\quad \mbox{where}\ 
\begin{pmatrix}
p & q\\
r & s
\end{pmatrix}\in{\rm Gl}_2\mathbb R^+,
$$
 hence
\begin{equation}\label{change}
(H_4, H_5)=(V_4, V_5)
\begin{pmatrix}
a & b\\
c & d
\end{pmatrix},\quad
\begin{pmatrix}
a & b\\
c & d
\end{pmatrix}
=
\begin{pmatrix}
p & q\\
r & s
\end{pmatrix}^{-1}.
\end{equation}
Therefore, the matrix of the scalar product
$\langle\langle\cdot, \cdot\rangle\rangle$ with respect to $\{H_4, H_5\}$ is given by
\begin{equation}\label{positive}
A\equiv {}^{t}_{}\!\begin{pmatrix}
a & b\\
c & d
\end{pmatrix}
\begin{pmatrix}
v_4 & 0\\
0 & v_5
\end{pmatrix}
\begin{pmatrix}
a & b\\
c & d
\end{pmatrix}.
\end{equation}
Now, for any 
$\bigl(
\begin{smallmatrix}
a & b\\
c & d
\end{smallmatrix}\bigr)
\in{\rm Gl}_2\mathbb{R}
$ and for any 
$v_4, v_5>0$
the matrix $A$ is positive definite, 
therefore, any scalar product 
 $\beta|_{\fr{h}_0}$ on $\frak h_0$ has the form
 (\ref{prod}).

We now consider decomposition (\ref{dec1}) 
\begin{equation}\label{dec2}
\frak g =\frak h_1\oplus\frak h_2\oplus\frak h_3\oplus
\tilde{\frak h}_4\oplus
\tilde{\frak h}_5\oplus\frak m_{12}
\oplus\frak m_{13}\oplus\frak m_{23}\equiv \frak n_1\oplus
\frak n_2\oplus\cdots\oplus\frak n_7\oplus\frak n_8,
\end{equation}
and left-invariant metrics on $\SU(\ell+m+n)$ determined by 
the
$\Ad(\s(\U(\ell)\times\U(m)\times\U(n)))$-invariant
scalar products on $\frak g =\frak{su}(\ell +m+n)$ of the form
$$
g_1=\langle\langle\cdot, \cdot\rangle\rangle|_{\fr{h}_0} +u_1B|_{\frak h_1}+u_2B|_{\frak h_2}+u_3B|_{\frak h_3}+\sum^{}_{i < j} x_{ij}B|_{\frak m_{ij}},\quad u_i, x_{ij}>0,
$$
or
\begin{eqnarray}\label{metrics1}
g_1&=&v_4\left.\langle\cdot ,\cdot \rangle\right|_{\tilde{\frak h}_4} +
 v_5\left.\langle\cdot ,\cdot \rangle\right|_{\tilde{\frak h}_5} +u_1B|_{\frak h_1}+u_2B|_{\frak h_2}+u_3B|_{\frak h_3}+\sum^{}_{i < j} x_{ij}B|_{\frak m_{ij}}\\
 &=& 
 u_1B|_{\frak n_1}+u_2B|_{\frak n_2}+u_3B|_{\frak n_3}
 +v_4\left.\langle\cdot ,\cdot \rangle\right|_{\frak n_4}
 +v_5\left.\langle\cdot ,\cdot \rangle\right|_{\frak n_5} + \nonumber\\ & + &
 x_{(6)}B|_{\frak n_{6}}
 +x_{(7)}B|_{\frak n_{7}}
 +x_{(8)}B|_{\frak n_{8}},\nonumber
\end{eqnarray}

\noindent
where we have set 
%$$
%x_{(i)} = 
%\begin{cases}
%x_{12}, \ \ \mbox{if} \ i = 6\\
%x_{13}, \ \ \mbox{if}\ i = 7\\
%x_{23}, \ \ \mbox{if}\ i = 8.
%\end{cases}
%$$
$x_{(6)}=x_{12}, x_{(7)}=x_{13}, x_{(8)}=x_{23}.
$
%Notice that left-invariant metrics (\ref{metrich}) are special case of metrics (\ref{metrics1}).

\smallskip
\noindent
We also consider $\SU(\ell+m+n)$-invariant metrics on the Stiefel manifolds $\SU(\ell+m+n)/\SU(n)$
determined by the $\Ad(\s(\U(\ell)\times\U(m)\times\U(n)))$-invariant scalar products on 
\begin{equation}\label{dec3}
\fr{p} = \fr{n}_1\oplus \fr{n}_{2} \oplus\fr{n}_4\oplus\fr{n}_5\oplus\fr{n}_{6}\oplus\fr{n}_{7}\oplus\fr{n}_{8}
\end{equation}
 of the form
 \begin{equation}\label{metrics2}
g_2=
 u_1B|_{\frak n_1}+u_2B|_{\frak n_2}
 +v_4\left.\langle\cdot ,\cdot \rangle\right|_{\frak n_4}
 +v_5\left.\langle\cdot ,\cdot \rangle\right|_{\frak n_5}
 +x_{(6)}B|_{\frak n_{6}}
 +x_{(7)}B|_{\frak n_{7}}
 +x_{(8)}B|_{\frak n_{8}}.
\end{equation}

%\medskip
\noindent
Let $\{\tilde{H}_4, \tilde{H}_5\}$ be an orthonormal basis of $\fr{h}_0$ with respect to $B$, where
$$
\tilde{H}_4 = c_1H_4, \quad \tilde{H}_5 = c_2 H_5, \ \ c_1, c_2\in \bb{R}.
$$
An easy calculation gives that 
\begin{equation}\label{ci}
c_1 = \displaystyle{\frac{\sqrt{(\ell+m)n}}{(\ell+m+n)\sqrt{2}}}, 
\quad 
c_2 = \displaystyle{\frac{\sqrt{\ell m}}{\sqrt{2(\ell+m+n)}\sqrt{m+\ell}}}. 
\end{equation}
Also, note that 
 $\langle\langle \tilde{H}_4, \tilde{H}_5\rangle\rangle$ might be   non zero. 
 Let $\{U_4, U_5\}$ be an orthonormal basis of $\fr{h}_0=\fr{h}_4\oplus\fr{h}_5$ with respect to $g_1$, where
$$
U_{i} = \frac{1}{\sqrt{v_{i}}}V_{i}, \ i = 4, 5.
$$
Let $\{\tilde{X}_{j}^{(i)}: j = 1, \dots , \dim\frak n_i\}$ be an orthonormal basis of $\fr{n}_{i}$, $i = 1,2,3,6,7,8$ with respect to $B$. If we  set
\begin{eqnarray}
X_{j}^{(i)} &=& 
\begin{cases}
\displaystyle{\frac{1}{\sqrt{u_{i}}}\tilde{X}_{j}^{(i)}}, \ i = 1,2,3  \\ \\
\displaystyle{\frac{1}{\sqrt{x_{(i)}}}\tilde{X}_{j}^{(i)}}, \ i=6,7,8 
%\displaystyle{\frac{1}{\sqrt{x_{(7)}}}\tilde{X}_{j}^{(7)}}  \\ 
%\displaystyle{\frac{1}{\sqrt{x_{(8)}}}\tilde{X}_{j}^{(8)} } \\ 
\end{cases}
\end{eqnarray} 
then the set
$\{X_{j}^{(i)}: j=1, \dots , \dim\frak n_i\}$ is
 an orthonormal basis of $\fr{n}_{i}$, $i = 1,2,3,6,7,8$ (resp. $i = 1,2, 6,7,8$) with respect to $g_1$ (resp. $g_2$).

Note that the scalar product $g_1$ is not in general bi-invariant, because
$$
g_1([\tilde{X}_{j}^{(i)}, \tilde{X}_{l}^{(i)}]\, ,\, U_{k})\neq g_1(\tilde{X}_{j}^{(i)}\, ,\, [\tilde{X}_{l}^{(i)}, U_{k}]), \ \ \mbox{for}\ \ k = 4, 5.
$$

Therefore, it is convenient to express the scalar product (\ref{prod}) in terms of $B$.
We have the following:

\begin{prop}\label{prop1}
For every $X_{j}^{(i)}, X_{l}^{(i)}$, $i = 6,7,8$  the following relations are satisfied:
\begin{eqnarray*}
\langle\langle[X_{j}^{(i)}, X_{l}^{(i)}]\, ,\, U_4\rangle\rangle &=& \sqrt{v_4}\big\lbrace aB([X_{j}^{(i)}, X_{l}^{(i)}]\, ,\, \tilde{H}_{4}) + bB([X_{j}^{(i)}, X_{l}^{(i)}]\, ,\, \tilde{H}_5)\big\rbrace,\\
\langle\langle[X_{j}^{(i)}, X_{l}^{(i)}]\, ,\, U_{5}\rangle\rangle &=&  \sqrt{v_5}\big\lbrace cB([X_{j}^{(i)}, X_{l}^{(i)}]\, ,\, \tilde{H}_{4}) + dB([X_{j}^{(i)}, X_{l}^{(i)}]\, ,\, \tilde{H}_5)\big\rbrace.
\end{eqnarray*}
\end{prop}
\begin{proof}
For the first relation we have:
\begin{eqnarray*}
&&\langle\langle[X_{j}^{(i)}, X_{l}^{(i)}]\, ,\, U_4\rangle\rangle = v_4\langle[X_{j}^{(i)}, X_{l}^{(i)}]\, ,\, U_4\rangle = \sqrt{v_4}\langle[X_{j}^{(i)}, X_{l}^{(i)}]\, ,\, V_4\rangle\\
&=&\sqrt{v_4}\big\langle B([X_{j}^{(i)}, X_{l}^{(i)}]\, ,\, \tilde{H}_4)\tilde{H}_4 + B([X_{j}^{(i)}, X_{l}^{(i)}]\, ,\, \tilde{H}_5)\tilde{H}_5 \, , \, V_4\big\rangle\\
&=  &\sqrt{v_4}\big(\langle B([X_{j}^{(i)}, X_{l}^{(i)}]\, ,\, \tilde{H}_4)\tilde{H}_4\, ,\, V_1\rangle + \langle B([X_{j}^{(i)}, X_{l}^{(i)}]\, ,\, \tilde{H}_5)\tilde{H}_5 \, , \, V_4\rangle\big)\\
&= &\sqrt{v_4}\big(\langle B ([X_{j}^{(i)}, X_{l}^{(i)}]\, ,\, \tilde{H}_4)\tilde{H}_4\, , \, p\tilde{H}_4 + r\tilde{H}_5\rangle 
+ \langle B ([X_{j}^{(i)}, X_{l}^{(i)}]\, ,\, \tilde{H}_5)\tilde{H}_5 \, ,\, p\tilde{H}_4 + r\tilde{H}_5\rangle\big)\\
&= & \sqrt{v_4}\big(B([X_{j}^{(i)}, X_{l}^{(i)}]\, ,\, \tilde{H}_4)\big(\langle\tilde{H}_4, p\tilde{H}_4\rangle + \langle\tilde{H}_4, r\tilde{H}_5\rangle\big) \\ & &
 + B([X_{j}^{(i)}, X_{l}^{(i)}]\, ,\, \tilde{H}_5)\big(\langle\tilde{H}_5, p\tilde{H}_4\rangle + \langle \tilde{H}_5, r\tilde{H}_5\rangle\big)\big). 
\end{eqnarray*}
By using the relations
 $\langle\tilde{H}_4, \tilde{H}_4\rangle = a^2 + c^{2}$, \, $\langle\tilde{H}_4, \tilde{H}_5\rangle = ab + cd$ and $\langle\tilde{H}_5, \tilde{H}_5\rangle = b^2 + d^2$, the right-hand side in the last equation above can be written as
\begin{eqnarray*}
&&  \sqrt{v_4}\big(B([X_{j}^{(i)}, X_{l}^{(i)}]\, ,\, \tilde{H}_4)\big(p(a^2 + c^2) + r(ab + cd)\big)\\&+& B([X_{j}^{(i)}, X_{l}^{(i)}]\, ,\, \tilde{H}_5)\big(p(ab +cd) + r(b^2 + d^2)\big)\big).
\end{eqnarray*}
On the other hand, by (\ref{change}) it follows that
$
p = \frac{d}{ad - bc}$, 
$r = \frac{-c}{ad - bc}$, 
$q = \frac{-b}{ad - bc}$, 
$s = \frac{a}{ad - bc}$,
therefore, we finally obtain that
\begin{eqnarray*}
&&\langle\langle[X_{j}^{(i)}, X_{l}^{(i)}]\, ,\, U_4\rangle\rangle =
\sqrt{v_4}\big(aB([X_{j}^{(i)}, X_{l}^{(i)}]\, ,\, \tilde{H}_4) + bB([X_{j}^{(i)}, X_{l}^{(i)}]\, ,\, \tilde{H}_5)\big).
\end{eqnarray*}
The second relation can be proved by a similar manner.
\end{proof}

To state the following lemma (which we will use shortly in the next section) we need to choose orthonormal Weyl bases for the modules $\fr{m}_{12}, \fr{m}_{13}$ and $\fr{m}_{23}$.
Let $E_{ij}$ denote the $N\times N$ matrix with $1$ in the $(i, j)$-entry and $0$ elsewhere, and define the matrices
$A_{ij}=E_{ij}-E_{ij}$, $B_{ij}=\sqrt{-1}(E_{ij}+E_{ij})$.

\begin{lemma}\label{lemma1}
The following Lie bracket relations are satisfied:
\begin{itemize}
\item[(1)] If $A_{ij}, B_{ij}\in \fr{n}_6$, then
$
[\tilde{H}_4, A_{ij}] = [\tilde{H}_4, B_{ij}]=0, \quad [\tilde{H}_{5}, A_{ij}] = \displaystyle{c_2\big(\frac{1}{\ell} + \frac{1}{m}\big)B_{ij}}, 
\newline\ \ \ [\tilde{H}_{5}, B_{ij}] = \displaystyle{-c_2\big(\frac{1}{\ell} + \frac{1}{m}\big)A_{ij}}.
$
\item[(2)] If $A_{ij}, B_{ij}\in \fr{n}_7$, then
$
[\tilde{H}_4, A_{ij}] = \displaystyle{c_1\big(\frac{1}{n} + \frac{1}{\ell + m}\big)B_{ij}}, \newline
[\tilde{H}_4, B_{ij}] = \displaystyle{-c_1\big(\frac{1}{n} + \frac{1}{\ell + m}\big)A_{ij}},\ \ \ [\tilde{H}_{5}, A_{ij}] = \displaystyle{c_2\frac{1}{\ell}B_{ij}}, 
\quad [\tilde{H}_{5}, B_{ij}] = \displaystyle{-c_2\frac{1}{\ell}A_{ij}}$.
\item[(3)] If $A_{ij}, B_{ij}\in \fr{n}_8$, then
$[\tilde{H}_4, A_{ij}] = \displaystyle{c_1\big(\frac{1}{n} + \frac{1}{\ell + m}\big)B_{ij}}, \newline
 [\tilde{H}_4, B_{ij}] = \displaystyle{-c_1\big(\frac{1}{n} + \frac{1}{\ell + m}\big)A_{ij}},
\ \ \  [\tilde{H}_{5}, A_{ij}] = \displaystyle{-c_2\frac{1}{\ell}B_{ij}},
\quad [\tilde{H}_{5}, B_{ij}] = \displaystyle{c_2\frac{1}{\ell}A_{ij}}.
$
\end{itemize}
\end{lemma}

%%%%%%%%%%%%%%%%%%%%%%%%%%%%%%%%%%%%%%%%%%%%%%Compute the numbers {ijk}

\section{The Ricci tensor for  left-invariant metrics on $\SU(\ell+m+n)$ and invariant metrics on $\SU(\ell+m+n)/\SU(n)$}
Note that any $\Ad(H)$-invariant symmetric bilinear form of $\fr{su}(N)$ can be expressed as
$$
\gamma|_{\fr{h}_0} +w_1B|_{\frak h_1}+w_2B|_{\frak h_2}+w_3B|_{\frak h_3}+\sum^{}_{i < j} w_{ij}B|_{\frak m_{ij}},\quad w_i, w_{ij}\in\mathbb{R}.
$$
where $\gamma|_{\fr{h}_0}$ is a symmetric bilinear form on $\fr{h}_0$.
In particular, the Ricci tensor of the metrics (\ref{metrics1}) and (\ref{metrics2})  is of the same form
 (for metrics (\ref{metrics2}) the term $w_3B|_{\frak h_3}$ is omitted).  Hence we divide
its study to the part in the center of $\fr{h}_0$ and its diagonal part.

\subsection{The Ricci tensor for the center part $\frak h_0$ of the scalar products (\ref{metrics1}),  (\ref{metrics2})} %page11
To compute the Ricci tensor of the metrics corresponding to the invariant scalar products (\ref{metrics1}) and (\ref{metrics2}), for the center part $\frak h_0$, we will use formula (\ref{r}).
We know that $[\fr{h}_0, \fr{n}_{j}] \subset \fr{n}_{j}$, $j = 6, 7, 8$ and $[\fr{h}_0, \fr{n}_{i}] = 0$, $i = 1,2,3$.  
%%%% 
Then we have the following:
\begin{prop}\label{prop2}  %prop4.1
 Let $g$ denote any of the  scalar products {\em (\ref{metrics1})},  {\em (\ref{metrics2})} and let $Z, W\in\fr{h}_0$.  Then it is  %for any  $X_{j}^{(i)}\in \fr{n}_{i}$, $i = 6,7,8$ it is
\begin{eqnarray}
\sum_{i = 6,7,8}\sum_{j=1}^{\dim\frak n_i}g([Z, X_{j}^{(i)}]\, ,\, [W, X_{j}^{(i)}]) = B(Z, W).
\end{eqnarray}
\end{prop}
\begin{proof}
Let $\{X_{j}^{(i)}: j =1, \dots ,\dim\frak n_i\}$ be the orthonormal basis of $\fr{n}_{i}$, $i = 6,7,8$ with respect to $g$.  Then:
\begin{eqnarray*}
& & \sum_{i = 6,7,8}\sum_{j}g([Z, X_{j}^{(i)}]\, ,\, [W, X_{j}^{(i)}]) = \sum_{i = 6,7,8}\sum_{j}x_{(i)}B([Z, X_{j}^{(i)}]\, ,\, [W, X_{j}^{(i)}])\\
\end{eqnarray*}
\begin{eqnarray*}
&=& \sum_{i = 6,7,8}\sum_{j}x_{(i)}B([Z\, ,\, \frac{1}{\sqrt{x_{(i)}}}\tilde{X}_{j}^{(i)}]\, ,\, [W, \frac{1}{\sqrt{x_{(i)}}}\tilde{X}_{j}^{(i)}]) 
= \sum_{i = 6,7,8}\sum_{j}B([Z, \tilde{X}_{j}^{(i)}]\, ,\, [W, \tilde{X}_{j}^{(i)}])\\
\end{eqnarray*}
\begin{eqnarray*}
&=& \sum_{i = 6,7,8}\sum_{j}B(\ad(Z)\tilde{X}_{j}^{(i)}\, ,\, \ad(W)\tilde{X}_{j}^{(i)})
 = \sum_{i = 6,7,8} \sum_{j}-B(\tilde{X}_{j}^{(i)}\, ,\, \ad(Z)\ad(W)\tilde{X}_{j}^{(i)})\\
& = &-{\rm tr}(\ad(Z)\circ\ad(W)) = B(Z, W). \end{eqnarray*}
\end{proof}
 In order to compute the Ricci tensor for the invariant metrics (\ref{metrics1}) 
 we define new numbers, which we call them {\it $Q$-structure constants}, by
$$
\bigg\lbrace
\begin{matrix}
k\\
ij
\end{matrix}
\bigg\rbrace = \sum_{\al, \beta, \gamma} Q([\tilde{X}_{\al}^{(i)}, \tilde{X}_{\beta}^{(j)}]\, ,\, 
\tilde{X}_{\gamma}^{(k)})^2,  
$$  
\noindent
where
$$
Q=\left.\langle\ ,\ \rangle\right|_{\tilde{\frak h}_4} +
 \left.\langle\ ,\ \rangle\right|_{\tilde{\frak h}_5} +B|_{\frak h_1}+B|_{\frak h_2}+B|_{\frak h_3}+\sum B|_{\frak m_{ij}}.
$$
  Note that it is
$
\bigg\lbrace
\begin{matrix}
k\\
ij
\end{matrix}
\bigg\rbrace = 
\bigg\lbrace
\begin{matrix}
k\\
ji
\end{matrix}
\bigg\rbrace
$, but 
$ 
\bigg\lbrace
\begin{matrix}
k\\
ij
\end{matrix}
\bigg\rbrace$ is not always equal to
$ 
\bigg\lbrace
\begin{matrix}
j\\
ik
\end{matrix}
\bigg\rbrace.
$
However, in view of  decomposition (\ref{dec2}) we have the following relations:

\begin{equation}
\begin{array}{lll}
\bigg\lbrace
\begin{matrix}
1\\
11
\end{matrix}
\bigg\rbrace = 
\displaystyle{1 \brack 11}, & 
\bigg\lbrace
\begin{matrix}
2\\
22
\end{matrix}
\bigg\rbrace =
\displaystyle{2\brack 22},  & 
\bigg\lbrace
\begin{matrix}
3\\
33
\end{matrix}
\bigg\rbrace =
\displaystyle{3\brack 33}, \label{eqn}
\\ 
\bigg\lbrace
\begin{matrix}
6\\
78
\end{matrix}
\bigg\rbrace = 
\displaystyle{6\brack 78}, &
\bigg\lbrace
\begin{matrix}
j\\
ij
\end{matrix}
\bigg\rbrace = 
\displaystyle{j\brack ij},&
 \ \mbox{for} \ i = 1,2,3 \ \mbox{and} \ j = 6, 7, 8.
\end{array}
\end{equation}

\noindent
\begin{remark}
{\rm  For the case of complex Stiefel manifolds we consider the decomposition (\ref{dec3}) and metrics
(\ref{metrics2}).  Then the term $B|_{\frak h_3}$ in $Q$ is omitted
and 
$
\lbrace
\begin{smallmatrix}
\ast\\
\ast  \ast
\end{smallmatrix}
\rbrace =0
$
if there is number $3$ at any place $*$.}
\end{remark}

%\smallskip
For the  center $\fr{h}_0$ we need to compute the following numbers:
\begin{equation}\label{numbc}
\bigg\lbrace
\begin{matrix}
4\\
66
\end{matrix}
\bigg\rbrace, \ \ 
\bigg\lbrace
\begin{matrix}
4\\
77
\end{matrix}
\bigg\rbrace, \ \ 
\bigg\lbrace
\begin{matrix}
4\\
88
\end{matrix}
\bigg\rbrace, \ \ 
\bigg\lbrace
\begin{matrix}
5\\
66
\end{matrix}
\bigg\rbrace, \ \ 
\bigg\lbrace
\begin{matrix}
5\\
77
\end{matrix}
\bigg\rbrace, \ \ 
\bigg\lbrace
\begin{matrix}
5\\
88
\end{matrix}
\bigg\rbrace.
\end{equation}

Let $\tilde{A}_{ij}=\mu A_{ij}$, $\tilde{B}_{ij}=\mu B_{ij}$ be $B$-orthonormal vectors of $\SU(N)$, for some
real constant $\mu$. 
Then the sets $\{\tilde{A}_{ij}, \tilde{B}_{ij}: i=\ell+1, \dots ,\ell +m; j=1,\dots ,\ell\}$, 
$\{\tilde{A}_{ij}, \tilde{B}_{ij}: i=\ell+m+1, \dots ,N; j=1,\dots ,\ell\}$ and
$\{\tilde{A}_{ij}, \tilde{B}_{ij}: i=\ell+m+1, \dots ,N; j=m+1,\dots ,m+\ell\}$ constitute orthonormal bases
for $\fr{n}_6=\fr{m}_{12}$, $\fr{n}_7=\fr{m}_{13}$ and $\fr{n}_8=\fr{m}_{23}$ respectively.

\begin{lemma}\label{thenumbers} %Lemm4.3%
 Let $N=\ell+m+n$. The numbers $(\ref{numbc})$ are given as follows:
\begin{eqnarray}\label{equ_2}
\bigg\lbrace
\begin{matrix}
4\\
66
\end{matrix}
\bigg\rbrace &=& \frac{b^{2}(\ell + m)}{N}, \qquad
\bigg\lbrace
\begin{matrix}
5\\
66
\end{matrix}
\bigg\rbrace 
=\frac{d^{2}(\ell + m)}{N}  \nonumber
\\
\bigg\lbrace
\begin{matrix}
4\\
77
\end{matrix}
\bigg\rbrace &=& \frac{a^{2}\ell}{\ell + m} + \frac{b^{2}mn}{N(\ell + m)} + \frac{2ab\sqrt{\ell m n}}{(\ell + m)\sqrt{N}}  \nonumber
\\
\bigg\lbrace
\begin{matrix}
4\\
88
\end{matrix}
\bigg\rbrace &=& \frac{a^{2}m}{\ell + m} + \frac{b^{2}\ell n}{(\ell + m)} - \frac{2ab\sqrt{\ell m n}}{(\ell + m)\sqrt{N}}  
\end{eqnarray}
\begin{eqnarray}
\bigg\lbrace
\begin{matrix}
5\\
77
\end{matrix}
\bigg\rbrace &=& \frac{c^{2}\ell}{\ell + m} + \frac{d^{2}mn}{(\ell + m)} + \frac{2cd\sqrt{\ell m n}}{(\ell + m)\sqrt{N}}  
\nonumber\\
%\end{eqnarray}
%\begin{eqnarray}
\bigg\lbrace 
\begin{matrix}
5\\
88
\end{matrix}
\bigg\rbrace &=& \frac{c^{2}m}{\ell + m} + \frac{d^{2}\ell n}{(\ell + m)} - \frac{2cd\sqrt{\ell m n}}{(\ell + m)\sqrt{N}}.
\nonumber
\end{eqnarray}
\end{lemma}
\begin{proof}
We will prove the first relation and the others can be calculted similarly.  It is
\begin{eqnarray*}
\bigg\lbrace
\begin{matrix}
4\\
66
\end{matrix}
\bigg\rbrace &=& \sum _{j,l} Q([\tilde{X}_{j}^{(6)}, \tilde{X}_{l}^{(6)}]\, ,\, V_{4})^{2} \\
&=& \sum \left( aB([\tilde{X}_{j}^{(6)}, \tilde{X}_{l}^{(6)}]\, ,\, \tilde{H}_4)  + bB([\tilde{X}_{j}^{(6)}, \tilde{X}_{l}^{(6)}]\, ,\, \tilde{H}_5)\right)^{2} \\
&=& \sum \left( a^{2}B([\tilde{X}_{j}^{(6)}, \tilde{X}_{l}^{(6)}]\, ,\, \tilde{H}_4)^{2} + b^{2}B([\tilde{X}_{j}^{(6)}, \tilde{X}_{l}^{(6)}]\, ,\, \tilde{H}_5)^{2} \right.\\
&&\left. + 2abB([\tilde{X}_{j}^{(6)}, \tilde{X}_{l}^{(6)}]\, ,\, \tilde{H}_4)B([\tilde{X}_{j}^{(6)}, \tilde{X}_{l}^{(6)}]\, ,\, \tilde{H}_5) \right)\\
&=& \sum \left( a^{2}B(\tilde{X}_{j}^{(6)}\, ,\, [\tilde{H}_{4}, \tilde{X}_{l}^{(6)}])^{2} + b^{2}B(\tilde{X}_{j}^{(6)}\, ,\, [\tilde{H}_{5}, \tilde{X}_{l}^{(6)}])^{2} \right.\\
&&\left. +2abB(\tilde{X}_{j}^{(6)}\, ,\, [\tilde{H}_{4}, \tilde{X}_{l}^{(6)}])B(\tilde{X}_{j}^{(6)}\, ,\, [\tilde{H}_{5}, \tilde{X}_{l}^{(6)}])\right)\\
&=& \sum b^{2}B(\tilde{X}_{j}^{(6)}\, ,\, [\tilde{H}_{5}, \tilde{X}_{l}^{(6)}])^{2} \\
&=& \sum_{\substack{\ell+1\le i,k\le\ell+m\\
     1\le j,l\le\ell}}b^2B(\tilde{A}_{ij}, [\tilde{H}_5, \tilde{A}_{kl}])^2+ 
\sum_{\substack{\ell+1\le i,k\le\ell+m\\
     1\le j,l\le\ell}} b^2B(\tilde{A}_{ij}, [\tilde{H}_5, \tilde{B}_{kl}])^2\\
&& +\sum_{\substack{\ell+1\le i,k\le\ell+m\\
     1\le j,l\le\ell}} b^2B(\tilde{B}_{ij}, [\tilde{H}_5, \tilde{A}_{kl}])^2
 + \sum_{\substack{\ell+1\le i,k\le\ell+m\\
     1\le j,l\le\ell}} b^2B(\tilde{B}_{ij}, [\tilde{H}_5, \tilde{B}_{kl}])^2\\
&=& b^{2} c_2^{2}(\frac{1}{\ell} + \frac{1}{m})^2\cdot 2\ell m.
\end{eqnarray*}
In the second equation above we used (\ref{prod}) and Proposition \ref{prop1},   in the
forth equation we used the bi-invariance of the Killing form and in the fifth and seventh equations we used Lemma \ref{lemma1}. 
By substituting $c_2$ from (\ref{ci}) in the last equation we obtain the desired expression.
\end{proof}

\begin{prop}\label{ricci1}  %prop 4.4
The components of the Ricci tensor of the left-invariant metric corresponding to the scalar product {\em (\ref{metrics1})} and of the $\SU(\ell+m+n)$-invariant metrics corresponding to the scalar products
{\em (\ref{metrics2})} for the center $\frak h_0$, are given as follows:
\begin{eqnarray}
r_4 &=& \frac{v_4}{4}\left(\frac{1}{{x_{(6)}}^{2}}\bigg\lbrace
\begin{matrix}
4\\
66
\end{matrix}
\bigg\rbrace + \frac{1}{{x_{(7)}}^{2}}\bigg\lbrace
\begin{matrix}
4\\
77
\end{matrix}
\bigg\rbrace + \frac{1}{{x_{(8)}}^{2}}\bigg\lbrace
\begin{matrix}
4\\
88
\end{matrix}
\bigg\rbrace \right)\label{r4}
\end{eqnarray}
\begin{eqnarray}
r_5 &=& \frac{v_5}{4}\left(\frac{1}{{x_{(6)}}^{2}}\bigg\lbrace
\begin{matrix}
5\\
66
\end{matrix}
\bigg\rbrace + \frac{1}{{x_{(7)}}^{2}}\bigg\lbrace
\begin{matrix}
5\\
77
\end{matrix}
\bigg\rbrace + \frac{1}{{x_{(8)}}^{2}}\bigg\lbrace
\begin{matrix}
5\\
88
\end{matrix}
\bigg\rbrace \right)\label{r5}%\\
\end{eqnarray}
\begin{eqnarray}
&& r_0 = \frac{\sqrt{v_4 v_5}}{4} \Bigg\lbrace \frac{bd}{{x_{(6)}}^{2}}\frac{(\ell + m)}{(\ell + m + n)}  \\
&&+ \frac{1}{{x_{(7)}}^{2} (\ell + m)}\left(\ell a c + \frac{\sqrt{\ell m n}}{\sqrt{(\ell + m + n)}}(ad + cb) 
+ \frac{bdmn}{(\ell + m + n)}\right)   \nonumber\\
&&+ \frac{1}{{x_{(8)}}^{2} (\ell + m)} \left( mac - \frac{\sqrt{\ell m n}}{\sqrt{(\ell + m + n)}}(ad + cb)
+ \frac{bdn\ell}{\sqrt{(\ell + m + n)}}\right) \Bigg\rbrace.\label{r0} \nonumber
\end{eqnarray}
\end{prop}
\begin{proof}
We will work with the left-invariant metrics $(\ref{metrics1})$ on the Lie group $\SU(\ell+m+n)$.
Let $U_4\in\tilde{\fr{h}}_4$, where $U_4 = \displaystyle{\frac{1}{\sqrt{v_4}}}V_4$.  Then
by using equation (\ref{r})
we have 
 \begin{eqnarray*}
&&r_4 = r(U_4, U_{4}) = -\frac{1}{2}\sum_{i \neq 4,5} 
\sum_{j=1}^{\dim\frak n_i}g_1([U_4, X_{j}^{(i)}]\, ,\, [U_4, X_{j}^{(i)}]) - \frac{1}{2}\, g_1([U_4, U_5]\, ,\, [U_4, U_5]) \\
&& +\frac{1}{2}\, B(U_4, U_4) 
 +\frac14 \sum_{i, k } 
\sum_{j=1}^{\dim\frak n_i} \sum_{l=1}^{\dim\frak n_k} g_1([X_{j}^{(i)}, X_{l}^{(k)}]\, ,\, U_4)\, g_1([X_{j}^{(i)}, X_{l}^{(k)}]\, ,\, U_{4})\\
&&= -\frac12\, B(U_4, U_4) + \frac12\, B(U_4, U_4) 
 + \frac{1}{4}\sum_{i, k }
\sum_{j=1}^{\dim\frak n_i}\sum_{l=1}^{\dim\frak n_k} g_1([X_{j}^{(i)}, X_{l}^{(k)}]\, ,\, U_{4})^{2}\\
&&= \frac{1}{4}\sum_{i, k } 
\sum_{j=1}^{\dim\frak n_i}\sum_{l=1}^{\dim\frak n_k}  g_1([X_{j}^{(i)}, X_{l}^{(k)}]\, ,\, U_{4})^{2},\\
\end{eqnarray*}
where the first term in the second equality was obtained 
by Proposition \ref{prop2}.  
We will simplify the last term in the above expression for $r_4$.  It is
\begin{eqnarray*}
\frac{1}{4}\sum_{i, k}\sum_{j, l} g_1([X_{j}^{(i)}, X_{l}^{(k)}]\, ,\, U_{4})^{2} = \frac{1}{4}\sum_{i} \sum_{j, l} g_1([X_{j}^{(i)}, X_{l}^{(i)}]\, ,\, U_4)^{2} + \frac{1}{4} \sum_{\substack{j, l \\ i\neq k}}g_1([X_{j}^{(i)}, X_{l}^{(k)}]\, ,\, U_{4})^{2} & & \\
 +\frac{1}{2} \sum_{i}\sum_{l} g_1([U_4, X_{l}^{(i)}]\, ,\, U_4)^{2} + \frac{1}{2} \sum_{i}\sum_{l} g_1([U_5,  X_{l}^{(i)}]\, ,\, U_4)^{2}. & & 
\end{eqnarray*}  
By using the Lie bracket relations of Lemma \ref{brackets} 
it follows that the  last three terms in the above sum are equal to zero.  
For the first term we have the following:

If $i = 1, 2, 3$ then $[X_{j}^{(i)}, X_{l}^{(i)}] \subset   \fr{n}_i$ and the term vanishes.
If $i = 6, 7, 8$ then $[X_{j}^{(i)}, X_{l}^{(i)}]\subset \fr{n}_1\oplus\cdots\oplus\fr{n}_5$,
so 
\begin{eqnarray*}
& & \frac14 \sum_{i}\sum_{j, l} g_1([X_{j}^{(i)}, X_{l}^{(i)}]\, ,\, U_4)^{2} = \frac14 \sum_{i}\sum_{j, l}\langle\langle [X_{j}^{(i)}, X_{l}^{(i)}]\, ,\, U_4\rangle\rangle^{2} \\
&& = \frac{1}{4v_1}\sum_{i=6,7,8} \sum_{j, l} \langle\langle[X_{j}^{(i)}, X_{l}^{(i)}]\, ,\, V_4\rangle\rangle^{2}  
  = \sum_{i=6,7,8} \frac{v_4}{4{x_{(i)}}^{2}} \sum_{j, l} \langle [\tilde{X}_{j}^{(i)}, \tilde{X}_{l}^{(i)}]\, ,\, V_{4}\rangle^{2} \\
&& = \frac{v_4}{4{x_{(6)}}^{2}}\bigg\lbrace
\begin{matrix}
4\\
66
\end{matrix}
\bigg\rbrace + 
\frac{v_4}{4 {x_{(7)}}^{2}}\bigg\lbrace
\begin{matrix}
4\\
77
\end{matrix}
\bigg\rbrace +
\frac{v_4}{4 {x_{(8)}}^{2}}\bigg\lbrace
\begin{matrix}
4\\
88
\end{matrix}
\bigg\rbrace, 
\end{eqnarray*}
from which (\ref{r4}) follows.  By similar computations we obtain (\ref{r5}).

\smallskip
Now let $U_4\in\tilde{\fr{h}}_4$ and $U_5\in\tilde{\fr{h}}_5$.  Then
\begin{eqnarray*}
&& r_0 = r(U_4, U_5)  = -\frac{1}{2}\sum_{i \neq 4,5 }\sum_{ j =1}^{\dim\fr n_i} g_1([U_4, X_{j}^{(i)}]
 \, ,\, [U_{5}, X_{j}^{(i)}]) \\  
 & &- \frac{1}{2} g_1([U_4, U_5]\, ,\, [U_4, U_5]) - \frac{1}{2} g_1([U_4, U_4]\, ,\, [U_4, U_5]) + \frac12 B(U_4, U_5) \\ &&
 +\frac{1}{4 }\sum_{i, k}\sum_{j=1}^{\dim\fr n_i}\sum_{ l=1}^{\dim\fr n_k} g_1([X_{j}^{(i)}, X_{l}^{(k)}]\, ,\, U_4) %\times \\& & 
 g_1([X_{j}^{(i)}, X_{l}^{(k)}]\, ,\, U_5) 
= -\frac{1}{2} B(U_4, U_5) +\frac{1}{2} B(U_4, U_5)  \\ &&
+ \frac{1}{4 }\sum_{i, k}\sum_{j=1}^{\dim\fr n_i}\sum_{ l=1}^{\dim\fr n_k}g_1([X_{j}^{(i)}, X_{l}^{(k)}]\, ,\, U_4) g_1([X_{j}^{(i)}, X_{l}^{(k)}]\, ,\, U_5). 
\end{eqnarray*}

We will simplify the last term in the above equation.  We have:
%%%%%
\begin{eqnarray*}
& & \frac{1}{4 }\sum_{i, k}\sum_{j=1}^{\dim\fr n_i}\sum_{ l=1}^{\dim\fr n_k} g_1([X_{j}^{(i)}, X_{l}^{(k)}]\, ,\, U_4) g_1([X_{j}^{(i)}, X_{l}^{(k)}]\, ,\, U_5) \\
 & & = \frac{1}{4}  \sum_{\substack{i} }\sum_{j,l=1}^{\dim\fr n_i} g_1([X_{j}^{(i)}, X_{l}^{(i)}]\, ,\, U_4) g_1([X_{j}^{(i)}, X_{l}^{(i)}]\, ,\, U_5)  \\
&& + \frac{1}{4} \sum_{i \neq k}\sum_{j=1}^{\dim\fr n_i}\sum_{ l=1}^{\dim\fr n_k} g_1([X_{j}^{(i)}, X_{l}^{(k)}]\, ,\, U_4) g_1([X_{j}^{(i)}, X_{l}^{(k)}]\, ,\, U_5) \\
& & + \frac{1}{2}  \sum_{i, l} g_1([U_4, X_{l}^{(i)}]\, ,\, U_4) g_1([U_4, X_{l}^{(i)}]\, ,\, U_5)  + \frac{1}{2} \sum_{i, l} g_1([U_5, X_{l}^{(i)}]\, ,\, U_4) g_1([U_5, X_{l}^{(i)}]\, ,\, U_5).
\end{eqnarray*}
By using the Lie brackets relations of Lemma \ref{brackets} it is easy to see that the last three terms are equal to zero.
 For the first term we have:
 
 If $i = 1,2,3$ then $[X_{j}^{(i)}, X_{l}^{(i)}]\in\fr{n}_{i}$, hence
$$
\frac14 \sum_{j, l} g_1([X_{j}^{(i)}, X_{l}^{(i)}]\, ,\, U_4) g_1([X_{j}^{(i)}, X_{l}^{(i)}]\, ,\, U_5) = 0.
$$
 If $i = 6, 7, 8$ then $[X_{j}^{(i)}, X_{l}^{(i)}]\subset\fr{n}_1\oplus\cdots\oplus\fr{n}_5$, and we introduce
 the following notations:
 $f(X,Y)=aB(X, [\tilde{H}_4, Y))+bB(X, [\tilde{H}_5, Y]))+cB(X, [\tilde{H}_4, Y))+dB(X, [\tilde{H}_5, Y]))$,
$I_6=\{(i,j): i=\ell+1, \dots ,\ell +m; j=1,\dots ,\ell\}$, 
$I_7=\{(i,j): i=\ell+m+1, \dots ,N; j=1,\dots ,\ell\}$ and
$I_8=\{(i,j): i=\ell+m+1, \dots ,N; j=m+1,\dots ,m+\ell\}$. 

Then by using Proposition \ref{prop1} we obtain that
\begin{eqnarray*}
&&
 \frac{1}{4} \sum_{i=6,7,8}\sum_{j,l=1}^{\dim\fr n_i}g_1([X_{j}^{(i)}, X_{l}^{(i)}]\, ,\, U_4) g_1([X_{j}^{(i)}, X_{l}^{(i)}]\, ,\, U_5) \end{eqnarray*}
\begin{eqnarray*} &&
= \frac{1}{4}\sum_{i=6,7,8}\sum_{j,l=1}^{\dim\fr n_i}\langle\langle [X_{j}^{(i)}, X_{l}^{(i)}]\, ,\, U_4\rangle\rangle \langle\langle [X_{j}^{(i)}, X_{l}^{(i)}]\, ,\, U_5\rangle\rangle  
\end{eqnarray*}
\begin{eqnarray*} 
 & &= \frac{\sqrt{v_4 v_5}}{4} \sum_{i=6,7,8}\sum_{j,l=1}^{\dim\fr n_i}\lbrace aB([X_{j}^{(i)}, X_{l}^{(i)}]\, ,\, \tilde{H}_4) 
+ bB([X_{j}^{(i)}, X_{l}^{(i)}]\, ,\, \tilde{H}_5)\rbrace \\  
& & \quad\quad \quad\quad\quad\quad \quad\quad \quad \times \lbrace cB([X_{j}^{(i)}, X_{l}^{(i)}]\, ,\, \tilde{H}_4) + dB([X_{j}^{(i)}, X_{l}^{(i)}]\, ,\, \tilde{H}_5)\rbrace \\
& &= \frac{\sqrt{v_4 v_5}}{4} \sum_{i=6,7,8}\sum_{j,l=1}^{\dim\fr n_i}\lbrace aB(X_{j}^{(i)}, [\tilde{H}_4, X_{l}^{(i)}]) 
+ bB(X_{j}^{(i)}, [\tilde{H}_5, X_{l}^{(i)}])\rbrace \\  
& &  \quad\quad \quad\quad\quad\quad \quad\quad \quad \times \lbrace cB(X_{j}^{(i)}, [\tilde{H}_4, X_{l}^{(i)}]) 
+ dB(X_{j}^{(i)}, [\tilde{H}_5, X_{l}^{(i)}])\rbrace \\ 
& &=\frac{\sqrt{v_4 v_5}}{4}\times \\ && \Bigg\lbrace
\sum_{\substack{(i,j)\in I_6\\(k,l)\in I_6}}f(\tilde{A}_{ij}, \tilde{A}_{kl})
+\sum_{\substack{(i,j)\in I_6\\(k,l)\in I_6}}f(\tilde{A}_{ij}, \tilde{B}_{kl})
+\sum_{\substack{(i,j)\in I_6\\(k,l)\in I_6}}f(\tilde{B}_{ij}, \tilde{A}_{kl})
+\sum_{\substack{(i,j)\in I_6\\(k,l)\in I_6}}f(\tilde{B}_{ij}, \tilde{B}_{kl})\\
%%%%%%%%%%%%5
& &+
\sum_{\substack{(i,j)\in I_7\\(k,l)\in I_7}}f(\tilde{A}_{ij}, \tilde{A}_{kl})
+\sum_{\substack{(i,j)\in I_7\\(k,l)\in I_7}}f(\tilde{A}_{ij}, \tilde{B}_{kl})
+\sum_{\substack{(i,j)\in I_7\\(k,l)\in I_7}}f(\tilde{B}_{ij}, \tilde{A}_{kl})
+\sum_{\substack{(i,j)\in I_7\\(k,l)\in I_7}}f(\tilde{B}_{ij}, \tilde{B}_{kl})\\
%\end{eqnarray*}
%\begin{eqnarray*} 
& & +
\sum_{\substack{(i,j)\in I_8\\(k,l)\in I_8}}f(\tilde{A}_{ij}, \tilde{A}_{kl})
+\sum_{\substack{(i,j)\in I_8\\(k,l)\in I_8}}f(\tilde{A}_{ij}, \tilde{B}_{kl})
+\sum_{\substack{(i,j)\in I_8\\(k,l)\in I_8}}f(\tilde{B}_{ij}, \tilde{A}_{kl})
+\sum_{\substack{(i,j)\in I_8\\(k,l)\in I_8}}f(\tilde{B}_{ij}, \tilde{B}_{kl})
\Bigg\rbrace \\  
& &=\frac{\sqrt{v_4 v_5}}{4} \bigg\lbrace \frac{bd}{{x_{(6)}}^{2}}\frac{(\ell + m)}{(\ell + m + n)} %\\
%& &\ \ \ \ \ \
 + \frac{1}{{x_{(7)}}^{2}(\ell + m)} \bigg(\ell a c + \frac{\sqrt{\ell m n}}{\sqrt{(\ell + m + n)}}(ad + cb) \\ &&
 + \frac{bdmn}{(\ell + m + n)}\bigg)   + \frac{1}{{x_{(8)}}^{2}(\ell + m)} \bigg( mac - \frac{\sqrt{\ell m n}}{\sqrt{(\ell + m + n)}}(ad + cb) + \frac{bdn\ell}{\sqrt{(\ell + m + n)}}\bigg)\bigg\rbrace, 
\end{eqnarray*}
where in the third equation we used  the bi-invariance of the Killing form and in the fifth equation we used Lemma \ref{lemma1}.  
Then equation (\ref{r0}) follows.
Similar calculations apply for the $\SU(\ell+m+n)$-invariant metrics $(\ref{metrics2})$ on 
the Stiefel manifolds $\SU(\ell+m+n)/\SU(n)$, where the terms for $i=3$ in all  sums above are omitted.
\end{proof}

\subsection{The Ricci tensor for the diagonal parts of the scalar products (\ref{metrics1}), (\ref{metrics2})}
We need the following variant of Lemma \ref{ric2}.  
Since the Ricci tensor for the metrics (\ref{metrics1}) and (\ref{metrics2}) is $\Ad(H)$-invariant,
by using Schur's lemma and the $Q$-structure constants, we can describe the Ricci components of the diagonal parts
of these metrics.

\begin{lemma}\label{ASS} %Lemma 4.5
The components of $r_1, r_2, r_3, r_6, r_7, r_8$ of the Ricci tensor $r$ for the metrics corresponding to the scalar products of the form
$(\ref{metrics1})$ are given as follows:
\begin{equation}
{r}_k = \frac{1}{2y_k}+\frac{1}{4d_k}\sum_{j,i}
\frac{y_k}{y_j y_i} 
\bigg\lbrace
\begin{matrix}
k\\
ji
\end{matrix}
\bigg\rbrace
-\frac{1}{2d_k}\sum_{j,i}\frac{y_j}{y_k y_i} 
\bigg\lbrace
\begin{matrix}
j\\
ki
\end{matrix}
\bigg\rbrace
 \quad (k= 1, 2, 3, 6, 7, 8),    \label{eq501}
\end{equation}
where the sum is taken over $i, j =1,\dots, 8$ and the variables $y_i$  denote corresponding  variables
$u_i, v_i, x_{(i)}$ of the metric $(\ref{metrics1})$.
\end{lemma}
By using relations (\ref{eqn}) we obtain the following:
\begin{prop}\label{ricci2}  %prop 4.6
The components of the Ricci tensor for the diagonal part of  the left-invariant metrics corresponding to the scalar products $(\ref{metrics1})$   are given as follows:
\begin{eqnarray*}
&& r_1 =\frac{1}{2u_{1}}   -\frac{1}{2 d_1u_{1}} 
\left(
{1\brack {11}}
 +
{6\brack 16}+{7\brack 17}
\right) 
 + \frac{1}{4 d_1}
\left(
\frac{1}{u_{1}}
{1\brack 11}
 +
\frac{u_{1}}{{x_{(6)}}^{2}}
{1\brack 66}+
\frac{u_{1}}{{x_{(7)}}^{2}}
{1\brack 77}
\right), 
\\ &&
r_2 =  \frac{1}{2u_{2}} -\frac{1}{2 d_2 u_{2}} 
\left(
{2\brack {22}}
 +
{6\brack 26}+{8\brack 28}
\right) 
+ \frac{1}{4 d_2}
\left(
\frac{1}{u_{2}}
{2\brack 22}
 +
\frac{u_{2}}{{x_{(6)}}^{2}}
{2\brack 66}+
\frac{u_{2}}{{x_{(8)}}^{2}}
{2\brack 88}
\right), 
\\ &&
r_3 =  \frac{1}{2u_{3}}-\frac{1}{2d_3  u_{3}} 
\left(
{3\brack {33}}
 +
{7\brack 37}+{8\brack 38}
\right) 
+ \frac{1}{4 d_3}
\left(
\frac{1}{u_{3}}
{3\brack 33}
 +
\frac{u_{3}}{{x_{(7)}}^{2}}
{3\brack 77}+
\frac{u_{3}}{{x_{(8)}}^{2}}
{3\brack 88}
\right), 
\\ &&
r_6 = \frac1{2x_{(6)}}-\frac1{2d_6 {x_{(6)}}^2}
\left(u_1{1\brack {66}}+u_2{2\brack {66}}+
v_4\bigg\lbrace
\begin{matrix}
4\\
66
\end{matrix}
\bigg\rbrace 
+v_5\bigg\lbrace
\begin{matrix}
5\\
66
\end{matrix}
\bigg\rbrace\right)
\\
&&
+\frac1{2d_6}{6\brack {78}}
\left(
\frac{x_{(6)}}{x_{(7)}x_{(8)}}-\frac{x_{(7)}}{x_{(6)}x_{(8)}}
-\frac{x_{(8)}}{x_{(6)}x_{(7)}}
\right),
\\ &&
r_7 = \frac1{2x_{(7)}}-\frac1{2d_7 {x_{(7)}}^2}
\left(u_1{1\brack {77}}+u_3{3\brack {77}}+
v_4\bigg\lbrace
\begin{matrix}
4\\
77
\end{matrix}
\bigg\rbrace 
+v_5\bigg\lbrace
\begin{matrix}
5\\
77
\end{matrix}
\bigg\rbrace\right)
\\
&&
+\frac1{2d_7}{6\brack {78}}
\left(
\frac{x_{(7)}}{x_{(6)}x_{(8)}}-\frac{x_{(6)}}{x_{(7)}x_{(8)}}
-\frac{x_{(8)}}{x_{(6)}x_{(7)}}
\right),
\\ &&
r_8 = \frac1{2x_{(8)}}-\frac1{2d_8 {x_{(8)}}^2}
\left(u_2{2\brack {88}}+u_3{3\brack {88}}+
v_4\bigg\lbrace
\begin{matrix}
4\\
88
\end{matrix}
\bigg\rbrace 
+v_5\bigg\lbrace
\begin{matrix}
5\\
88
\end{matrix}
\bigg\rbrace\right)
\\
&&
+\frac1{2d_8}{6\brack {78}}
\left(
\frac{x_{(8)}}{x_{(6)}x_{(7)}}-\frac{x_{(6)}}{x_{(7)}x_{(8)}}
-\frac{x_{(7)}}{x_{(6)}x_{(8)}}
\right).
\end{eqnarray*}
For the $\SU(\ell+m+n)$-invariant metrics corresponding to the scalar products $(\ref{metrics2})$, there is no $r_3$ component and the components $r_7, r_8$ simplify by using that ${3\brack {**}}=0$.
\end{prop}
\begin{proof} 
It follows from Lemma \ref{ASS} and relations (\ref{eqn}).
\end{proof}

By substituting the values of the numbers $\displaystyle{{i\brack jk}}$, 
${\small \bigg\lbrace
\begin{matrix}
i\\
jk
\end{matrix}
\bigg\rbrace}$ from  Lemmas \ref{ijk} and \ref{thenumbers} respectively, to the Ricci components
in Propositions \ref{ricci1} and \ref{ricci2} we finally obtain:

\begin{prop}\label{ricci3}  The components of the Ricci tensor for the diagonal part of the left invariant metric {\em (\ref{metrics1}) }
on $\SU(\ell +m+n)$ are given as follows:
\begin{eqnarray*}
&& r_1 = \frac{\ell}{4N}\frac{1}{u_1}+\frac{u_1}{4N} \left(\frac{m}{{x_{(6)}}^2}+
          \frac{n}{{x_{(7)}}^2}\right), \quad 
          r_2 = \frac{m}{4N}\frac{1}{u_2}+\frac{u_2}{4N}\left(\frac{\ell}{{x_{(6)}}^2}+
          \frac{n}{{x_{(8)}}^2}\right),\\
&&  r_3 = \frac{n}{4N}\frac{1}{u_3}+\frac{u_3}{4N}\left(\frac{\ell}{{x_{(7)}}^2}+
          \frac{m}{{x_{(8)}}^2}\right),\\
 && r_6 = \frac1{2x_{(6)}} + \frac{n}{4N}
         \left(\frac{x_{(6)}}{x_{(7)}x_{(8)}} - \frac{x_{(7)}}{x_{(6)}x_{(8)}}
           - \frac{x_{(8)}}{x_{(6)}x_{(7)}}\right) \\
 &&  - \frac{1}{4\ell mN}\frac1{{x_{(6)}}^2}
            \left((\ell ^2-1)mu_1+(m^2-1)\ell u_2+(\ell +m)b^2v_4+(\ell+m)d^2v_5\right),  
             \end{eqnarray*}
 \begin{eqnarray*}
&&  r_7 = \frac1{2x_{(7)}} + \frac{m}{4N}\left(\frac{x_{(7)}}{x_{(6)}x_{(8)}} - \frac{x_{(6)}}{x_{(7)}x_{(8)}}
 - \frac{x_{(8)}}{x_{(6)}x_{(7)}}\right) \\ 
&&  - \frac{1}{4 \ell n N}\frac1{{x_{(7)}}^2}
            \bigg((\ell ^2-1) n u_1+(n^2-1)\ell u_3   +\bigg(\frac{a^{2}\ell N}{\ell + m} + \frac{b^{2}mn}{\ell + m} + \frac{2ab\sqrt{\ell m n}\sqrt{N}}{\ell + m}\bigg)v_4
  \\ & &+\bigg(\frac{c^{2}\ell N}{\ell + m} + \frac{d^{2}mn}{\ell + m} + \frac{2cd\sqrt{\ell m n}\sqrt{N}}{\ell + m}\bigg)v_5 \bigg), \\ 
&&  r_8 = \frac1{2x_{(8)}} + \frac{\ell}{4N}
         \bigg(\frac{x_{(8)}}{x_{(6)}x_{(7)}} - \frac{x_{(7)}}{x_{(6)}x_{(8)}} - \frac{x_{(6)}}{x_{(7)}x_{(8)}}\bigg) \\
&&
  - \frac{1}{4 m n N}\frac1{{x_{(8)}}^2}
            \bigg((m^2-1) n u_2+(n^2-1) m u_3   +\bigg(\frac{a^{2} m N}{\ell + m} + \frac{b^{2} \ell n}{\ell + m} - \frac{2ab\sqrt{\ell m n}\sqrt{N}}{\ell + m}\bigg)v_4 \\
   &&          +\bigg(\frac{c^{2} m N}{\ell + m} + \frac{d^{2}\ell n}{\ell + m} - \frac{2cd\sqrt{\ell m n}\sqrt{N}}{\ell + m}\bigg)v_5 \bigg). 
\end{eqnarray*}
\end{prop}

\begin{prop}\label{ricci31}  %prop 4.8
The components of the Ricci tensor for the diagonal part of the $\SU(\ell+m+n)$-invariant metric {\em (\ref{metrics2})} on $\SU(\ell +m+n)/\SU(n)$ are given as follows:
\begin{eqnarray*} &&
r_1 = \frac{\ell}{4N}\frac{1}{u_1}+\frac{u_1}{4N} \left(\frac{m}{{x_{(6)}}^2}+
          \frac{n}{{x_{(7)}}^2}\right), \quad
          r_2 = \frac{m}{4N}\frac{1}{u_2}+\frac{u_2}{4N}\left(\frac{\ell}{{x_{(6)}}^2}+
          \frac{n}{{x_{(8)}}^2}\right), \\&&
r_6 = \frac1{2x_{(6)}} + \frac{n}{4N}
         \left(\frac{x_{(6)}}{x_{(7)}x_{(8)}} - \frac{x_{(7)}}{x_{(6)}x_{(8)}}
           - \frac{x_{(8)}}{x_{(6)}x_{(7)}}\right) \\
    &&     - \frac{1}{4\ell mN}\frac{1}{{x_{(6)}}^2}
            \left((\ell ^2-1)mu_1+(m^2-1)\ell u_2+(\ell +m)b^2v_4+(\ell+m)d^2v_5\right),   \\
            &&
r_7 = \frac{1}{2x_{(7)}} + \frac{m}{4N}
         \bigg(\frac{x_{(7)}}{x_{(6)}x_{(8)}} - \frac{x_{(6)}}{x_{(7)}x_{(8)}}
           - \frac{x_{(8)}}{x_{(6)}x_{(7)}}\bigg)      - \frac{1}{4 \ell n N}\frac{1}{{x_{(7)}}^2}
            \bigg((\ell ^2-1) n u_1 \\
    & &   + \bigg(\frac{a^{2}\ell N}{\ell + m} + \frac{b^{2}mn}{\ell + m} + \frac{2ab\sqrt{\ell m n}\sqrt{N}}{\ell + m} \bigg)v_4+\bigg(\frac{c^{2}\ell N}{\ell + m} + \frac{d^{2}mn}{\ell + m} + \frac{2cd\sqrt{\ell m n}\sqrt{N}}{\ell + m}\bigg)v_5 \bigg), \\ 
&& r_8 = \frac1{2x_{(8)}} + \frac{\ell}{4N}
         \left(\frac{x_{(8)}}{x_{(6)}x_{(7)}} - \frac{x_{(7)}}{x_{(6)}x_{(8)}}
           - \frac{x_{(6)}}{x_{(7)}x_{(8)}}\right) 
     - \frac{1}{4 m n N}\frac1{{x_{(8)}}^2}
              \bigg((m^2-1) n u_2 \\
    & &   + \bigg(\frac{a^{2} m N}{\ell + m} + \frac{b^{2} \ell n}{\ell + m} - \frac{2ab\sqrt{\ell m n}\sqrt{N}}{\ell + m} \bigg)v_4+\ \bigg(\frac{c^{2} m N}{\ell + m} + \frac{d^{2}\ell n}{\ell + m} - \frac{2cd\sqrt{\ell m n}\sqrt{N}}{\ell + m}\bigg)v_5  \bigg). 
\end{eqnarray*}
\end{prop}

 \section{Non naturally reductive Einstein metrics on the compact Lie group $\SU(\ell+m+n)$}

\subsection{Naturally reductive metrics on $\SU(\ell+m+n)$}
A Riemannian homogeneous space $(M=G/H, g)$ with reductive complement $\fr{m}$ of $\fr{h}$ in $\fr{g}$ is called {\it naturally reductive} if
$$
\langle [X, Y]_\fr{m}, Z\rangle +\langle Y, [X, Z]_\fr{m}\rangle=0 \quad\mbox{for all}\ 
X, Y, Z\in\fr{m}.
$$ 
Here $\langle\ , \  \rangle$ denotes the inner product on $\fr{m}$ induced from the Riemannian metric $g$.  Classical examples of naturally reductive homogeneous spaces include
irreducible symmetric spaces, isotropy irreducible homogeneous manifolds, 
and Lie groups with bi-invariant metrics.
In general it is not always easy to decide if a given homogeneous Riemannian manifold is naturally reductive, since one has to consider all possible transitive actions of subgroups
$G$ of the isometry group of $(M, g)$.

In \cite{DZ}  D'Atri and Ziller  investigated naturally reductive metrics among left-invariant metrics on compact Lie groups and gave a complete classification in the case of simple Lie groups.
Let $G$ be a compact, connected semisimple Lie group, $L$ a closed subgroup of $G$ and let
$\fr{g}$ be the Lie algebra of $G$ and $\fr{l}$ the subalgebra corresponding to $L$.

Recall that $B$ is the negative of the Killing form of $\fr{g}$, so $B$ is an
$\Ad (G)$-invariant inner product on $\fr{g}$.
Let $\fr{m}$ be an orthogonal complement of $\fr{l}$ with respect to $B$.  Then we have
$$
\fr{g}=\fr{l}\oplus\fr{m}, \quad \Ad(L)\fr{m}\subset\fr{m}.
$$
Let $\fr{l}=\fr{l}_0\oplus\fr{l}_1\oplus\cdots\oplus\fr{l}_p$ be a decomposition of $\fr{l}$ into ideals, where $\fr{l}_0$ is the center of $\fr{l}$ and $\fr{l}_i$ $(i=1,\dots , p)$ are simple ideals of $\fr{l}$.
Let $A_0|_{\fr{l}_0}$ be an arbitrary metric on $\fr{l}_0$.

\begin{theorem}\label{DZ} {\rm (\cite[Theorem 1, p. 9]{DZ})}  Under the notations above a left-invariant metric on $G$ of the form
\begin{equation}\label{natural}
\langle\  ,\  \rangle =x\cdot B|_{\fr{m}}+A_0|_{\fr{l}_0}+u_1\cdot B|_{\fr{l}_1}+\cdots
+ u_p\cdot B|_{\fr{l}_p}, \quad (x, u_1, \dots , u_p >0)
\end{equation}
is naturally reductive with respect to $G\times L$, where $G\times L$ acts on $G$ by
$(g, l)y=gyl^{-1}$.

Moreover, if a left-invariant metric $\langle\ ,\ \rangle$ on a compact simple Lie group
$G$ is naturally reductive, then there is a closed subgroup $L$ of $G$ and the metric
$\langle\ ,\ \rangle$ is given by the form $(\ref{natural})$.
\end{theorem}

\medskip
For the Lie group  $\SU(\ell+m+n)$ we consider  left-invariant metrics determined by $\Ad(\s(\U(\ell)\times\U(m)\times\U(n)))$-invariant scalar products (\ref{inv_metrics}).

\begin{prop}\label{reductive}
If a left invariant metric of the form   {\em (\ref{inv_metrics})} on $\SU(\ell+m+n)$  is naturally reductive with respect to $\SU(\ell+m+n) \times L$, for some closed subgroup $L$ of $\SU(\ell+m+n)$, then one of the
following holds: 
\begin{itemize}
\item[(1)]   the metric   {\em (\ref{inv_metrics})} is either 
\begin{itemize}
\item[(i)]
$\Ad(\s(\U(\ell+m)\times \U(n)))$-invariant and  $x_{(7)}=x_{(8)}$, or 
\item[(ii)]
 $\Ad(\s(\U(\ell)\times \U(m+n)))$-invariant and  $x_{(6)}=x_{(7)}$,  or
%\item[$(ii)$]  $\s(\U(\ell)\times \U(m+n))$, then $x_{(6)}=x_{(7)}$, or
\item[(iii)]  $\Ad(\s(\U(\ell+n)\times \U(m)))$-invariant  and  $x_{(6)}=x_{(8)}$.
\end{itemize}
\item[(2)] $x_{(6)}=x_{(7)}=x_{(8)}$. 
\end{itemize}
Conversely, if one of the conditions {\em (1)},  {\em (2)} is satisfied,  then the metric 
 of the form {\em (10)}  is  naturally reductive  with respect to $\SU(\ell+m+n)\times L$, for some closed subgroup $L$ of $\SU(\ell+m+n)$.
  \end{prop}
  
 %%%%%%%%
\begin{proof}  
Let $\fr{l}$ be the Lie algebra of $L$.  Then we have that either 
${\frak l} \subset {\frak h}=\frak{s}(\frak{u}(\ell)\oplus\frak{u}(m)\oplus\frak{u}(n))$
 or ${\frak l} \not\subset {\frak h}$. 
First we consider the case of  ${\frak l} \not\subset {\frak h}$. Let ${\frak k}$ be the subalgebra of ${\frak g}$ generated by ${\frak l}$ and ${\frak h}$. 
 Since 
 $ \fr{su}(\ell + m + n) $ splits into $\fr{h}$ and three $\Ad(H)$-irreducible modules $\fr{m}_{12}$, $\fr{m}_{13}$,  $\fr{m}_{23}$ as 
$
\fr{g}=\fr{h}\oplus\fr{m}=\fr{h}_0\oplus\fr{h}_1\oplus\fr{h}_2\oplus\fr{h}_3\oplus\fr{m}_{12}
\oplus\fr{m}_{13}\oplus\fr{m}_{23}$,
where the $\Ad(H)$-irreducible modules $\fr{m}_{12}$, $\fr{m}_{13}$, $\fr{m}_{23}$ are mutually non equivalent, 
%$ \fr{su}(\ell + m + n) = \fr{h}_1\oplus \fr{h}_2\oplus \fr{h}_3\oplus\fr{h}_4\oplus\fr{h}_5\oplus \fr{m}_{12}\oplus  \fr{m}_{13}\oplus  \fr{m}_{23}$ is an irreducible decomposition as $\mbox{Ad}(H)$-modules, 
we see that the Lie algebra $\frak k$  contains  at least one of  ${\frak m}_{12}$,  ${\frak m}_{13}$, ${\frak m}_{23}$. %Without loss of generality, 
 Let us assume that  $\frak k$  contains ${\frak m}_{12}$. 
 Note that 
$\left[{\frak m}_{12}, {\frak m}_{12}\right] \subset {\frak h}_1 \oplus \fr{h}_2\oplus\fr{h}_0$ and thus $\frak k$ contains $\fr{su}(\ell + m)$. 
Thus we see that $\frak k$ contains the Lie subalgebra $\fr{s}(\fr{u}(\ell + m)\oplus\fr{u}(n))$. 
If $\frak k =\fr{s}(\fr{u}(\ell + m)\oplus\fr{u}(n))$, then we obtain an irreducible decomposition $ \fr{su}(\ell + m + n) = \frak k \oplus \fr{n}$, where $\fr{n} = {\frak m}_{13}\oplus{\frak m}_{23}$. Hence, a  naturally reductive left invariant metric of the form  (\ref{inv_metrics})  is $\Ad(\s(\U(\ell+m)\times \U(n)))$-invariant and   
$x_{(7)}=x_{(8)}$, so we obtain case (i).  Cases (ii) and (iii) are obtained by a  similar way. 
Furthermore, if $\frak k  \neq \fr{s}(\fr{u}(\ell + m)\oplus\fr{u}(n))$, then $\frak k$ contains ${\frak m}_{13}$ or ${\frak m}_{23}$. In this case  we can see that  $\frak k =  \fr{su}(\ell+m+n)$. Thus the metric is bi-invariant. 

Now we consider the case ${\frak l} \subset {\frak h}$.  Since the  orthogonal complement
 ${\frak l}^{\bot}$ of ${\frak l}$ with respect to $B$ contains the  orthogonal complement 
${\frak h}^{\bot}$ of ${\frak h}$, we see that ${\frak l}^{\bot} \supset {\frak m}_{12} \oplus  {\frak m}_{13}\oplus  {\frak m}_{23}$.   
Since the left invariant metric of the form (\ref{inv_metrics}) is naturally reductive  with respect to $G\times L$,  
 it follows that  $ x_{(6)} = x_{(7)} = x_{(8)} $ by  Theorem \ref{DZ}.   
The converse is a direct consequence of Theorem \ref{DZ}.
\end{proof}
  
 %%%%%%%%

\subsection{Non naturally reductive Einstein metrics on $\SU(\ell+m+n)$} %page 20
To find non naturally reductive Einstein metrics on $\SU(\ell+m+n)$ we need to solve the system
(cf. Propositions \ref{ricci1}, \ref{ricci3})
\begin{eqnarray}\label{sist1}
&& r_{0} = 0, \, r_{1} - r_2 = 0,\, r_2 - r_3 = 0, \, r_3 - r_4 = 0,\nonumber\\
&& r_4 - r_{5} = 0, \, r_{5} - r_{6} = 0, \, r_{6} - r_{7} = 0, \ r_{7} - r_{8} = 0.
\end{eqnarray}

We claim that it is possible to choose a basis $\{V_4', V_5'\}$ of $\fr{h}_0$ so that the matrix of the scalar
product (\ref{prod}) with respect to $\{H_4, H_5\}$ is given by
\begin{equation}\label{form}
{}^{t}_{}\!\begin{pmatrix}
1 & 0\\
\gamma & 1
\end{pmatrix}
\begin{pmatrix}
v_4' & 0\\
0 &  v_5'
\end{pmatrix}
\begin{pmatrix}
1 & 0\\
\gamma & 1
\end{pmatrix},
\end{equation}
for some real number $\gamma$  and $v_4', v_5'>0$.  Hence, without loss of generality we may choose
$a=d=1$ and $b=0$.

Indeed, by using the QR-decomposition  we obtain that
$$
\begin{pmatrix}
a & b\\
c & d
\end{pmatrix}=
\begin{pmatrix}
\cos t & -\sin t\\
\sin t & \cos t
\end{pmatrix}
\begin{pmatrix}
x & 0\\
0 & y
\end{pmatrix}
\begin{pmatrix}
1 & 0\\
\gamma & 1
\end{pmatrix},
$$
for some real numbers $x,y, \gamma$, with $x,y$ non zero.
So the matrix $A$ (cf. (\ref{positive}))  takes the form
$$
{}^{t}_{}\!\begin{pmatrix}
1 & 0\\
\gamma & 1
\end{pmatrix}
\begin{pmatrix}
x & 0\\
0 &  y
\end{pmatrix}
\begin{pmatrix}
\cos t & \sin t\\
-\sin t & \cos t
\end{pmatrix}
\begin{pmatrix}
v_4 & 0\\
0 &  v_5
\end{pmatrix}
\begin{pmatrix}
\cos t & -\sin t\\
\sin t & \cos t
\end{pmatrix}
\begin{pmatrix}
x & 0\\
0 &  y
\end{pmatrix}
\begin{pmatrix}
1 & 0\\
\gamma & 1
\end{pmatrix}.
$$
By changing the orthonormal basis $\{V_4, V_5\}$ into some orthonormal basis
$\{V_4', V_5'\}$, the matrix $A$ can be expressed as
$$
{}^{t}_{}\!\begin{pmatrix}
1 & 0\\
\gamma & 1
\end{pmatrix}
\begin{pmatrix}
x & 0\\
0 &  y
\end{pmatrix}
\begin{pmatrix}
v_4 & 0\\
0 &  v_5
\end{pmatrix}
\begin{pmatrix}
x & 0\\
0 &  y
\end{pmatrix}
\begin{pmatrix}
1 & 0\\
\gamma & 1
\end{pmatrix},
$$ which gives expression (\ref{form}).

We also assume that
 $\ell = 1, m = 2$. In this case it is $\fr h_1=0$, so for $\SU(3+n)$  system (\ref{sist1}) reduces to 
\begin{eqnarray}\label{sist2}
&& r_0=0, \, r_2 - r_3 = 0, \, r_3 - r_4 = 0, \nonumber\\
&& r_4 - r_{5} = 0, \, r_{5} - r_{6} = 0, \, r_{6} - r_{7} = 0, \ r_{7} - r_{8} = 0.
\end{eqnarray}
 Notice that in the above system there is no $u_1$ variable. 
  By setting $x_{(7)} = 1$ in the equation $r_{0} = 0$ we obtain that
$$
c = \frac{\sqrt{2n(3+n)}}{(3 + n)}\frac{(1-{x_{(8)}}^2)}{(2 + {x_{(8)}}^2)}.
$$ 
 We substitute $c$ into the system (\ref{sist2}) and  interestingly, we observe that the equations 
$$
r_4 - r_{5} = 0,\, r_{5} - r_{6} = 0, \, r_{6} - r_{7} = 0, \ r_{7} - r_{8} = 0,
$$
that is, %page 21
 {\small \begin{equation*}
\begin{array}{l} 
%-2 {u_2} {u_3}^2 {x_1}^2-{u_2}{u_3}^2 {x_2}^2 {x_1}^2-n \,{u_2} {x_2}^2 {x_1}^2 +2 {u_3} {x_2}^2{x_1}^2+n \,{u_2}^2 {u_3} {x_1}^2 +{u_2}^2 {u_3} {x_2}^2=0,\\
  % 3{x_2}^2 {u_3}^2+6 {u_3}^2-n {u_4}  {x_2}^2 {u_3}-3 {u_4} {x_2}^2 {u_3}-2 n {u_4} {u_3}-6 {u_4} {u_3}+3 n {x_2}^2 =0,\\
  n {v_4} {x_{(6)}}^2
   {x_{(8)}}^4+3 {v_4} {x_{(6)}}^2 {x_{(8)}}^4-9
   {v_5} {x_{(8)}}^4+4 n {v_4} {x_{(6)}}^2
   {x_{(8)}}^2+12 {u_4} {x_{(6)}}^2 {x_{(8)}}^2-9 n
   {v_5} {x_{(6)}}^2 {x_{(8)}}^2 \\
   -18 {v_5}
   {x_{(8)}}^2+4 n {v_4} {x_{(6)}}^2+12 {v_4}
   {x_{(6)}}^2 =0, \\
 2 n {x_{(6)}}
   {x_{(8)}}^4+3 {u_2} {x_{(8)}}^3+9 {v_5}
   {x_{(8)}}^3-4 n {x_{(6)}} {x_{(8)}}^3-12 {x_{(6)}}
   {x_{(8)}}^3-2 n {x_{(6)}}^3 {x_{(8)}}^2+6 n {x_{(6)}}
   {x_{(8)}}^2 \\+6 n {v_5} {x_{(6)}}^2 {x_{(8)}}+6
   {u_2} {x_{(8)}}+18 {v_5} {x_{(8)}}-8 n
   {x_{(6)}} {x_{(8)}}-24 {x_{(6)}} {x_{(8)}}-4 n
   {x_{(6)}}^3+4 n {x_{(6)}}=0,\\
  -6 n^2 {x_{(6)}}
   {x_{(8)}}^7-6 n {x_{(6)}} {x_{(8)}}^7-9 n {u_2}
   {x_{(8)}}^6-9 n {v_5} {x_{(8)}}^6+12 n^2 {x_{(6)}}
   {x_{(8)}}^6+36 n {x_{(6)}} {x_{(8)}}^6 \\
   +6 n^2{x_{(6)}}^3 {x_{(8)}}^5    +6 n {x_{(6)}}^3 {x_{(8)}}^5-12
   n^2 {x_{(6)}}^2 {x_{(8)}}^5-36 n {x_{(6)}}^2
   {x_{(8)}}^5-30 n^2 {x_{(6)}} {x_{(8)}}^5-18 n {x_{(6)}} {x_{(8)}}^5
   \\
   +9 n {u_2} {x_{(6)}}^2 {x_{(8)}}^4
   +6 n^2 {u_3} {x_{(6)}}^2 {x_{(8)}}^4-6
   {u_3} {x_{(6)}}^2 {x_{(8)}}^4+2 n {v_4}
   {x_{(6)}}^2 {x_{(8)}}^4+6 {u_4} {x_{(6)}}^2
   {x_{(8)}}^4\\
   +9 n {u_5} {x_{(6)}}^2 {x_{(8)}}^4-36 n
   {u_2} {x_{(8)}}^4
   -36 n {v_5} {x_{(8)}}^4+48
   n^2 {x_{(6)}} {x_{(8)}}^4+144 n {x_{(6)}}
   {x_{(8)}}^4+24 n^2 {x_{(6)}}^3 {x_{(8)}}^3\\
   +24 n {x_{(6)}}^3 {x_{(8)}}^3-48 n^2 {x_{(6)}}^2
   {x_{(8)}}^3 
   -144 n {x_{(6)}}^2 {x_{(8)}}^3-48 n^2
   {x_{(6)}} {x_{(8)}}^3+36 n {u_2} {x_{(6)}}^2
   {x_{(8)}}^2 \\
   +24 n^2 {u_3} {x_{(6)}}^2
   {x_{(8)}}^2-24 {u_3} {x_{(6)}}^2 {x_{(8)}}^2 
    +8 n{v_4} {x_{(6)}}^2 {x_{(8)}}^2+24 {v_4}
   {x_{(6)}}^2 {x_{(8)}}^2-36 n {u_2} {x_{(8)}}^2\\ 
   -36 n {v_5} {x_{(8)}}^2 +48 n^2 {x_{(6)}}
   {x_{(8)}}^2+144 n {x_{(6)}} {x_{(8)}}^2 
   +24 n^2
   {x_{(6)}}^3 {x_{(8)}} +24 n {x_{(6)}}^3 {x_{(8)}}-48
   n^2 {x_{(6)}}^2 {x_{(8)}} \\-144 n {x_{(6)}}^2
   {x_{(8)}}-24 n^2 {x_{(6)}} {x_{(8)}}+24 n {x_{(6)}}
   {x_{(8)}}+36 n {u_2} {x_{(6)}}^2+24 n^2 {u_3}
   {x_{(6)}}^2-24 {u_3} {x_{(6)}}^2\\+8 n {v_4}
   {x_{(6)}}^2+24 {v_4} {x_{(6)}}^2=0,
   \\ 
   18 n {x_{(8)}}^7-12 n^2 {x_{(6)}} {x_{(8)}}^6-36 n
   {x_{(6)}} {x_{(8)}}^6+6 n^2 {u_3} {x_{(6)}}
   {x_{(8)}}^6-6 {u_3} {x_{(6)}} {x_{(8)}}^6+2 n
   {u_4} {x_{(6)}} {x_{(8)}}^6 \\
   +6 {v_4} {x_{(6)}}
   {x_{(8)}}^6 +6 n {x_{(6)}}^2 {x_{(8)}}^5+54 n
   {x_{(8)}}^5+12 n^2 {x_{(6)}} {x_{(8)}}^5+36 n
   {x_{(6)}} {x_{(8)}}^5-48 n^2 {x_{(6)}} {x_{(8)}}^4\\
   -144 n {x_{(6)}} {x_{(8)}}^4 -9 n {u_2} {x_{(6)}}
   {x_{(8)}}^4+18 n^2 {u_3} {x_{(6)}} {x_{(8)}}^4-18
   {u_3} {x_{(6)}} {x_{(8)}}^4+6 n {v_4}{x_{(6)}} {x_{(8)}}^4\\
   +18 {v_4} {x_{(6)}}
   {x_{(8)}}^4-9 n {v_5} {x_{(6)}} {x_{(8)}}^4 +24 n
   {x_{(6)}}^2 {x_{(8)}}^3+48 n^2 {x_{(6)}}
   {x_{(8)}}^3+144 n {x_{(6)}} {x_{(8)}}^3-48 n^2
   {x_{(6)}} {x_{(8)}}^2 \\
   -144 n {x_{(6)}} {x_{(8)}}^2-36 n
   {u_2} {x_{(6)}} {x_{(8)}}^2 +36 n {v_5}
   {x_{(6)}} {x_{(8)}}^2+24 n {x_{(6)}}^2 {x_{(8)}}-72 n
   {x_{(8)}}+48 n^2 {x_{(6)}} {x_{(8)}}\\
   +144 n {x_{(6)}}
   {x_{(8)}}-36 n {u_2} {x_{(6)}}-24 n^2 {u_3}
   {x_{(6)}}+24 {u_3} {x_{(6)}}-8 n {v_4}
   {x_{(6)}}-24 {v_4} {x_{(6)}}=0, 
 \end{array}
\end{equation*}  }
\noindent
are {\it linear} with respect to $u_2, u_3, v_4$ and $v_5$.
 By solving the above equations with respect to $u_2, u_3, v_4$ and $v_5$ we obtain 
{\small \begin{equation*}
 \begin{array}{l} 
 u_2 = -1/({3 {x_{(8)}} \left(-n {x_{(6)}}^4+n {x_{(8)}}^2 {x_{(6)}}^2-{x_{(8)}}^2 {x_{(6)}}^2-n {x_{(6)}}^2-4 {x_{(6)}}^2+{x_{(8)}}^4+{x_{(8)}}^2-2\right))}\times \\ 
 \big( {x_{(6)}}  (2 n {x_{(8)}}^6+3 {x_{(8)}}^6-4 n {x_{(8)}}^5-12{x_{(8)}}^5+2 n^2 {x_{(6)}}^2 {x_{(8)}}^4+n {x_{(6)}}^2 {x_{(8)}}^4-3 {x_{(6)}}^2 {x_{(8)}}^4 +4 n {x_{(8)}}^4\\+6 n {x_{(6)}} {x_{(8)}}^4 
+18 {x_{(6)}} {x_{(8)}}^4 +9 {x_{(8)}}^4-4 n^2 {x_{(6)}}^2 {x_{(8)}}^3-14 n {x_{(6)}}^2 {x_{(8)}}^3-6 {x_{(6)}}^2 {x_{(8)}}^3-4 n {x_{(8)}}^3 \\
-6 n {x_{(6)}} {x_{(8)}}^3-18 {x_{(6)}} {x_{(8)}}^3 -12 {x_{(8)}}^3 -2 n^2 {x_{(6)}}^4 {x_{(8)}}^2-3 n {x_{(6)}}^4 {x_{(8)}}^2+4 n^2 {x_{(6)}}^3 {x_{(8)}}^2 \\
+12 n {x_{(6)}}^3 {x_{(8)}}^2-n {x_{(6)}}^2 {x_{(8)}}^2-2 n {x_{(8)}}^2+12 n {x_{(6)}} {x_{(8)}}^2  
+36 {x_{(6)}} {x_{(8)}}^2-4 n^2 {x_{(6)}}^3 {x_{(8)}}\\
-12 n {x_{(6)}}^3 {x_{(8)}}+4 n^2 {x_{(6)}}^2 {x_{(8)}}+16 n {x_{(6)}}^2 {x_{(8)}}+12 {x_{(6)}}^2 {x_{(8)}}+8 n {x_{(8)}}
-12 n {x_{(6)}} {x_{(8)}}-36 {x_{(6)}} {x_{(8)}}\\
+24 {x_{(8)}}+2 n^2 {x_{(6)}}^4+6 n {x_{(6)}}^4-2 n^2 {x_{(6)}}^2-2 n {x_{(6)}}^2+12 {x_{(6)}}^2-4 n-12)
\big),
\\ \\
u_3 = -1/\big((n-1) (n+1) {x_{(6)}} {x_{(8)}} (n {x_{(6)}}^4-n {x_{(8)}}^2 {x_{(6)}}^2+{x_{(8)}}^2 {x_{(6)}}^2+n {x_{(6)}}^2+4 {x_{(6)}}^2-{x_{(8)}}^4\\
-{x_{(8)}}^2+2) \big)\times 
\big(n^3 {x_{(6)}}^6+3 n^2 {x_{(6)}}^6-2 n^3 {x_{(8)}} {x_{(6)}}^5-6 n^2 {x_{(8)}} {x_{(6)}}^5-n^3 {x_{(6)}}^4-n^2 {x_{(6)}}^4
\\ -n^3 {x_{(8)}}^2 {x_{(6)}}^4+n^2 {x_{(8)}}^2 {x_{(6)}}^4 
+4 n {x_{(8)}}^2 {x_{(6)}}^4+6 n {x_{(6)}}^4+2 n^3 {x_{(8)}} {x_{(6)}}^4+8 n^2 {x_{(8)}} {x_{(6)}}^4 \\
+6 n {x_{(8)}} {x_{(6)}}^4+2 n^3 {x_{(8)}}^3 {x_{(6)}}^3+4 n^2 {x_{(8)}}^3 {x_{(6)}}^3 -6 n {x_{(8)}}^3 {x_{(6)}}^3-2 n^3 {x_{(8)}}^2 {x_{(6)}}^3-8 n^2 {x_{(8)}}^2 {x_{(6)}}^3\\
-6 n {x_{(8)}}^2 {x_{(6)}}^3-6 n^2 {x_{(8)}} {x_{(6)}}^3-18 n {x_{(8)}} {x_{(6)}}^3-4 n^2 {x_{(8)}}^4 {x_{(6)}}^2
-n {x_{(8)}}^4 {x_{(6)}}^2+{x_{(8)}}^4 {x_{(6)}}^2\\
+2 n^2 {x_{(8)}}^3 {x_{(6)}}^2+8 n {x_{(8)}}^3 {x_{(6)}}^2+6 {x_{(8)}}^3 {x_{(6)}}^2-2 n^2 {x_{(6)}}^2+2 n {x_{(8)}}^2 {x_{(6)}}^2-2 {x_{(8)}}^2 {x_{(6)}}^2
\\
-6 n {x_{(6)}}^2+4 n^2 {x_{(8)}} {x_{(6)}}^2+12 n {x_{(8)}} {x_{(6)}}^2+2 n^2 {x_{(8)}}^5 {x_{(6)}}+6 n {x_{(8)}}^5 {x_{(6)}}-2 n^2 {x_{(8)}}^4 {x_{(6)}}
\end{array}
\end{equation*}
\begin{equation*}
\begin{array}{l}
-8 n {x_{(8)}}^4 {x_{(6)}}-6 {x_{(8)}}^4 {x_{(6)}} 
+4 n^2 {x_{(8)}}^3 {x_{(6)}}+14 n {x_{(8)}}^3 {x_{(6)}}+6 {x_{(8)}}^3 {x_{(6)}}-4 n^2 {x_{(8)}}^2 {x_{(6)}}\\
-12 n {x_{(8)}}^2 {x_{(6)}}
-3 n {x_{(8)}}^6-{x_{(8)}}^6-3 n {x_{(8)}}^4 
-{x_{(8)}}^4+6 n {x_{(8)}}^2+2 {x_{(8)}}^2
\big), %\\ \\
\end{array}
\end{equation*} }
{\small \begin{equation*}
\begin{array}{l} 
v_4 =1/\big({(n+3) {x_{(6)}} ({x_{(8)}}^2+2 ) (-n {x_{(6)}}^4+n {x_{(8)}}^2 {x_{(6)}}^2-{x_{(8)}}^2 {x_{(6)}}^2-n {x_{(6)}}^2-4 {x_{(6)}}^2+{x_{(8)}}^4}\\
+{x_{(8)}}^2-2 ) \big) \times 
3 {x_{(8)}} (n {x_{(6)}}^2+{x_{(8)}}^2+2) (-n {x_{(6)}}^4+n {x_{(8)}}^2 {x_{(6)}}^2-{x_{(8)}}^2 {x_{(6)}}^2 
+n {x_{(6)}}^2\\
-2 n {x_{(8)}} {x_{(6)}}^2-6 {x_{(8)}} {x_{(6)}}^2+2 {x_{(6)}}^2 
+2 n {x_{(8)}}^2 {x_{(6)}}+6 {x_{(8)}}^2 {x_{(6)}}-2 n {x_{(8)}} {x_{(6)}}-6 {x_{(8)}} {x_{(6)}}+{x_{(8)}}^4\\
+{x_{(8)}}^2-2)
,
\\ \\
v_5 = 1/\big({3{x_{(8)}} \left(n {x_{(6)}}^4-n {x_{(8)}}^2 {x_{(6)}}^2+{x_{(8)}}^2 {x_{(6)}}^2+n {x_{(6)}}^2+4 {x_{(6)}}^2-{x_{(8)}}^4-{x_{(8)}}^2+2\right)}\big) \times\\ 
{x_{(6)}} ({x_{(8)}}^2+2) (n {x_{(6)}}^4-n {x_{(8)}}^2 {x_{(6)}}^2+{x_{(8)}}^2 {x_{(6)}}^2-n {x_{(6)}}^2 
+2 n {x_{(8)}} {x_{(6)}}^2+6 {x_{(8)}} {x_{(6)}}^2-2{x_{(6)}}^2\\
-2 n {x_{(8)}}^2 {x_{(6)}}-6 {x_{(8)}}^2 {x_{(6)}}+2 n {x_{(8)}} {x_{(6)}}+6 {x_{(8)}} {x_{(6)}}-{x_{(8)}}^4-{x_{(8)}}^2+2). 
\end{array}
\end{equation*} }
 We substitute the above expressions for $u_2, u_3, v_4$ and $v_5$ into the equations
{\begin{eqnarray*}
&& r_2 - r_3 = 0 \Leftrightarrow\\
&& -2 {u_2} {u_3}^2 {x_{(6)}}^2-{u_2}{u_3}^2 {x_{(8)}}^2 {x_{(6)}}^2-n \,{u_2} {x_{(8)}}^2 {x_{(6)}}^2 +2 {u_3} {x_{(8)}}^2{x_{(6)}}^2\\
&&+n \,{u_2}^2 {u_3} {x_{(6)}}^2
+{u_2}^2 {u_3} {x_{(8)}}^2=0,\\
&& r_3 - r_4 = 0 \Leftrightarrow\\
&&  3{x_{(8)}}^2 {u_3}^2+6 {u_3}^2-n {v_4}  {x_{(8)}}^2 {u_3}-3 {v_4} {x_{(8)}}^2 {u_3}-2 n {v_4} {u_3}-6 {v_4} {u_3}+3 n {x_{(8)}}^2 =0
\end{eqnarray*} }

\noindent
and obtain two equations $F_1(x_{(6)}, x_{(8)}) = 0$ and $F_2(x_{(6)}, x_{(8)}) = 0$ with parameter $n$.

 It is possible to pursue  computations for any value of $n$.   However, we will restrict ourselves to the case
$n=2$, not only due to  space limitations, but also because  this corresponds to the  Lie group $\SU(5)$
(the special unitary group of lowest rank known up to now to admit a non naturally reductive Einstein metric).
In this case 
$c = -\frac{2 \left({x_{(8)}}^2-1\right)}{\sqrt{5} \left({x_{(8)}}^2+2\right)}$ and 
 the   substitution of $u_2, u_3, v_4, v_5$ into the equations
$r_2 - r_3 = 0$ and $r_3 - r_4 = 0$ 
gives the  equations $ -10 ({x_{(6)}}-1) {x_{(6)}} F_1(x_{(6)}, x_{(8)})=0$ and 
$ 10 ({x_{(8)}}-{x_{(6)}}) F_2(x_{(6)}, x_{(8)})=0$, where 
{\small \begin{eqnarray*}
&& F_1(x_{(6)}, x_{(8)}) =  (-98 {x_{(8)}}^{19}+273 {x_{(6)}} {x_{(8)}}^{18}+273{x_{(8)}}^{18} -378 {x_{(6)}}^2 {x_{(8)}}^{17}-1368 {x_{(6)}} {x_{(8)}}^{17}\\
&&-378{x_{(8)}}^{17}+791t{x_{(6)}}^3 {x_{(8)}}^{16} +2201 {x_{(6)}}^2 {x_{(8)}}^{16}+2201{x_{(6)}} {x_{(8)}}^{16}+791 {x_{(8)}}^{16}+42 {x_{(6)}}^4 {x_{(8)}}^{15}\\
&& -4994{x_{(6)}}^3 {x_{(8)}}^{15}+1092 {x_{(6)}}^2 {x_{(8)}}^{15}-4994 {x_{(6)}}{x_{(8)}}^{15}+42 {x_{(8)}}^{15} -903 {x_{(6)}}^5 {x_{(8)}}^{14}\\
&& +3397 {x_{(6)}}^4{x_{(8)}}^{14}-4014 {x_{(6)}}^3 {x_{(8)}}^{14} -4014 {x_{(6)}}^2 {x_{(8)}}^{14}+3397{x_{(6)}} {x_{(8)}}^{14} -903 {x_{(8)}}^{14} \\
&& +1330 {x_{(6)}}^6 {x_{(8)}}^{13}+1044{x_{(6)}}^5 {x_{(8)}}^{13}+17190 {x_{(6)}}^4 {x_{(8)}}^{13} +29176 {x_{(6)}}^3{x_{(8)}}^{13} +17190 {x_{(6)}}^2 {x_{(8)}}^{13}\\
&& +1044 {x_{(6)}} {x_{(8)}}^{13} +1330{x_{(8)}}^{13}-2919 {x_{(6)}}^7 {x_{(8)}}^{12}-4909 {x_{(6)}}^6 {x_{(8)}}^{12} -43035{x_{(6)}}^5 {x_{(8)}}^{12}\\
 && -81465 {x_{(6)}}^4 {x_{(8)}}^{12}
 -81465 {x_{(6)}}^3{x_{(8)}}^{12} -43035 {x_{(6)}}^2 {x_{(8)}}^{12}-4909 {x_{(6)}} {x_{(8)}}^{12}-2919{x_{(8)}}^{12} \\
&& +336 {x_{(6)}}^8 {x_{(8)}}^{11}  +15158 {x_{(6)}}^7 {x_{(8)}}^{11} 
 +41888{x_{(6)}}^6 {x_{(8)}}^{11}+192250 {x_{(6)}}^5 {x_{(8)}}^{11}\\
 && +144528 {x_{(6)}}^4{x_{(8)}}^{11} +192250 {x_{(6)}}^3 {x_{(8)}}^{11}  +41888 {x_{(6)}}^2{x_{(8)}}^{11} +15158 {x_{(6)}} {x_{(8)}}^{11}+336 {x_{(8)}}^{11}\\ &&
 +1946 {x_{(6)}}^9{x_{(8)}}^{10} -9574 {x_{(6)}}^8 {x_{(8)}}^{10}  -42436 {x_{(6)}}^7 {x_{(8)}}^{10}-236716{x_{(6)}}^6 {x_{(8)}}^{10} \\&&
 -305528 {x_{(6)}}^5 {x_{(8)}}^{10}  -305528 {x_{(6)}}^4{x_{(8)}}^{10}
  -236716 {x_{(6)}}^3 {x_{(8)}}^{10}-42436 {x_{(6)}}^2 {x_{(8)}}^{10} \\&&
  -9574{x_{(6)}} {x_{(8)}}^{10} +1946 {x_{(8)}}^{10}-1680 {x_{(6)}}^{10} {x_{(8)}}^9-808{x_{(6)}}^9 {x_{(8)}}^9 +15456 {x_{(6)}}^8 {x_{(8)}}^9 
   \\&&+238584 {x_{(6)}}^7{x_{(8)}}^9+373104 {x_{(6)}}^6 {x_{(8)}}^9+490608 {x_{(6)}}^5 {x_{(8)}}^9 +373104{x_{(6)}}^4 {x_{(8)}}^9 \\
% \end{eqnarray*}
% \begin{eqnarray*}
  && +238584 {x_{(6)}}^3 {x_{(8)}}^9 +15456 {x_{(6)}}^2{x_{(8)}}^9-808 {x_{(6)}} {x_{(8)}}^9 -1680 {x_{(8)}}^9+3108 {x_{(6)}}^{11}{x_{(8)}}^8 
  \end{eqnarray*}
  \begin{eqnarray*}
   \\&&+4988 {x_{(6)}}^{10} {x_{(8)}}^8 +42868 {x_{(6)}}^9 {x_{(8)}}^8 -181092{x_{(6)}}^8 {x_{(8)}}^8-320756 {x_{(6)}}^7 {x_{(8)}}^8-523916 {x_{(6)}}^6{x_{(8)}}^8 
    \\&&-523916 {x_{(6)}}^5 {x_{(8)}}^8 -320756 {x_{(6)}}^4 {x_{(8)}}^8-181092{x_{(6)}}^3 {x_{(8)}}^8+42868 {x_{(6)}}^2 {x_{(8)}}^8 +4988 {x_{(6)}} {x_{(8)}}^8
     \\&&+3108{x_{(8)}}^8-224 {x_{(6)}}^{12} {x_{(8)}}^7 -14856 {x_{(6)}}^{11} {x_{(8)}}^7 -48528{x_{(6)}}^{10} {x_{(8)}}^7-9192 {x_{(6)}}^9 {x_{(8)}}^7 \\&&
     +347248 {x_{(6)}}^8{x_{(8)}}^7+186248 {x_{(6)}}^7 {x_{(8)}}^7  +747408 {x_{(6)}}^6 {x_{(8)}}^7+186248{x_{(6)}}^5 {x_{(8)}}^7 \\&&
     +347248 {x_{(6)}}^4 {x_{(8)}}^7
      -9192 {x_{(6)}}^3{x_{(8)}}^7-48528 {x_{(6)}}^2 {x_{(8)}}^7  -14856 {x_{(6)}} {x_{(8)}}^7 -224{x_{(8)}}^7 
      \\&&
      -2520 {x_{(6)}}^{13} {x_{(8)}}^6+6280 {x_{(6)}}^{12} {x_{(8)}}^6 +46656{x_{(6)}}^{11} {x_{(8)}}^6+78176 {x_{(6)}}^{10} {x_{(8)}}^6\\
&& -181024 {x_{(6)}}^9{x_{(8)}}^6  -122464 {x_{(6)}}^8 {x_{(8)}}^6-416304 {x_{(6)}}^7 {x_{(8)}}^6-416304{x_{(6)}}^6 {x_{(8)}}^6
  \\ && -122464 {x_{(6)}}^5 {x_{(8)}}^6  -181024 {x_{(6)}}^4{x_{(8)}}^6+78176 {x_{(6)}}^3 {x_{(8)}}^6 +46656 {x_{(6)}}^2 {x_{(8)}}^6+6280 {x_{(6)}}{x_{(8)}}^6\\
 && -2520 {x_{(8)}}^6+672 {x_{(6)}}^{14} {x_{(8)}}^5 
 +5824 {x_{(6)}}^{13}{x_{(8)}}^5-35360 {x_{(6)}}^{12} {x_{(8)}}^5-88704 {x_{(6)}}^{11} {x_{(8)}}^5 \\&&+79968{x_{(6)}}^{10} {x_{(8)}}^5-17760 {x_{(6)}}^9 {x_{(8)}}^5  
  +290720 {x_{(6)}}^8{x_{(8)}}^5 +173280 {x_{(6)}}^7 {x_{(8)}}^5\\&&
  +290720 {x_{(6)}}^6 {x_{(8)}}^5-17760{x_{(6)}}^5 {x_{(8)}}^5 +79968 {x_{(6)}}^4 {x_{(8)}}^5 
 -88704 {x_{(6)}}^3{x_{(8)}}^5-35360 {x_{(6)}}^2 {x_{(8)}}^5 \\&&
 +5824 {x_{(6)}} {x_{(8)}}^5+672{x_{(8)}}^5+224 {x_{(6)}}^{15} {x_{(8)}}^4-2656 {x_{(6)}}^{14} {x_{(8)}}^4  +320{x_{(6)}}^{13} {x_{(8)}}^4\\ &&
 +114880 {x_{(6)}}^{12} {x_{(8)}}^4-88544 {x_{(6)}}^{11}{x_{(8)}}^4 +156576 {x_{(6)}}^{10} {x_{(8)}}^4-194400 {x_{(6)}}^9 {x_{(8)}}^4\\
&& -74400{x_{(6)}}^8 {x_{(8)}}^4 -74400 {x_{(6)}}^7 {x_{(8)}}^4-194400 {x_{(6)}}^6{x_{(8)}}^4+156576 {x_{(6)}}^5 {x_{(8)}}^4 -88544 {x_{(6)}}^4 {x_{(8)}}^4
 \\    && +114880{x_{(6)}}^3 {x_{(8)}}^4 +320 {x_{(6)}}^2 {x_{(8)}}^4 -2656 {x_{(6)}} {x_{(8)}}^4+224{x_{(8)}}^4+8640 {x_{(6)}}^{14} {x_{(8)}}^3 
 \\ &&
-54720 {x_{(6)}}^{13} {x_{(8)}}^3 -18880{x_{(6)}}^{12} {x_{(8)}}^3+42560 {x_{(6)}}^{11} {x_{(8)}}^3-207040 {x_{(6)}}^{10}{x_{(8)}}^3 
\\ &&
+412160 {x_{(6)}}^9 {x_{(8)}}^3 -365440 {x_{(6)}}^8 {x_{(8)}}^3 +412160{x_{(6)}}^7 {x_{(8)}}^3 -207040 {x_{(6)}}^6 {x_{(8)}}^3
\\ &&
+42560 {x_{(6)}}^5{x_{(8)}}^3-18880 {x_{(6)}}^4 {x_{(8)}}^3 -54720 {x_{(6)}}^3 {x_{(8)}}^3 +8640{x_{(6)}}^2 {x_{(8)}}^3-2560 {x_{(6)}}^{15} {x_{(8)}}^2
\\&& 
 +7040 {x_{(6)}}^{14}{x_{(8)}}^2+46720 {x_{(6)}}^{13} {x_{(8)}}^2 -68480 {x_{(6)}}^{12} {x_{(8)}}^2 +136960{x_{(6)}}^{11} {x_{(8)}}^2\\
 &&-141440 {x_{(6)}}^{10} {x_{(8)}}^2+21760 {x_{(6)}}^9{x_{(8)}}^2 +21760 {x_{(6)}}^8 {x_{(8)}}^2-141440 {x_{(6)}}^7 {x_{(8)}}^2 \\&&
 +136960{x_{(6)}}^6 {x_{(8)}}^2 -68480 {x_{(6)}}^5 {x_{(8)}}^2+46720 {x_{(6)}}^4{x_{(8)}}^2+7040 {x_{(6)}}^3 {x_{(8)}}^2 -2560 {x_{(6)}}^2 {x_{(8)}}^2
 \\ &&
 -19200{x_{(6)}}^{14} {x_{(8)}} +9600 {x_{(6)}}^{13} {x_{(8)}} -19200 {x_{(6)}}^{12}{x_{(8)}}+38400 {x_{(6)}}^{10} {x_{(8)}}-19200 {x_{(6)}}^9 {x_{(8)}} 
 \\
 &&+38400 {x_{(6)}}^8{x_{(8)}}-19200 {x_{(6)}}^6 {x_{(8)}}+9600 {x_{(6)}}^5 {x_{(8)}} -19200 {x_{(6)}}^4{x_{(8)}}+3200 {x_{(6)}}^{15}\\
 &&+3200 {x_{(6)}}^{14}
+3200 {x_{(6)}}^{13}+3200{x_{(6)}}^{12}-6400 {x_{(6)}}^{11}-6400 {x_{(6)}}^{10} -6400 {x_{(6)}}^9-6400{x_{(6)}}^8\\&&+3200 {x_{(6)}}^7 +3200 {x_{(6)}}^6+3200 {x_{(6)}}^5 +3200 {x_{(6)}}^4)
\end{eqnarray*} }
and 
{\small  \begin{eqnarray*}
&& F_2(x_{(6)}, x_{(8)}) = (-64 {x_{(6)}}^{11} {x_{(8)}}^2-80{x_{(6)}}^{11}+144 {x_{(6)}}^{10} {x_{(8)}}^3 +240 {x_{(6)}}^{10} {x_{(8)}}\\
&& -44{x_{(6)}}^9 {x_{(8)}}^4-432 {x_{(6)}}^9 {x_{(8)}}^3-88 {x_{(6)}}^9 {x_{(8)}}^2-480{x_{(6)}}^9 {x_{(8)}}-116 {x_{(6)}}^8 {x_{(8)}}^5 +720 {x_{(6)}}^8 {x_{(8)}}^4\\
&& -40{x_{(6)}}^8 {x_{(8)}}^3+960 {x_{(6)}}^8 {x_{(8)}}^2+480 {x_{(6)}}^8 {x_{(8)}}+156{x_{(6)}}^7 {x_{(8)}}^6-288 {x_{(6)}}^7 {x_{(8)}}^5 -548 {x_{(6)}}^7 {x_{(8)}}^4\\
&& -592{x_{(6)}}^7 {x_{(8)}}^3-1088 {x_{(6)}}^7 {x_{(8)}}^2-320 {x_{(6)}}^7 {x_{(8)}}+160{x_{(6)}}^7-136 {x_{(6)}}^6 {x_{(8)}}^7 +76 {x_{(6)}}^6 {x_{(8)}}^5\\
&& +1280 {x_{(6)}}^6{x_{(8)}}^4-320 {x_{(6)}}^6 {x_{(8)}}^3+2080 {x_{(6)}}^6 {x_{(8)}}^2-160 {x_{(6)}}^6{x_{(8)}} +49 {x_{(6)}}^5 {x_{(8)}}^8\\ &&
+324 {x_{(6)}}^5 {x_{(8)}}^7
 -330 {x_{(6)}}^5{x_{(8)}}^6+448 {x_{(6)}}^5 {x_{(8)}}^5-1396 {x_{(6)}}^5 {x_{(8)}}^4+368 {x_{(6)}}^5{x_{(8)}}^3 \\
 &&
 -1968 {x_{(6)}}^5 {x_{(8)}}^2+480 {x_{(6)}}^5 {x_{(8)}}  +51 {x_{(6)}}^4{x_{(8)}}^9-540 {x_{(6)}}^4 {x_{(8)}}^8+866 {x_{(6)}}^4 {x_{(8)}}^7\\
   \end{eqnarray*}
  \begin{eqnarray*}
  &&
 -1720 {x_{(6)}}^4{x_{(8)}}^6 +1708 {x_{(6)}}^4 {x_{(8)}}^5-560 {x_{(6)}}^4 {x_{(8)}}^4  +240 {x_{(6)}}^4{x_{(8)}}^3+960 {x_{(6)}}^4 {x_{(8)}}^2\\
 &&
 -480 {x_{(6)}}^4 {x_{(8)}}-70 {x_{(6)}}^3{x_{(8)}}^{10} +324 {x_{(6)}}^3 {x_{(8)}}^9-394 {x_{(6)}}^3 {x_{(8)}}^8  +880 {x_{(6)}}^3{x_{(8)}}^7\\
 &&-440 {x_{(6)}}^3 {x_{(8)}}^6+448 {x_{(6)}}^3 {x_{(8)}}^5-268 {x_{(6)}}^3{x_{(8)}}^4 +128 {x_{(6)}}^3 {x_{(8)}}^3-848 {x_{(6)}}^3 {x_{(8)}}^2\\
&& +320 {x_{(6)}}^3{x_{(8)}}-80 {x_{(6)}}^3+50 {x_{(6)}}^2 {x_{(8)}}^{11}-180 {x_{(6)}}^2{x_{(8)}}^{10} +450 {x_{(6)}}^2 {x_{(8)}}^9-1060 {x_{(6)}}^2 {x_{(8)}}^8\\
&& +1360{x_{(6)}}^2 {x_{(8)}}^7-1920 {x_{(6)}}^2 {x_{(8)}}^6+1700 {x_{(6)}}^2 {x_{(8)}}^5-1040{x_{(6)}}^2 {x_{(8)}}^4 +720 {x_{(6)}}^2 {x_{(8)}}^3\\
&& -80 {x_{(6)}}^2 {x_{(8)}} -27{x_{(6)}} {x_{(8)}}^{12}+72 {x_{(6)}} {x_{(8)}}^{11}-180 {x_{(6)}} {x_{(8)}}^{10}+284{x_{(6)}} {x_{(8)}}^9 -311 {x_{(6)}} {x_{(8)}}^8\\ &&
+204 {x_{(6)}} {x_{(8)}}^7  +66 {x_{(6)}}{x_{(8)}}^6-288 {x_{(6)}} {x_{(8)}}^5+396 {x_{(6)}} {x_{(8)}}^4-272 {x_{(6)}}{x_{(8)}}^3+56 {x_{(6)}} {x_{(8)}}^2\\
&& +7 {x_{(8)}}^{13}+28 {x_{(8)}}^{11}+7{x_{(8)}}^9-70 {x_{(8)}}^7-28 {x_{(8)}}^5+56 {x_{(8)}}^3).
\end{eqnarray*} }  %%%%%page 24
Actually, the computer outputs are 
$-10 ({x_{(6)}}-1) {x_{(6)}}$ $ F_1(x_{(6)}, x_{(8)})/A$ and
$10({x_{(8)}}-{x_{(6)}}) F_2(x_{(6)}, x_{(8)})/B$,
 where 
$A = 27 {x_{(8)}}^3 (2 {x_{(6)}}^4-{x_{(6)}}^2 {x_{(8)}}^2+6{x_{(6)}}^2-{x_{(8)}}^4-{x_{(8)}}^2+2)^3$ and
$ B= 3 {x_{(6)}}^2 {x_{(8)}}^2 (-2 {x_{(6)}}^4+{x_{(6)}}^2 {x_{(8)}}^2-6
   {x_{(6)}}^2+ {x_{(8)}}^4+{x_{(8)}}^2-2)^2$,
but we omit $A$ and $B$ since these are non zero.

\smallskip
 If either  $x_{(8)} - x_{(6)} = 0$  or $x_{(6)}=1$ then the solutions obtained 
correspond to naturally reductive Einstein metrics.

Next, we solve the equations $F_1 = 0, F_2 = 0$.
We consider the polynomial ring $R= {\mathbb Q}[z, x_{(6)}, x_{(8)}] $ and the ideal $I$  generated by the polynomials $\{ \,z \, x_{(6)} \, x_{(8)}-1,$  $F_1, \, F_2 \}$.   We take a lexicographic ordering $>$   with $ z >  x_{(6)} > x_{(8)}$ for a monomial ordering on $R$. Then, by the aid of computer, we see that a  Gr\"obner basis for the ideal $I$ contains a  polynomial  of   $x_{(8)}$  given by 
$$({x_{(8)}}-1)^4 \left(3 {x_{(8)}}^2+4 {x_{(8)}}+8\right) \left(5
   {x_{(8)}}^2+2\right)^3 \left(9 {x_{(8)}}^2+5\right)^2 g_1(x_{(8)}), 
$$
where 
%%%%%page 24&&&&&&&&
\begin{footnotesize}
\begin{eqnarray*}
& & 
g_1(x_{(8)})= 806688936348626758637763 {x_{(8)}}^{50}-13263262804910158271167230
   {x_{(8)}}^{49}\\
&&+122600899294485079451070969
   {x_{(8)}}^{48}-822398683556288995256129652
   {x_{(8)}}^{47}\\
&&+4399773837901125736682019222
   {x_{(8)}}^{46}-19777767225231052945149636492
   {x_{(8)}}^{45}\\
&&+77189030664419243971210900356
   {x_{(8)}}^{44}-267412481847662378680151841744
   {x_{(8)}}^{43}\\
&&+835652804996925618391403544816
   {x_{(8)}}^{42}-2384123593335214303646778048624
   {x_{(8)}}^{41}\\
&&+6267932916334423480011027655836
   {x_{(8)}}^{40}-15295986913633691395698095682816
   {x_{(8)}}^{39}\\
&&+34849018083414433690580776827648
   {x_{(8)}}^{38}-74464685236414498523260321269792
   {x_{(8)}}^{37}\\
&&+149775187669204629807886224706848
   {x_{(8)}}^{36}-284392847271756402839231579737856
   {x_{(8)}}^{35}\\
&&+510957262333041322777745962488048
   {x_{(8)}}^{34}-870202594655466073484714801237984
   {x_{(8)}}^{33}\\
&&+1406785689907755973691204239043056
   {x_{(8)}}^{32}-2161000636271658048523009797379136
   {x_{(8)}}^{31}\\
 %   \end{eqnarray*} }
%{\small \begin{eqnarray*}
&&+3156573099492698926305125730681312
   {x_{(8)}}^{30}-4386411014617644307370320689532352
   {x_{(8)}}^{29}\\
&&+5800052119705927583296215652661568
   {x_{(8)}}^{28}-7297773273562989425436434239009536
   {x_{(8)}}^{27}\\
&&+8735998247540128939755053826847872
   {x_{(8)}}^{26}-9946218172356372780360570705842688
   {x_{(8)}}^{25} \\
 %   \end{eqnarray*} 
%\begin{eqnarray*}
&&+10765377064651986382470072553346112
   {x_{(8)}}^{24}-11070907901497474448561753730673152
   {x_{(8)}}^{23} \\
 %   \end{eqnarray*} 
%\begin{eqnarray*} 
&&+10810407449543829554208165804540672
   {x_{(8)}}^{22}-10016169623180806451167033142842368
   {x_{(8)}}^{21}  
\\&&+8799062122173126538327519312214016
   {x_{(8)}}^{20}-7323154324975883645506836366458880
   {x_{(8)}}^{19}
     \end{eqnarray*} 
\begin{eqnarray*} 
&&+5769126123281919878491583853674496
   {x_{(8)}}^{18}-4297867752440180606285372758228992
   {x_{(8)}}^{17} \\
  % \end{eqnarray*} 
%\begin{eqnarray*} 
&&+3024464692354751261405789995008000
   {x_{(8)}}^{16}-2007868005418133840016358187728896
   {x_{(8)}}^{15}
\\
&&+1255570192626380637861179751923712
   {x_{(8)}}^{14}-738157675727844934568279109795840
   {x_{(8)}}^{13}\\
&&+407056027405833286440301394657280
   {x_{(8)}}^{12}-209943346903337341167588291379200
   {x_{(8)}}^{11}\\
&&+100903380382253021546263923916800
   {x_{(8)}}^{10}-44980383610420820508382593024000
   {x_{(8)}}^9\\
&&+18483164273291582549151186944000
   {x_{(8)}}^8-6943558170619368927689441280000
   {x_{(8)}}^7\\
&&+2358038540179746141860003840000
   {x_{(8)}}^6-712584577413002234429440000000
   {x_{(8)}}^5\\
&&+187291856473914145872281600000
   {x_{(8)}}^4-41355915143242649174016000000
   {x_{(8)}}^3\\
&&+7244299545842737479680000000
   {x_{(8)}}^2-901211137791366266880000000 {x_{(8)}}\\
&&+59960536211694551040000000. 
   \end{eqnarray*} 
   \end{footnotesize}
 %  \end{small}
We consider an ideal $J$ generated by  the polynomials $\{ F_1, \, F_2,  g_1 \}$.  
  We take a lexicographic ordering $>$   with $  x_{(6)} > x_{(8)}$ for a monomial ordering on $R$. Then, by the aid of computer, we see that a  Gr\"obner basis for the ideal $J$ contains the  polynomial $g_1$ of   $x_{(8)}$ and 
  a polynomial $h_1$ of  $x_{(6)}$ and  $x_{(8)}$ of the form   
  $$ h_1(x_{(6)}, x_{(8)}) = a x_{(6)} + \sum_{k =0}^{49} b_k {x_{(8)}^{}}^{k},
  $$
  where $a \in {\mathbb R}$ and $b_i  \in {\mathbb R}$ ($ i = 0,1,\dots, 49$). 
 By solving $ g_1= 0$ and $h_1=0$
approximately, we obtain the following results:
\begin{itemize}
\item[(1)] $(x_{(6)}, x_{(8)}) \approx (1.887796062233598, 1.815613725084982)$
\item[(2)] $(x_{(6)}, x_{(8)}) \approx(0.5297182359925161,  0.9617636996958176)$.
%\item[(3)] $(x_{(6)}, x_{(8)}) = (1.032633825555096, 0.04620725426268913)$
%\item[(4)] $(x_{(6)}, x_{(8)}) = (0.02537901781555798, 1.035934945885168)$
\end{itemize}

\noindent By substituting these values  into $u_2, u_3, v_4, v_5$,  we obtain the following:
For the solutions $(1)$ we have
$$
(u_2, u_3, v_4, v_5) \approx (0.614275909576,  0.790016897212,  1.4193906403596, 1.9248702704348)
  $$
 and for the solutions $(2)$ we have
$$
(u_2, u_3, v_4, v_5) \approx ( 0.4184863571955,  0.325393151233,  
1.3614688261843,  0.5631000275946).
$$
%\item[] For $(3)$ we take $u_3 < 0$
%\item[] For $(4)$ we take $u_3<0$.

From the above computations we obtain the following:

\begin{theorem}  The compact Lie group
$\SU(5)$ admits  two non naturally reductive Einstein metrics which correspond to  $\Ad(\s(\U(1)\times\U(2)\times\U(2)))$-invariant inner products of the form {\rm (\ref{metrics1})}. 
 \end{theorem}

 It is possible to show that 
 the compact Lie group $\SU(n+3)$ admits 
{\it two} left-invariant non naturally reductive Einstein metrics, which correspond to  $\Ad(\s(\U(1)\times\U(2)\times\U(n)))$-invariant inner products of the form {\rm (\ref{metrics1})},  for $2\leq n\leq 12$.
Also, we {\it conjecture} that for
$n\geq 13$, $\SU(n+3)$ admits     
   {\it four} left-invariant non naturally reductive Einstein metrics.  In this case  the difficulty is to find, for general $n$, a
   Gr\"obner basis   for the system of polynomials.

\subsection{A generalization of Mori's result} %page 26
Now we consider the cases when   $\ell=m=2$, $ n \geq 2$ and  $c=0$, so that $x_{(7)}=x_{(8)}=1$. 
 In \cite{M} K. Mori  proved existence of one Einstein metric on $\SU(4+n)$, which corresponds to  $\Ad(\s(\U(2)\times\U(2)\times\U(n)))$-invariant inner products of the form {\rm (\ref{metrics1})}.
 We generalize this result as follows:
 
 \begin{theorem}  The compact Lie group
$\SU(4+n)$ admits two non naturally reductive Einstein metrics for $ 2 \leq n \leq 25$  and four   non naturally reductive Einstein metrics for $n\ge 26$,  which correspond to  $\Ad(\s(\U(2)\times\U(2)\times\U(n)))$-invariant inner products of the form {\rm (\ref{metrics1})}. 
 \end{theorem}
 \begin{proof}
 To find non naturally reductive Einstein metrics on $\SU(4+n)$ we will use  Propositions \ref{ricci1} and \ref{ricci3}. 
We consider the system of equations  
\begin{eqnarray*}\label{mori1}
&&   r_{1}-\lambda =0,   r_2 -\lambda =0,\,   r_3 -\lambda =0, \,  r_4 -\lambda =0,\nonumber\\
&&  r_{5}-\lambda =0, \,   r_{6} -\lambda =0, \,   r_{7} -\lambda =0, \  r_{8} -\lambda =0.
\end{eqnarray*}

 From $r_4 -\lambda =0$ we see that $1/4 (-4 \lambda + v_4)=0$  and from $ r_{7} -\lambda =0, \  r_{8} -\lambda =0$, we obtain that $u_2 =u_1$.  By substituting these values into  $r_5-\lambda=0 $,  $r_6-\lambda=0 $ and $r_7-\lambda=0$, we obtain the equations:     
\begin{eqnarray*}\label{mori2}
&&   -n {v_4} {x_{(6)}}^2+n {v_5} {x_{(6)}}^2-4 {v_ 4} {x_{(6)}}^2+4 {v_5} = 0,\nonumber \\ 
 && n {v_4} {x_{(6)}}^2-n {x_{(6)}}^3+3 {u_1}+4 {v_4} {x_{(6)}}^2+{v_5} -8 {x_{(6)}} = 0, \\ 
 && -4 n^2 {u_3}-4 n^2 {v_4}+8 n^2-6 n {u_1}-17 n {v_4}-n{v_5}-8 n {x_{(6)}}+32 n+4 {u_3}-4 {v_4} = 0. \nonumber
\end{eqnarray*} 

By solving these equations with respect to $u_1,u_3, v_4$, we obtain 
\begin{eqnarray*}
& & {u_1}= \frac{1}{3} \left(-n {v_5}{x_{(6)}}^2+n {x_{(6)}}^3-5 {v_5}+8 {x_{(6)}}\right), \\
& & {u_3}= -\frac{1} {2 \left(n^2-1\right) {x_{(6)}}^2}(-n^2 {v_5}
   {x_{(6)}}^4+2 n^2 {v_5} {x_{(6)}}^2+n^2
   {x_{(6)}}^5-4 n^2 {x_{(6)}}^2  \nonumber\\
   & & -4 n {v_5} {x_{(6)}}^2+8 n {v_5}+12 n {x_{(6)}}^3-16 n {x_{(6)}}^2+2 {v_5}),  \\
   & & {v_4}=  \frac{(n  {x_{(6)}}^2+4 ){v_5}}{(n+4){x_{(6)}}^2}
\end{eqnarray*}
Now we see that the equations $ r_{1}-\lambda =0$ and  $r_{3}-\lambda =0$ become respectively:
\begin{eqnarray*}
& & n {u_1}^2 {x_{(6)}}^2-n {u_1} {v_4} {x_{(6)}}^2+2 {u_1}^2-4 {u_1} {v_4} {x_{(6)}}^2+2 {x_{(6)}}^2 = 0, \\
& & -n {u_3} {v_4}+n+4 {u_3}^2-4 {u_3} {v_4} =0. 
\end{eqnarray*}
By substituting the values $u_1,  v_4$ into $ r_{1}-\lambda =0$, we obtain  
\begin{eqnarray}\label{mori3}
& & \frac{1}{9} ({x_{(6)}}-{v_5}) \big(-n^3 {v_5} {x_{(6)}}^6+n^3 {x_{(6)}}^7-15 n^2 {v_5}
   {x_{(6)}}^4+18 n^2 {x_{(6)}}^5  \nonumber\\
  & & -72 n {v_5} {x_{(6)}}^2+96 n {x_{(6)}}^3-110 {v_5}+146 {x_{(6)}}\big) =0,  
\end{eqnarray}
and then by substituting the values $u_3,  v_4$ into $ r_{3}-\lambda =0$, we obtain
 \begin{eqnarray}\label{mori4}
  & & \frac{n}{2 (n-1)^2 (n+1)^2
   {x_{(6)}}^4}  \big(-n^4 {v_5}^2
   {x_{(6)}}^6+2 n^4 {v_5}^2 {x_{(6)}}^4+n^4 {v_5}
   {x_{(6)}}^7-4 n^4 {v_5} {x_{(6)}}^4+2 n^4 {x_{(6)}}^4\nonumber\\
   & & +2 n^3 {v_5}^2 {x_{(6)}}^8-8 n^3
   {v_5}^2 {x_{(6)}}^6+16 n^3 {v_5}^2
   {x_{(6)}}^2-4 n^3 {v_5} {x_{(6)}}^9+8 n^3 {v_5}
   {x_{(6)}}^7+16 n^3 {v_5} {x_{(6)}}^6\nonumber\\
   & & +16 n^3 {v_5} {x_{(6)}}^5    -48 n^3 {v_5} {x_{(6)}}^4-16
   n^3 {v_5} {x_{(6)}}^2+2 n^3 {x_{(6)}}^{10}-16 n^3
   {x_{(6)}}^7+32 n^3 {x_{(6)}}^4\nonumber\\
   & &+17 n^2 {v_5}^2
   {x_{(6)}}^6-66 n^2 {v_5}^2 {x_{(6)}}^4+50 n^2
   {v_5}^2 {x_{(6)}}^2+32 n^2 {v_5}^2-65 n^2
   {v_5} {x_{(6)}}^7+64 n^2 {v_5} {x_{(6)}}^6 \nonumber\\
      & &+128
   n^2 {v_5} {x_{(6)}}^5-60 n^2 {v_5}
   {x_{(6)}}^4+48 n^2 {v_5} {x_{(6)}}^3-192 n^2
   {v_5} {x_{(6)}}^2+48 n^2 {x_{(6)}}^8-64 n^2
   {x_{(6)}}^7 \nonumber\\ & &
   -192 n^2 {x_{(6)}}^5+252 n^2 {x_{(6)}}^4+32
   n {v_5}^2 {x_{(6)}}^4-128 n {v_5}^2
   {x_{(6)}}^2+136 n {v_5}^2-200 n {v_5}
   {x_{(6)}}^5\nonumber\\
    & &
    +272 n {v_5} {x_{(6)}}^4+384 n {v_5}
   {x_{(6)}}^3-528 n {v_5} {x_{(6)}}^2+288 n
   {x_{(6)}}^6-768 n {x_{(6)}}^5+512 n {x_{(6)}}^4  \nonumber\\
      &&-18 {v_5}^2 {x_{(6)}}^2+32 {v_5}^2+48 {v_5}
   {x_{(6)}}^3-64 {v_5} {x_{(6)}}^2+2
   {x_{(6)}}^4 \big) =0. 
\end{eqnarray} 
From equation (\ref{mori3}) we obtain that 
\begin{eqnarray}
& & v_5 = x_{(6)}, \nonumber \\
& & \mbox{or}\  {v_5} =\frac{n^3 {x_{(6)}}^7+18 n^2 {x_{(6)}}^5+96 n
   {x_{(6)}}^3+146 {x_{(6)}}}{n^3 {x_{(6)}}^6+15 n^2
   {x_{(6)}}^4+72 n {x_{(6)}}^2+110}.\label{mori6}
 \end{eqnarray} %%
 If $v_5 = x_{(6)}$, then we obtain naturally reductive metrics. 
We then consider the case (\ref{mori6}). By substituting the value $v_5$ into (\ref{mori4})
we obtain  an equation for $x_{(6)}$ of degree 16:
{\small \begin{eqnarray*}
& F(x_{(6)}, n) & = n^7 (2 n+5) \left(n^2+4 n+9\right) {x_{(6)}}^{16}-4 n^7 (n+4) \left(n^2+8 n+19\right) {x_{(6)}}^{15} \\
&&+2 n^6 \left(n^4+60 n^3+400 n6 n^2+996 n+763\right) {x_{(6)}}^{14} 
\\ & & 
-4 n^6 (n+4) \left(37 n^2+29+607\right) {x_{(6)}}^{13} \\
&&+4 n^5 \left(15 n^4+650 n^3+4333 n^2+9854 n+5310\right) {x_{(6)}}^{12}
\\&&-120 n^5 (n+4) \left(19 n^2+152 n+265\right) {x_{(6)}}^{11} \\
&&+2 n^4 \left(369 n^4+14356 n^3+94595
   n^2+200356 n+77916\right) {x_{(6)}}^{10}\\ &&
   -32 n^4 (n+4)
   \left(593 n^2+4744 n+6842\right) {x_{(6)}}^9
   \\& & +8 n^3\left(595 n^4+22698 n^3+146413 n^2+288990
   n+80098\right) {x_{(6)}}^8 \\&&
   -8 n^3
   (n+4) \left(11525 n^2+92200 n+104939\right)
   {x_{(6)}}^7 \\ & &
   +8 n^2 \left(2121 n^4+84880
   n^3+529391 n^2+958984 n+179028\right) {x_{(6)}}^6 \\ & &
   -96 n^2 (n+4) \left(2725 n^2+21800 n+17941\right) {x_{(6)}}^5  \\
   && +8 n \left(3960 n^4+182641 n^3+1078736 n^2+1715695 n+186684\right) {x_{(6)}}^4 \\
   & & -16 n (n+4) \left(25087 n^2+200696 n+98017\right) {x_{(6)}}^3 \\ &&
   +8 \left(3025 n^4+203160
    n^3+1087279 n^2+1316352 n+51424\right) {x_{(6)}}^2\\ & & 
    -256960 (n+4) \left(n^2+8 n+1\right) {x_{(6)}}  
+170528 (n+4) (4 n+1) = 0.
 \end{eqnarray*} }
 For $n=2$ we have
\begin{eqnarray*} && F(x_{(6)}, 2)  = 288 ({x_{(6)}}-1) (84 {x_{(6)}}^{15}-332
   {x_{(6)}}^{14}+1824 {x_{(6)}}^{13}-5360
   {x_{(6)}}^{12}+15880 {x_{(6)}}^{11} \\ && -35720
   {x_{(6)}}^{10}+72920 {x_{(6)}}^9-126568
   {x_{(6)}}^8+192284 {x_{(6)}}^7-254968
   {x_{(6)}}^6+292536 {x_{(6)}}^5 \\ & & -286992
   {x_{(6)}}^4+238425 {x_{(6)}}^3-161413
   {x_{(6)}}^2+80446 {x_{(6)}}-31974.  
       \end{eqnarray*}
       Thus we see that the solutions are given by ${x_{(6)}}=1$ and ${x_{(6)}}\approx 1.17941$. 
 
 For $n\geq 3$ we see, for $x_{(6)} =1$, 
 $$F(1, n) = -3 (n-2) (n+2)^2 (n+6)^2 \left(n^4+14 n^3+69 n^2+134 n+76\right) < 0. $$ 
 
 We also see that   
 \begin{eqnarray*} &&F(2, n) = 32 \big(1024 n^{10}+8704 n^9+34368 n^8+87488 n^7+164144
   n^6+239040 n^5\\  & &+274217 n^4 +242156 n^3+150555 n^2+55429
   n+8500\big) > 0. 
      \end{eqnarray*}
 
 We have 
{\small  \begin{eqnarray*} && F(1/2, n)  = \frac{1}{65536}  \big(2 n^{10}+173 n^9+4494 n^8-20915 n^7-3177896
   n^6-61117056 n^5 \\ & & -449950208 n^4-272740352
   n^3+11681267712 n^2+39390150656 n+17762877440\big)  \\&& 
   = \frac{1}{65536} \big(2 (n-26)^{10}+693 (n-26)^9+105816 (n-26)^8+9342205
   (n-26)^7 \\ & & +525422358 (n-26)^6+19518844644
   (n-26)^5+479005265720 (n-26)^4\\ & &+7495082389872
   (n-26)^3+68109703630368 (n-26)^2+279535235098560
   (n-26)\\ & & +82793227996800
   \big),
       \end{eqnarray*} }
       and thus  we see that $F(1/2, n) > 0$, for $n \geq 26$. 
       
     We  also have   
{\small \begin{eqnarray*} && F(292/(55 n), n) = -\frac{511584}{7011372354671045074462890625 n^9}\\ & & 
  \times \big(35056861773355225372314453125
   n^{10}+1131629115268488355792109375000
   n^9 \\ & & +13721578697735127898701000000000
   n^8+70465121420063004319160003000000
   n^7\\ 
 %%%%%%%  
   & & +66607303696117170858210917600000
   n^6-592316367137587606416109021900800
   n^5\\ & & -861191206103397273418047743426560
   n^4+642183745708072078163896886362112
   n^3\\
% \end{eqnarray*} }
%{\small \begin{eqnarray*}   
   & & -38667489524999416179489857656061952
   n^2-190447581247326183326819893905457152
   n\\ & & -245709067472143332413008967400161280\big)\\   
   & & 
  =  -\frac{511584}{7011372354671045074462890625 n^9}
    \times  \big(35056861773355225372314453125
  (n-4)^{10}\\ & &+2533903586202697370684687500000
   (n-4)^9 \\ & & +79701167324216470975283343750000
   (n-4)^8\\ & &+1430610708561604520873222003000000
   (n-4)^7\\ & &+16155192972662188907166744001600000
   (n-4)^6\\ & &+119408859580384221935668378008499200
   (n-4)^5\\ & &+583172621773938265497841838762557440
   (n-4)^4\\ & &+1853844985914852789445447754895409152
   (n-4)^3\\ & &+3622885223042644610041726583690821632
   (n-4)^2\\ & &+3539894892993878348278554600791605248
   (n-4)\\ & &+247914639656696354355632138106175488 \big) <0, \ \ \  \mbox{for } n \geq 4.
 \end{eqnarray*}  }
 
 %%%page 29
 Thus we obtain that, for $ 2 \leq n \leq 25$ there exist two positive solutions  and  for $n\ge 26$
there exist four  positive solutions of  the equation $F(x_{(6)}, n) = 0$. 

By substituting the value of $v_5$ into $u_1, u_3$ and $v_4$, we also have the following: 
\begin{eqnarray}
& & {u_1}= \frac{2 {x_{(6)}} \left(n {x_{(6)}}^2+5\right)}{n^2 {x_{(6)}}^4+10 n {x_{(6)}}^2+22},  \nonumber\\
& & {u_3}= -\frac{1}{2 (n-1) (n+1) {x_{(6)}} \left(n {x_{(6)}}^2+5\right) \left(n^2 {x_{(6)}}^4+10 n {x_{(6)}}^2+22\right)} \label{mori113} \\ & & \times \Big(2 n^5 {x_{(6)}}^8-4
   n^5 {x_{(6)}}^7+5 n^4 {x_{(6)}}^8-16 n^4 {x_{(6)}}^7+44
   n^4 {x_{(6)}}^6-60 n^4 {x_{(6)}}^5+86 n^3
   {x_{(6)}}^6  \nonumber \\ & & -240 n^3 {x_{(6)}}^5+336 n^3 {x_{(6)}}^4-288
   n^3 {x_{(6)}}^3+480 n^2 {x_{(6)}}^4-1152 n^2
 {x_{(6)}}^3+1060 n^2 {x_{(6)}}^2 \nonumber 
 \\ & & -440 n^2 {x_{(6)}}+928 n {x_{(6)}}^2-1760 n {x_{(6)}}+1168 n+292\Big),  \nonumber \\
   & & {v_4}=  \frac{\left(n {x_{(6)}}^2+4\right) \left(n^3 {x_{(6)}}^6+18 n^2 {x_{(6)}}^4+96 n {x_{(6)}}^2+146\right)}{(n+4) {x_{(6)}} \left(n {x_{(6)}}^2+5\right) \left(n^2 {x_{(6)}}^4+10 n {x_{(6)}}^2+22\right)}. \nonumber
\end{eqnarray}
We claim that  the value of ${u_3}$ in (\ref{mori113}) is positive, whenever ${x_{(6)}}$  is a solution of  $F(x_{(6)}, n) = 0.$
Indeed, from  equation (\ref{mori113}) it follows that %page 29
\begin{eqnarray*} & & 
G(x_{(6)},u_3)= 2 n^5 {u_3} {x_{(6)}}^7+2 n^5 {x_{(6)}}^8-4 n^5
   {x_{(6)}}^7+30 n^4 {u_3} {x_{(6)}}^5+5 n^4
   {x_{(6)}}^8-16 n^4 {x_{(6)}}^7
   \\ & &
   +44 n^4 {x_{(6)}}^6  -60
   n^4 {x_{(6)}}^5-2 n^3 {u_3} {x_{(6)}}^7+144 n^3
   {u_3} {x_{(6)}}^3+86 n^3 {x_{(6)}}^6-240 n^3
   {x_{(6)}}^5\\ & &
   +336 n^3 {x_{(6)}}^4-288 n^3 {x_{(6)}}^3  
   -30 n^2 {u_3} {x_{(6)}}^5+220 n^2 {u_3}
   {x_{(6)}}+480 n^2 {x_{(6)}}^4-1152 n^2
   {x_{(6)}}^3\\ & &
   +1060 n^2  {x_{(6)}}^2-440 n^2 {x_{(6)}} 
   -144 n {u_3} {x_{(6)}}^3+928 n {x_{(6)}}^2-1760 n
   {x_{(6)}}+1168 n\\ & &
   -220 {u_3} {x_{(6)}}+292= 0. 
\end{eqnarray*}
Now, by taking the resultant of ${\rm Res}_{x_{(6)}}(F(x_{(6)}, n), G(x_{(6)},u_3)$ with respect to $x_{(6)}$, 
we obtain the equation of ${u_3} $:
%%%%%%%%%%%%%%%%%%%%%%%
%{\footnotesize 
{\small 
\begin{eqnarray*} & & {\rm Res}_{x_{(6)}}(F, G) = 79805105467783573929984 (n-1)^{16} n^{49} (n+1)^{16}
   (n+3)^4 (2 n+3)^2 (2 n+5) \\ 
   & & 
  \times (4 n+1) \Big(32 (n+4) \left(n^2+4 n+5\right)^2 \left(n^2+4 n+9\right)
   \left(n^2+4 n+11\right)^4 {u_3}^{16} -64 (n+4) \\ &&
   \times (n^2+4 n+5) (n^2+4 n+11)^3
   (8 n^6+96 n^5+521 n^4+1608 n^3+2824 n^2+2528
   n+495) {u_3}^{15}\\ 
   & &+24 \left(n^2+4
   n+11\right)^2 (160 n^{11}+3821 n^{10}+42144
   n^9+283768 n^8+1292016 n^7+4148638 n^6\\ 
   & & 
   +9485824
   n^5  +15185192 n^4+16188848 n^3+10306293 n^2+3014816
   n+154272 ) {u_3}^{14} \\ 
   & &-16 n (n+4) \left(n^2+4
   n+11\right) (1120 n^{11}+26481 n^{10}+294364
   n^9+2029009 n^8+9595280 n^7 \\ 
   & &
   +32503106 n^6+79894120
   n^5+141057818 n^4+172564400 n^3+135486477 n^2+58005020 n
    \\ 
   & &+8732629 ) {u_3}^{13} 
+4 n  (14560
   n^{14}+455546 n^{13}+6737593 n^{12}+62487392
   n^{11}+405558430 n^{10}\\ 
   & &
   +1943683970 n^9+7067440828
   n^8  +19705190624 n^7+42032951806 n^6+67552851414
   n^5
   \\ 
   & &+79223664915 n^4+63957854256 n^3+31859786756
   n^2+7711909438 n+410101496 ) {u_3}^{12}\\ 
   & & -48 n^2 (n+4)  (2912 n^{12}+78078 n^{11} +980205 n^{10}+7626696 n^9+40900061 n^8+158743428 n^7 
   \\ 
   & &+455081712 n^6+964514064 n^5+1485523623
   n^4+1597309366 n^3+1109259109 n^2
   \\ 
   & &+424815888 n+59882762 ) {u_3}^{11}+2 n^2 (128128
   n^{13}+3871868 n^{12}+54127876 n^{11} \\
  % \end{eqnarray*} %%%%page 30%%%
  % \begin{eqnarray*}
   & &+464149161 n^{10}+2723268432 n^9+11520926166 n^8 +36006811328
   n^7+83614509733 n^6
      \end{eqnarray*} %%%%page 30%%%
   \begin{eqnarray*}
   & &+142792464992 n^5+174250392508
   n^4+143668803716 n^3+71910270196 n^2+17301717480 n
   \\ 
   & &+926030448 ) {u_3}^{10}-4 n^3 (n+4)
     (91520 n^{11}+2334332 n^{10}+27152796
   n^9
   +190488357 n^8\\ 
   & &+895297632 n^7+2951744016
   n^6 +6926787136 n^5+11464153999 n^4+12921488592
   n^3 \\ 
   & &+9210867932 n^2+3562013612 n+504494076 )
   {u_3}^9  +3 n^3  (137280 n^{12} +3936504
   n^{11}\\ 
   & &+51188148 n^{10}+399221022 n^9 +2077926319
   n^8+7583085392 n^7 +19785908796 n^6\\ 
   & &+36840346114
   n^5+47825033196 n^4+41093570040 n^3 +21077483624 n^2 +5140865464 n\\ 
   & &+281471440 ) {u_3}^8 -8 n^4
   (n+4)  (45760 n^{10}+1085656 n^9+11454564
   n^8+70805726 n^7 \\ 
   & &+283650905 n^6+766586712 n^5 +1404513863 n^4+1699119089 n^3+1268244972 n^2\\ 
   & &+504006728 n+72600275 ) {u_3}^7+4 n^4
   (64064 n^{11}+1707706 n^{10}+20248844
   n^9+140694517 n^8\\ 
   & &+634497443 n^7+1939967836
   n^6+4068002539 n^5+5765215640 n^4+5273901874
   n^3\\ 
   & &+2822612617 n^2+707889047 n+38868874 )
   {u_3}^6-12 n^5 (n+4)  (11648 n^9+250068
   n^8\\ 
   & &+2328840 n^7+12317562 n^6+40609170 n^5+85937948
   n^4 +115337388 n^3+92656707 n^2\\ 
   & &+38780924
   n+5814614 ) {u_3}^5 +2 n^5  (29120
   n^{10}+703612 n^9+7421294 n^8+44788306 n^7 \\
   &&+170172755 n^6  +421324252 n^5+678403707 n^4+684577731
   n^3+395658522 n^2+105906354 n
   \\ 
   & &+6175980 )
   {u_3}^4  -8 n^6 (n+3) (n+4) \big(2240 n^7+35252
   n^6+227370 n^5+772912 n^4+1475857 n^3\\ 
   & &+1549629
   n^2+797301 n+140922\big) {u_3}^3  +6 n^6 (n+3)^2
   \big(640 n^7+9724 n^6+59268 n^5+184899 n^4\\ 
   & &+310716
   n^3+268199 n^2+98498 n+7209\big) {u_3}^2 -4 n^7
   (n+3)^3 (n+4) (2 n+3) \big(64 n^3+322 n^2\\ 
   & &+465
   n+159\big) {u_3} +n^7 (n+3)^4 (2 n+3)^2 (2 n+5)
   (4 n+1)  \Big)=0. 
   \end{eqnarray*}
   }
 By looking at the coefficients of the polynomial, we see that if 
 $ {\rm Res}_{x_{(6)}}(F, G) $ has  real solutions, then these are positive. 
 \end{proof}

%%%%page 31%%%
\subsection{The case of $\SU(4)$ and $\SU(3)$}
 We will show that the compact Lie groups $\SU(4)$ and $\SU(3)$ admit only naturally reductive Einstein metrics of the form (\ref{metrics1}).
 For $\SU(4)$ we prove the following:

\begin{theorem}  The compact Lie group
$\SU(4)$  with metrics  corresponding to  $\Ad(\s(\U(1)\times\U(1)\times\U(2)))$-invariant inner products of the form {\rm (\ref{metrics1})} admits only naturally reductive Einstein metrics, that is, bi-invariant metric and the metric {\rm (\ref{metrics1})} with $x_{(6)}=u_3=v_5= 5/11$, $x_{(7)}=x_{(8)}= 1$, $v_4 = 73/55$.
 \end{theorem} 
\begin{proof}
Let $\ell = m =1$ and $n = 2$.  In this case we have $\fr{h}_1 = \fr{h}_2 = 0$, so we do not have $u_1$ and $u_2$ variables. 
To find Einstein metrics we need to solve the system of equations
\begin{equation}\label{einstein1}
r_3 - r_4=0,\ r_4 - r_5=0,\ r_5 - r_6=0,\ r_6 - r_7=0,\ r_7 - r_8=0, \ r_0 = 0.
\end{equation}
We set $x_{(7)} = 1$.  Then from $r_0 = 0$ we have that 
$
c =- ({x_{(8)}}-1) ({x_{(8)}}+1)/\big(\sqrt{2} ({x_{(8)}}^2+1 )\big).
$
By substituting $c$ into the first five equations of system (\ref{einstein1}), we obtain the system 
{ \begin{eqnarray*}
&& f_1 = 2 {u_3}^2 {x_{(8)}}^2+2 {u_3}^2-4 {u_3} {v_4} {x_{(8)}}^2-4
   {u_3} {v_4}+4 {x_{(8)}}^2 =0,\\
&&   f_2=  4 {v_4} {x_{(6)}}^2 {x_{(8)}}^4+8 {v_4}
   {x_{(6)}}^2 {x_{(8)}}^2+4 {v_4} {x_{(6)}}^2-8 {v_5} {x_{(6)}}^2
   {x_{(8)}}^2-4 {v_5} {x_{(8)}}^4-4 {v_5} {x_{(8)}}^2=0,
      \end{eqnarray*} 
 \begin{eqnarray*}
&&    f_3=4 {v_5}{x_{(6)}}^2 {x_{(8)}}+4 {v_5} {x_{(8)}}^3+4 {v_5} {x_{(8)}}-2 {x_{(6)}}^3
   {x_{(8)}}^2-2 {x_{(6)}}^3
        +2 {x_{(6)}} {x_{(8)}}^4\\&&-8 {x_{(6)}} {x_{(8)}}^3+4 {x_{(6)}} {x_{(8)}}^2-8 {x_{(6)}} {x_{(8)}}+2 {x_{(6)}} =0, \\
&&   f_4=   6 {u_3} {x_{(6)}}^2{x_{(8)}}^4+12 {u_3} {x_{(6)}}^2 {x_{(8)}}^2+6 {u_3} {x_{(6)}}^2+4
   {v_4} {x_{(6)}}^2 {x_{(8)}}^4+8 {v_4} {x_{(6)}}^2 {x_{(8)}}^2\\&&
   +4 {v_4}{x_{(6)}}^2+8 {v_5} {x_{(6)}}^2 {x_{(8)}}^4    -8 {v_5} {x_{(8)}}^6   -16 {v_5} {x_{(8)}}^4-8 {v_5} {x_{(8)}}^2+12 {x_{(6)}}^3 {x_{(8)}}^5\\&&
   +24 {x_{(6)}}^3 {x_{(8)}}^3+12 {x_{(6)}}^3 {x_{(8)}}-32 {x_{(6)}}^2 {x_{(8)}}^5-64{x_{(6)}}^2 {x_{(8)}}^3   -32 {x_{(6)}}^2 {x_{(8)}} -12 {x_{(6)}} {x_{(8)}}^7\\&&
   +32
   {x_{(6)}} {x_{(8)}}^6-28 {x_{(6)}} {x_{(8)}}^5+64 {x_{(6)}} {x_{(8)}}^4
   -20 {x_{(6)}} {x_{(8)}}^3+32 {x_{(6)}} {x_{(8)}}^2-4 {x_{(6)}} {x_{(8)}} =0, \\
%   \end{eqnarray*} }
%{ \begin{eqnarray*}
&&  f_5= ({x_{(8)}}-1)(3 {u_3} {x_{(6)}} {x_{(8)}}^5+3 {u_3} {x_{(6)}}
   {x_{(8)}}^4+6 {u_3} {x_{(6)}} {x_{(8)}}^3+6 {u_3}
   {x_{(6)}} {x_{(8)}}^2\\&&
   +3 {u_3} {x_{(6)}} {x_{(8)}}+3
   {u_3} {x_{(6)}} +2 {v_4} {x_{(6)}} {x_{(8)}}^5
  +2{v_4} {x_{(6)}} {x_{(8)}}^4+4 {v_4} {x_{(6)}}
   {x_{(8)}}^3+4 {v_4} {x_{(6)}} {x_{(8)}}^2\\&&
   +2 {v_4}
   {x_{(6)}} {x_{(8)}}+2 {v_4} {x_{(6)}}-4 {v_5}
   {x_{(6)}} {x_{(8)}}^3
  -4 {v_5} {x_{(6)}} {x_{(8)}}^2-16 {x_{(6)}} {x_{(8)}}^5-32 {x_{(6)}} {x_{(8)}}^3
 \\&&
  -16 {x_{(6)}} {x_{(8)}}+4 {x_{(8)}}^6+4 {x_{(8)}}^5+8
   {x_{(8)}}^4+8 {x_{(8)}}^3 +4 {x_{(8)}}^2+4 {x_{(8)}})=0.    
\end{eqnarray*} }

First we study the case when $x_{(8)} = 1$. Then $c=0$ and  equations $r_3-r_4=0$, $r_4 - r_5=0$, $r_5 - r_6=0$, $r_6 - r_7=0$  reduce to 
\begin{eqnarray*}
&& f_1 = {u_3}^2-2  {u_3} {v_4}+1 = 0,\\
&& f_2 = 4 {v_4} {x_{(6)}}^2-2 {v_5} {x_{(6)}}^2-2 {v_5} = 0,\\
&& f_3 = \left({x_{(6)}}^2+2\right) ({v_5}-{x_{(6)}}) = 0,\\
&& f_4 = 12 {u_3} {x_{(6)}}^2+8 {v_4} {x_{(6)}}^2+4 {v_5}{x_{(6)}}^2-16 {v_5}+24 {x_{(6)}}^3-64 {x_{(6)}}^2+32{x_{(6)}} = 0.
\end{eqnarray*}
From $f_3 = 0$ we have $v_5 = x_{(6)}$, so by substituting this into $f_2, f_4$  we obtain:
\begin{eqnarray*} &&
f_2 = -2 {x_{(6)}} \left(-2{v_4}{x_{(6)}}+{x_{(6)}}^2+1\right) = 0, \\ 
&&
f_4 = 4 {x_{(6)}} \left(3 {u_3} {x_{(6)}}+2 {v_4} {x_{(6)}}+7{x_{(6)}}^2-16 {x_{(6)}}+4\right) = 0.
\end{eqnarray*}
By solving $f_2=0$ with respect to $v_4$, we have
$
v_4 = (x_{(6)}^2+1)/2 {x_{(6)}}.
$
We substitute  this into $f_1 =0$ and $f_4=0$ and thus we obtain:
\begin{eqnarray*} &&
2 ({u_3}-{x_{(6)}}) ({u_3}{x_{(6)}}-1)/{x_{(6)}} = 0, \\ & & 4 {x_{(6)}} \left(3 {u_3} {x_{(6)}}+8 {x_{(6)}}^2-16{x_{(6)}}+5\right) = 0.
\end{eqnarray*}
From the first equation above we see that $u_3 = x_{(6)}$ or $ u_3 =  1/x_{(6)}$.  We substitute the $u_3 = x_{(6)}$ into the second equation above and we have 
$
11x_{(6)}^2 - 16x_{(6)} + 5 = 0,
$
whose solutions are $x_{(6)} =5/11$ and $x_{(6)} = 1$.  
Now we substitute $u_3 = 1/x_{(6)}$ and we obtain 
$
8x_{(6)}^2 -16x_{(6)} + 8 = 0
$
whose solution is $x_{(6)} = 1$.  

For $x_{(6)} = 1$ we have $u_3 = v_4 = v_5 = x_{(8)} = 1$.  This metric corresponds to a bi-invariant metric which is naturally reductive.  For $x_{(6)} =  5/11$ from the above we have $u_3 = v_5 =  5/11$, $v_4 =   73/55$, so from Proposition 5.2 we have that this metric is also naturally reductive.

Now we study the case when $x_{(8)} \neq 1$. By solving $ f_2=0, \ f_3=0,\  f_4=0$, we obtain  
\begin{eqnarray*}
&&{u_3}= -\frac{{x_{(8)}}}{3 {x_{(6)}}\left({x_{(6)}}^2+{x_{(8)}}^2+1\right)} \Big(8 {x_{(6)}}^4-16 {x_{(6)}}^3-4
   {x_{(6)}}^2 {x_{(8)}}^2+24 {x_{(6)}}^2 {x_{(8)}} \\
   & & +{x_{(6)}}^2-16 {x_{(6)}} {x_{(8)}}^2-16 {x_{(6)}}-4 {x_{(8)}}^4+8 {x_{(8)}}^3-5 {x_{(8)}}^2+12
   {x_{(8)}}-1\Big), 
   \end{eqnarray*}
  \begin{eqnarray*} %%%%%page 32
   & & 
   {v_4}= \frac{{x_{(8)}}
   \left({x_{(6)}}^2-{x_{(8)}}^2+4 {x_{(8)}}-1\right) \left(2
   {x_{(6)}}^2+{x_{(8)}}^2+1\right)}{2 {x_{(6)}} \left({x_{(8)}}^2+1\right)
   \left({x_{(6)}}^2+{x_{(8)}}^2+1\right)},\\
   & & 
   {v_5}= -\frac{{x_{(6)}}
   \left({x_{(8)}}^2+1\right) \left(-{x_{(6)}}^2+{x_{(8)}}^2-4 {x_{(8)}}+1\right)}{2
   {x_{(8)}} \left({x_{(6)}}^2+{x_{(8)}}^2+1\right)}.   
\end{eqnarray*}
By substituting these $u_3, \ v_4, \ v_5$  into $f_1, f_5$, we can see that the equations $f_1 =0$ and $f_5=0$ reduce to
 the polynomial equations  of $x_{(6)}$ and $x_{(8)}$:  
 \begin{eqnarray*}
&&
F_1(x_{(6)}, x_{(8)}) = 16 {x_{(6)}}^8 {x_{(8)}}^2+28 {x_{(6)}}^8-64 {x_{(6)}}^7
   {x_{(8)}}^2-88 {x_{(6)}}^7-16 {x_{(6)}}^6 {x_{(8)}}^4\\ &&
   +96 {x_{(6)}}^6 {x_{(8)}}^3 +40 {x_{(6)}}^6 {x_{(8)}}^2 +180 {x_{(6)}}^6 {x_{(8)}}+68
   {x_{(6)}}^6-32 {x_{(6)}}^5 {x_{(8)}}^4-192 {x_{(6)}}^5 {x_{(8)}}^3
 \\&&  -116 {x_{(6)}}^5 {x_{(8)}}^2-288 {x_{(6)}}^5 {x_{(8)}} -84 {x_{(6)}}^5  -12
   {x_{(6)}}^4 {x_{(8)}}^6-16 {x_{(6)}}^4 {x_{(8)}}^5+229 {x_{(6)}}^4
   {x_{(8)}}^4 \\&&
   +38 {x_{(6)}}^4 {x_{(8)}}^3+510 {x_{(6)}}^4 {x_{(8)}}^2+90
   {x_{(6)}}^4 {x_{(8)}} +125 {x_{(6)}}^4 +64 {x_{(6)}}^3 {x_{(8)}}^6-256
   {x_{(6)}}^3 {x_{(8)}}^5\\&&
   +152 {x_{(6)}}^3 {x_{(8)}}^4-688 {x_{(6)}}^3
   {x_{(8)}}^3+112 {x_{(6)}}^3 {x_{(8)}}^2 -432 {x_{(6)}}^3 {x_{(8)}} +24 {x_{(6)}}^3+8 {x_{(6)}}^2 {x_{(8)}}^8\\&&
   -64 {x_{(6)}}^2 {x_{(8)}}^7+182 {x_{(6)}}^2 {x_{(8)}}^6-204 {x_{(6)}}^2 {x_{(8)}}^5+576 {x_{(6)}}^2 {x_{(8)}}^4 -176 {x_{(6)}}^2 {x_{(8)}}^3\\&&  
   +494 {x_{(6)}}^2 {x_{(8)}}^2-36
   {x_{(6)}}^2 {x_{(8)}}+68 {x_{(6)}}^2+32 {x_{(6)}} {x_{(8)}}^8-64 {x_{(6)}}
   {x_{(8)}}^7+116 {x_{(6)}} {x_{(8)}}^6 \\&&
   -272 {x_{(6)}} {x_{(8)}}^5 +156 {x_{(6)}} {x_{(8)}}^4 
         -352 {x_{(6)}} {x_{(8)}}^3+92 {x_{(6)}}
   {x_{(8)}}^2-144 {x_{(6)}} {x_{(8)}}+20 {x_{(6)}}\\&&
   +4 {x_{(8)}}^{10}-16
   {x_{(8)}}^9  +33 {x_{(8)}}^8-78 {x_{(8)}}^7 +116 {x_{(8)}}^6
         -126{x_{(8)}}^5+166 {x_{(8)}}^4-82 {x_{(8)}}^3
         \\&&
         +80 {x_{(8)}}^2-18{x_{(8)}}+1=0,\\
%\end{eqnarray*} }
%\begin{eqnarray*} 
&&
F_5(x_{(6)}, x_{(8)})= -2 {x_{(6)}}^4 {x_{(8)}}-2 {x_{(6)}}^4+4 {x_{(6)}}^3
   {x_{(8)}}+{x_{(6)}}^2 {x_{(8)}}^3-5 {x_{(6)}}^2 {x_{(8)}}^2\\
   &&
   -5 {x_{(6)}}^2
   {x_{(8)}}+{x_{(6)}}^2 +4 {x_{(6)}} {x_{(8)}}^3+4 {x_{(6)}}
         {x_{(8)}}  +{x_{(8)}}^5-{x_{(8)}}^4-{x_{(8)}}+1=0. 
\end{eqnarray*} 
By taking the resultant ${\rm Res}_{x_{(8)}}(F_1, F_5)$ of $F_1$ and $F_5$ with respect to $x_{(8)}$,  we obtain 
\begin{eqnarray*}
&& {\rm Res}_{x_{(8)}}(F_1, F_5)= 2304 ({x_{(6)}}-1)^2 {x_{(6)}}^4 \left({x_{(6)}}^2+2\right)^4 \left(3
   {x_{(6)}}^2+2 {x_{(6)}}+1\right)\\&&
    \big(2048409 {x_{(6)}}^{16} -12016026 {x_{(6)}}^{15}
   +33986766 {x_{(6)}}^{14}-64594134 {x_{(6)}}^{13}\\&&
   +106943061
   {x_{(6)}}^{12}-161713368 {x_{(6)}}^{11} +193524028 {x_{(6)}}^{10}  -175182968
   {x_{(6)}}^9\\&&
   +135191836 {x_{(6)}}^8-99593120 {x_{(6)}}^7+64876480
   {x_{(6)}}^6 -34635680 {x_{(6)}}^5 \\&& +15920560 {x_{(6)}}^4-6174080
   {x_{(6)}}^3+1785280 {x_{(6)}}^2-412800 {x_{(6)}}+72000\big). 
   \end{eqnarray*}
   
   Now we can see that the factor of degree 16 in the polynomial   ${\rm Res}_{x_{(8)}}(F_1, F_5)$ of $x_{(6)}$ does not have real roots by using computer manipulation, thus 
 there are no solutions for the system of equations for $x_{(8)} \neq 1$. 
 If $x_{(6)}=1$ then Proposition \ref{reductive} (1, (ii)) implies that the metric is naturally reductive and this concludes the proof.
 \end{proof}
 %page 32
 
  We now proceed with $\SU(3)$ and we prove the following:
 \begin{theorem}  The compact Lie group
$\SU(3)$ admits  only naturally reductive Einstein metrics which correspond to  $\Ad(\s(\U(1)\times\U(1)\times\U(1)))$-invariant inner products of the form {\rm (\ref{metrics1})} 
 \end{theorem}
  \begin{proof}
  Let $\ell = m =n = 1$.  In this case we have $\fr{h}_1 = \fr{h}_2 = \fr{h}_3 =  0$, so we do not have $u_1$, $u_2$ and $u_3$ variable.  To find Einstein metrics we need to solve the system
\begin{equation}\label{einstein2}
r_4 - r_5 =0, \  r_5 - r_6=0, \ r_6 - r_7=0, \ r_7 - r_8=0.  
\end{equation}
We set $x_{(7)} = 1$.  Then from $r_0 = 0$ we have that
$
c = -({x_{(8)}}^2-1)/\big(\sqrt{3}({x_{(8)}}^2+1 ) \big).
$
By substituting $c$ into the system (\ref{einstein2}), this reduces to the system 
\begin{eqnarray*}
&&
g_1= 3 {v_4} {x_{(6)}}^2 {x_{(8)}}^4+6 {v_4} {x_{(6)}}^2
   {x_{(8)}}^2+3 {v_4} {x_{(6)}}^2-4 {v_5} {x_{(6)}}^2 {x_{(8)}}^2-4
   {v_5} {x_{(8)}}^4-4 {v_5} {x_{(8)}}^2=0,\\
&&    
g_2=   2 {v_5} {x_{(6)}}^2{x_{(8)}}+4 {v_5} {x_{(8)}}^3+4 {v_5} {x_{(8)}}-{x_{(6)}}^3
   {x_{(8)}}^2-{x_{(6)}}^3+{x_{(6)}} {x_{(8)}}^4 -6 {x_{(6)}} {x_{(8)}}^3\\&&
   +2 {x_{(6)}} {x_{(8)}}^2 -6 {x_{(6)}} {x_{(8)}}+{x_{(6)}}=0,\\
&&    
g_3=3 {v_4}{x_{(6)}}^2 {x_{(8)}}^4+6 {v_4} {x_{(6)}}^2 {x_{(8)}}^2+3 {v_4}
   {x_{(6)}}^2+4 {v_5} {x_{(6)}}^2 {x_{(8)}}^4-4 {v_5} {x_{(8)}}^6  -8
   {v_5} {x_{(8)}}^4
   \\&& 
   -4 {v_5} {x_{(8)}}^2 +4 {x_{(6)}}^3 {x_{(8)}}^5+8
   {x_{(6)}}^3 {x_{(8)}}^3+4 {x_{(6)}}^3 {x_{(8)}}   -12 {x_{(6)}}^2
   {x_{(8)}}^5-24 {x_{(6)}}^2 {x_{(8)}}^3   \\&&  -12 {x_{(6)}}^2 {x_{(8)}}  -4
   {x_{(6)}} {x_{(8)}}^7+12 {x_{(6)}} {x_{(8)}}^6-8 {x_{(6)}} {x_{(8)}}^5+24
   {x_{(6)}} {x_{(8)}}^4-4 {x_{(6)}} {x_{(8)}}^3\\&&
   +12 {x_{(6)}}{x_{(8)}}^2 =0, \\
&&    
g_4=({x_{(8)}}-1) (3 {v_4} {x_{(6)}} {x_{(8)}}^5+3 {v_4} {x_{(6)}} {x_{(8)}}^4+6 {v_4} {x_{(6)}}{x_{(8)}}^3+6 {v_4} {x_{(6)}} {x_{(8)}}^2 +3 {v_4} {x_{(6)}} {x_{(8)}}\\&&
+3 {v_4} {x_{(6)}} -4 {v_5}{x_{(6)}} {x_{(8)}}^3
     -4 {v_5} {x_{(6)}} {x_{(8)}}^2-12 {x_{(6)}} {x_{(8)}}^5 -24 {x_{(6)}} {x_{(8)}}^3-12{x_{(6)}} {x_{(8)}} +4 {x_{(8)}}^6\\&&
     +4 {x_{(8)}}^5+8 {x_{(8)}}^4+8 {x_{(8)}}^3
     +4 {x_{(8)}}^2+4 {x_{(8)}}) = 0.
   \end{eqnarray*} 

First we study the case where $x_{(8)} = 1$. Then the  equations $r_4 - r_5 =0, \  r_5 - r_6=0,  \ r_6 - r_7=0$ reduce to the system
\begin{eqnarray*}
g_1 &=& -3 {v_4} {x_{(6)}}^2+{v_5} {x_{(6)}}^2+2 {v_5}=0,\\
g_2 &=& \left({x_{(6)}}^2+4\right) ({v_5}-{x_{(6)}})=0,\\
g_3 &=& 3 {v_4} {x_{(6)}}^2+{v_5} {x_{(6)}}^2-4 {v_5}+4 {x_{(6)}}^3-12 {x_{(6)}}^2+8 {x_{(6)}}=0.
\end{eqnarray*}
From $g_2 = 0$ we have $v_5 = x_{(6)}$, so we substitute into $g_1, g_3$ and we take
$$
g_1 = -3 {v_4} {x_{(6)}}+{x_{(6)}}^2+2=0,\quad
g_3 = 3 {v_4} {x_{(6)}}+5 {x_{(6)}}^2-12 {x_{(6)}}+4=0.
$$

We solve $g_1$ with respect to $v_4$ and we have
$
v_4 =  ({x_{(6)}}^2+2)/(3 x_{(6)}).
$
We substitute into $g_3$ and we take
$
6x_{(6)}^2 -12x_{(6)} + 6 = 0,
$
whose solution is $x_{(6)} = 1$.  From the above calculations we have that $v_5 = v_4 = 1$.  So the only Einstein metric is the bi-invariant metric which is naturally reductive.

Now we study the case when $x_{(8)} \neq 1$. By solving $ g_2=0, \ g_3=0$, we obtain  
\begin{eqnarray*}
&&{v_4}= -\frac{2 {x_{(8)}}}{3 {x_{(6)}}
   \left({x_{(8)}}^2+1\right) \left({x_{(6)}}^2+2
   {x_{(8)}}^2+2\right)} \big(3 {x_{(6)}}^4
   {x_{(8)}}^2+2 {x_{(6)}}^4-6 {x_{(6)}}^3 {x_{(8)}}^2-6 {x_{(6)}}^3\\ & &+12
   {x_{(6)}}^2 {x_{(8)}}^3+3 {x_{(6)}}^2 {x_{(8)}}^2 +6 {x_{(6)}}^2
   {x_{(8)}}+3 {x_{(6)}}^2-12 {x_{(6)}} {x_{(8)}}^4-24 {x_{(6)}}
   {x_{(8)}}^2-12 {x_{(6)}}-3 {x_{(8)}}^6\\ & &+6 {x_{(8)}}^5-5 {x_{(8)}}^4+12
   {x_{(8)}}^3-{x_{(8)}}^2+6 {x_{(8)}}+1\big), \\
&&  {v_5}= -\frac{{x_{(6)}}
   \left({x_{(8)}}^2+1\right) \left(-{x_{(6)}}^2+{x_{(8)}}^2-6
   {x_{(8)}}+1\right)}{2 {x_{(8)}} \left({x_{(6)}}^2+2
   {x_{(8)}}^2+2\right)}. 
\end{eqnarray*}
By substituting these $ v_4, \ v_5$  into $g_1, g_4$, we can see that the equations $g_1 =0$ and $g_4=0$ reduce to
 the polynomial equations  of $x_{(6)}$ and $x_{(8)}$:  
\begin{eqnarray*}
 G_1(x_{(6)}, x_{(8)})&=& ({x_{(6)}}-{x_{(8)}}) ({x_{(6)}}^3+{x_{(6)}}^2 {x_{(8)}}-2
   {x_{(6)}}^2+{x_{(6)}} {x_{(8)}}^2+2 {x_{(6)}}\\
&&   {x_{(8)}}+{x_{(6)}}+{x_{(8)}}^3-2{x_{(8)}}^2+{x_{(8)}}-4)=0,
\end{eqnarray*}
\begin{eqnarray*}
G_4(x_{(6)}, x_{(8)})&=& -{x_{(6)}}^4 {x_{(8)}}-{x_{(6)}}^4+2
   {x_{(6)}}^3 {x_{(8)}}-4 {x_{(6)}}^2 {x_{(8)}}^2-4 {x_{(6)}}^2 {x_{(8)}}+4
   {x_{(6)}} {x_{(8)}}^3\\
&&  +4 {x_{(6)}}
    {x_{(8)}}  +{x_{(8)}}^5 -{x_{(8)}}^4-{x_{(8)}}+1=0.
\end{eqnarray*}
By taking the resultant ${\rm Res}_{x_{(8)}}(G_1, G_4)$ of $G_1$ and $G_4$ with respect to $ x_{(8)}$,  we obtain 
\begin{eqnarray*}
{\rm Res}_{x_{(8)}}(G_1, G_4) = -32 (x_{(6)}-1)^4 ({ x_{(6)}}^2+4)^2 ( {x_{(6)}}^2+x_{(6)} + 1) ( 4 {x_{(6)}}^2+1).
\end{eqnarray*}
Thus  the polynomial  ${\rm Res}_{x_{(8)}}(G_1, G_4)$ of $x_{(6)}$ does not have real roots, for $x_{(6)}\ne 1$. 
 If $x_{(6)}=1$ then Proposition \ref{reductive} (1, (ii)) implies that the metric is naturally reductive and this concludes the proof.
  \end{proof}

   %%page 34
 \section{Invariant Einstein metrics on certain Stiefel manifolds $V_{\ell+m}\mathbb{C}^{\ell+m+n}$}
  A complete description for the set of all $\SU(\ell+m+n)$-invariant metrics on the Stiefel manifolds 
  $V_{\ell+m}\mathbb{C}^{\ell+m+n}\cong \SU(\ell+m+n)/\SU(n)$ is not easy.
  This is  because the isotropy representation $\chi$
  % : \U(n)\to \Aut(T_{o}(\U(\ell+m+n)/\U(n)))$ 
  of $\U(\ell+m+n)/\U(n)\cong\SU(\ell+m+n)/\SU(n)$ contains some equivalent subrepresentations. In fact, it is given as 
$$
\chi = \underbrace{1\oplus\cdots\oplus 1}_{(\ell+m)^{2}\mbox{-times}}\oplus\underbrace{\left( \right.(\mu_{n}\oplus\bar{\mu}_{n})\oplus\cdots\oplus(\mu_{n}\oplus\bar{\mu}_{n})\left. \right)}_{(\ell+m)\mbox{-times}},
$$   
where $\mu_{n} : \U(n)\to \Aut(\bb{C}^{n})$ is the standard representation of $\U(n)$ and $\Ad^{\U(n)}\otimes\bb{C} = \mu_{n}\otimes\bar{\mu}_{n}$ is its  complexified adjoint representation.  
In this section we search for $\Ad(\s(\U(\ell)\times\U(m)\times\U(n)))$-invariant Einstein metrics of the form (\ref{metrics2}), which correspond to a subset of all $\SU(\ell+m+n)$-invariant metrics on $V_{\ell+m}\mathbb{C}^{\ell+m+n}$.

The metric (\ref{metrics2}) is Einstein if and only if the system 
\begin{eqnarray}\label{sist3}
&& r_{0} = 0, \, r_{1} - r_2 = 0,\, r_2 - r_4 = 0,\, r_4 - r_{5} = 0,\nonumber\\
&& r_{5} - r_{6} = 0, \, r_{6} - r_{7} = 0, \ r_{7} - r_{8} = 0
\end{eqnarray}
has positive solutions (cf. Propositions \ref{ricci1}, \ref{ricci31}).  
As for the case of the special unitary group, in the above system we may assume that 
 $a=d=1$, $b=0$.

\subsection{The Stiefel manifold $V_2\bb{C}^{4}\cong \SU(4)/\SU(2)$}
In this case we have $\ell =m = 1$, $n = 2$ and in system (\ref{sist3}) the second and third equations are absent.
From the equation $r_0=0$ we obtain that
$$
c=\frac{1-{x_{(8)}}^2}{\sqrt{2}(1+{x_{(8)}}^2)}.
$$
Next, we observe that in the  equations $r_4 - r_5=0, r_5 - r_6=0$ the variables
$v_4, v_5$ are {\it linear expressions} of $x_{(6)}, x_{(8)}$.
We substitute $v_4, v_5$ and the above value of $c$ in the equations $r_6-r_7=0, r_7-r_8=0$  and we obtain the solutions $x_{(8)}=0, -1$ (both rejected) and $x_{(8)}=1$, which implies that $c=0$.
We set $x_{(7)}=1$ in the last four equations of system (\ref{sist3}) and this reduces to the system
\begin{eqnarray*}
&& 2 {v_4} {x_{(6)}}^2-{v_5} {x_{(6)}}^2-{v_5} = 0, \quad \left({x_{(6)}}^2+2\right) ({v_5}-{x_{(6)}})=0,\\
&& 2 {v_4} {x_{(6)}}^2+{v_5} {x_{(6)}}^2-4 {v_5}+6 {x_{(6)}}^3-16 {x_{(6)}}^2+8{x_{(6)}} = 0.
\end{eqnarray*}
%%%%page35%%%%
From the second equation above we have $v_5 = x_{(6)}$. We substitute into the first and third equations above and we obtain two polynomials of $x_{(6)}$ and $v_4$:
$$
g_1(x_{(6)}, u_4) = 2 {v_4} {x_{(6)}}^2-{x_{(6)}}^3-{x_{(6)}}, \quad g_2(x_{(6)}, v_4) = 2 {v_4} {x_{(6)}}^2+7 {x_{(6)}}^3-16 {x_{(6)}}^2+4 {x_{(6)}}.
$$
Thus we have  a polynomial of $x_{(6)}$ given by
%We consider the polynomial ring $R = Q[z, x_{(6)}, v_4]$ and the ideal $I$ generated by the polynomials $\{g_1, g_2, zx_{(6)}v_4-1\}$.  We take a lexicographic ordering $>$ with $z > v_4 > x_{(6)}$ for a monomial ordering on $R$. Then, by the aid of computer, we see that a Gr\"obner basis for the ideal $I$ contains a polynomial of $x_{(6)}$ given by
$
8 {x_{(6)}}^2-16 {x_{(6)}}+5 = 0
$
whose solutions  are
$
x_{(6)} = (4 \pm \sqrt{6})/4
$.
We substitute  these values into $g_1 = 0$ and we see that, 
for  $x_{(6)} =  (4 -\sqrt{6})/4, v_4 =  (52+3\sqrt{6})/4 $ and, 
 for  $ x_{(6)} = (4 + \sqrt{6})/4,  v_4 =  (52 - 3\sqrt{6})/4.$

Therefore, we obtain two Einstein metrics for $V_2\bb{C}^{4}$ which are of Jensen's type:
\begin{itemize}
\item[(1)] $(v_4, v_5, x_{(6)}, x_{(7)}, x_{(8)}) =   
( (52+3\sqrt{6})/4, \ (4 -\sqrt{6})/4, \ (4 -\sqrt{6})/4, \ 1,\ 1)$
\item[(2)] $(v_4, v_5, x_{(6)}, x_{(7)}, x_{(8)}) = 
 ( (52 - 3\sqrt{6})/4, \ (4 + \sqrt{6})/4, \ (4 + \sqrt{6})/4, \ 1,\ 1)$.
\end{itemize}

\subsection{The Stiefel manifold $V_3\bb{C}^{5}\cong \SU(5)/\SU(2)$} %%\page 35%
In this case we have $\ell =1$, $m =n = 2$ and in system (\ref{sist3}) the second equation is absent.
To find Einstein metrics we solve the system
$$
r_2 - r_4 = 0, \ r_4 - r_5 = 0, \ r_5- r_6 = 0, \ r_6- r_7 = 0, \ r_7 - r_8 = 0.
$$
We set $x_{(7)}= 1$ and  this reduces to the system:
{ \begin{eqnarray*}
&& f_1 = 6 {u_2}^2 {x_{(6)}}^2+3 {u_2}^2 {x_{(8)}}^2-5 {u_2} {v_4} {x_{(6)}}^2{x_{(8)}}^2-10 {u_2} {v_4} {x_{(6)}}^2+6 {x_{(6)}}^2 {x_{(8)}}^2=0,\\
&& f_2 = 5 {v_4} {x_{(6)}}^2 {x_{(8)}}^4+20 {v_4} {x_{(6)}}^2 {x_{(8)}}^2+20 {u_4} {x_{(6)}}^2-18 {v_5} {x_{(6)}}^2 {x_{(8)}}^2-9 {v_5} {x_{(8)}}^4\\&&
-18 {v_5} {x_{(8)}}^2 =0, \\
&& f_3 = 3 {u_2} {x_{(8)}}^3+6 {u_2} {x_{(8)}}+12 {v_5} {x_{(6)}}^2 {x_{(8)}}+9{u5} {x_{(8)}}^3+18 {v_5} {x_{(8)}}-4 {x_{(6)}}^3 {x_{(8)}}^2-8 {x_{(6)}}^3\\
&& +4 {x_{(6)}} {x_{(8)}}^4-20 {x_{(6)}} {x_{(8)}}^3+12 {x_{(6)}} {x_{(8)}}^2-40 {x_{(6)}} {x_{(8)}}+8 {x_{(6)}} = 0, \\
&& f_4 = 9 {u_2} {x_{(6)}}^2 {x_{(8)}}^4+36 {u_2} {x_{(6)}}^2 {x_{(8)}}^2+36 {u_2}{x_{(6)}}^2-9 {u_2} {x_{(8)}}^6-36 {u_2} {x_{(8)}}^4-36 {u_2}{x_{(8)}}^2 \\&&
+5 {u_4} {x_{(6)}}^2 {x_{(8)}}^4 +20 {v_4} {x_{(6)}}^2{x_{(8)}}^2+20 {v_4} {x_{(6)}}^2+9 {v_5} {x_{(6)}}^2 {x_{(8)}}^4-9 {v_5}{x_{(8)}}^6-36 {v_5} {x_{(8)}}^4\\&& 
-36 {v_5} {x_{(8)}}^2+18 {x_{(6)}}^3{x_{(8)}}^5 +72 {x_{(6)}}^3 {x_{(8)}}^3+72 {x_{(6)}}^3 {x_{(8)}}-60 {x_{(6)}}^2{x_{(8)}}^5-240 {x_{(6)}}^2 {x_{(8)}}^3 \\&&
-240 {x_{(6)}}^2 {x_{(8)}}-18 {x_{(6)}}{x_{(8)}}^7+60 {x_{(6)}} {x_{(8)}}^6 -78 {x_{(6)}} {x_{(8)}}^5+240 {x_{(6)}}{x_{(8)}}^4-96 {x_{(6)}} {x_{(8)}}^3\\&&
+240 {x_{(6)}} {x_{(8)}}^2-24 {x_{(6)}} {x_{(8)}} = 0, \\
&& f_5 = -9 {u_2} {x_{(6)}} {x_{(8)}}^4-36 {u_2} {x_{(6)}} {x_{(8)}}^2-36 {u_2}{x_{(6)}}+5 {v_4} {x_{(6)}} {x_{(8)}}^6+15 {v_4} {x_{(6)}} {x_{(8)}}^4 \\&&
-20{u_4} {x_{(6)}} -9 {u5} {x_{(6)}} {x_{(8)}}^4 +36 {v_5} {x_{(6)}}{x_{(8)}}^2+6 {x_{(6)}}^2 {x_{(8)}}^5+24 {x_{(6)}}^2 {x_{(8)}}^3+24 {x_{(6)}}^2{x_{(8)}} \\&&
-60 {x_{(6)}} {x_{(8)}}^6+60 {x_{(6)}} {x_{(8)}}^5 -240 {x_{(6)}}{x_{(8)}}^4+240 {x_{(6)}} {x_{(8)}}^3 -240 {x_{(6)}} {x_{(8)}}^2+240 {x_{(6)}}{x_{(8)}}\\ &&
+18 {x_{(8)}}^7+54 {x_{(8)}}^5-72 {x_{(8)}} = 0.
\end{eqnarray*} }
Also, from $r_0 = 0$ we obtain that
$$
c = -\frac{2 \left({x_{(8)}}^2-1\right)}{\sqrt{5} \left({x_{(8)}}^2+2\right)}.
$$

We observe that the equations $f_3, f_4$ and $f_5$ {\it are linear} with respect to $u_2, u_4$ and $v_5$ and by
solving  we obtain  the following:
{\small \begin{eqnarray*}
&& u_2 =1/\left({3 {x_{(8)}} (-2 {x_{(6)}}^4+{x_{(6)}}^2 {x_{(8)}}^2-6{x_{(6)}}^2+{x_{(8)}}^4+{x_{(8)}}^2-2)}\right)\times \big( {x_{(6)}} (14 {x_{(6)}}^4 {x_{(8)}}^2
 \\&&
 -20 {x_{(6)}}^4-40 {x_{(6)}}^3 {x_{(8)}}^2 +40{x_{(6)}}^3 {x_{(8)}}-7 {x_{(6)}}^2 {x_{(8)}}^4+50 {x_{(6)}}^2 {x_{(8)}}^3+2{x_{(6)}}^2 {x_{(8)}}^2\\&&
 -60 {x_{(6)}}^2 {x_{(8)}} -30 {x_{(6)}} {x_{(8)}}^4+30 {x_{(6)}}{x_{(8)}}^3 -60 {x_{(6)}} {x_{(8)}}^2+60 {x_{(6)}} {x_{(8)}}-7 {x_{(8)}}^6 \\&&
 +20{x_{(8)}}^5-17 {x_{(8)}}^4 +20 {x_{(8)}}^3+4 {x_{(8)}}^2-40 {x_{(8)}}+20)\big) \nonumber \\
%-----------------
&& v_4 =-1/\big({5 {x_{(6)}} {x_{(8)}} ({x_{(8)}}^2+2 ) (2 {x_{(6)}}^4-{x_{(6)}}^2{x_{(8)}}^2+6 {x_{(6)}}^2-{x_{(8)}}^4-{x_{(8)}}^2+2)}\big)\times \\&&
 \big( 6  (8 {x_{(6)}}^6 {x_{(8)}}^2+20 {x_{(6)}}^6 -20 {x_{(6)}}^5 {x_{(8)}}^3-40{x_{(6)}}^5 {x_{(8)}}+2 {x_{(6)}}^4 {x_{(8)}}^4+20 {x_{(6)}}^4 {x_{(8)}}^3\\&&
 +6{x_{(6)}}^4 {x_{(8)}}^2+60 {x_{(6)}}^4 {x_{(8)}} +10 {x_{(6)}}^3 {x_{(8)}}^5 -20{x_{(6)}}^3 {x_{(8)}}^4-20 {x_{(6)}}^3 {x_{(8)}}^3-60 {x_{(6)}}^3 {x_{(8)}}^2\\&&
 -60{x_{(6)}}^3 {x_{(8)}}-7 {x_{(6)}}^2 {x_{(8)}}^6+10 {x_{(6)}}^2 {x_{(8)}}^5-12{x_{(6)}}^2 {x_{(8)}}^4+40 {x_{(6)}}^2 {x_{(8)}}^3-6 {x_{(6)}}^2 {x_{(8)}}^2 \\&&
 +40{x_{(6)}}^2 {x_{(8)}}-20 {x_{(6)}}^2+10 {x_{(6)}} {x_{(8)}}^7-10 {x_{(6)}}{x_{(8)}}^6+40 {x_{(6)}} {x_{(8)}}^5 -40 {x_{(6)}} {x_{(8)}}^4\\&&
 +40 {x_{(6)}}{x_{(8)}}^3-40 {x_{(6)}} {x_{(8)}}^2-3 {x_{(8)}}^8-9 {x_{(8)}}^6+12 {x_{(8)}}^2)\big)\nonumber\\
%-------------------
%\end{eqnarray*} }
%{\small \begin{eqnarray*}
 && v_5 =  -1/({3 {x_{(8)}} \left(-2 {x_{(6)}}^4+{x_{(6)}}^2 {x_{(8)}}^2-6{x_{(6)}}^2+{x_{(8)}}^4+{x_{(8)}}^2-2\right)}) \times\\&&
  \big(2 {x_{(6)}}^5 {x_{(8)}}^2+4 {x_{(6)}}^5-{x_{(6)}}^3 {x_{(8)}}^4 +10 {x_{(6)}}^3{x_{(8)}}^3-6 {x_{(6)}}^3 {x_{(8)}}^2+20 {x_{(6)}}^3 {x_{(8)}}-8 {x_{(6)}}^3 \\&&
  -10   {x_{(6)}}^2 {x_{(8)}}^4 +10 {x_{(6)}}^2 {x_{(8)}}^3-20 {x_{(6)}}^2 {x_{(8)}}^2+20{x_{(6)}}^2 {x_{(8)}}-{x_{(6)}} {x_{(8)}}^6-3 {x_{(6)}} {x_{(8)}}^4+4 {x_{(6)}}\big).
\end{eqnarray*} }
We substitute the above expressions into  $f_1=0$, $f_2=0$ and we obtain 
$$
-{x_{(6)}}^2 g_1(x_{(6)}, x_{(8)}) = 0 \ \ \ \mbox{and} \ \ \ -3x_{(6)}({x_{(8)}}^2 + 2) g_2(x_{(6)}, x_{(8)})=0,
$$
where $g_1$ and $g_2$ are given as follows:
{\small  \begin{eqnarray*}
&&g_1(x_{(6)}, x_{(8)}) = 280 {x_{(6)}}^{10} {x_{(8)}}^4+1840 {x_{(6)}}^{10} {x_{(8)}}^2-3200 {x_{(6)}}^{10}-1680{x_{(6)}}^9 {x_{(8)}}^5+320 {x_{(6)}}^9 {x_{(8)}}^4 \\
&& -1280 {x_{(6)}}^9 {x_{(8)}}^3-8000{x_{(6)}}^9 {x_{(8)}}^2+12800 {x_{(6)}}^9 {x_{(8)}}+28 {x_{(6)}}^8 {x_{(8)}}^6+6080{x_{(6)}}^8 {x_{(8)}}^5\\&&
-8664 {x_{(6)}}^8 {x_{(8)}}^4 +29120 {x_{(6)}}^8{x_{(8)}}^3 -13520 {x_{(6)}}^8 {x_{(8)}}^2-19200 {x_{(6)}}^8 {x_{(8)}}+1680{x_{(6)}}^7 {x_{(8)}}^7\\&&
-7920 {x_{(6)}}^7 {x_{(8)}}^6+2000 {x_{(6)}}^7 {x_{(8)}}^5 -18880 {x_{(6)}}^7 {x_{(8)}}^4-18880 {x_{(6)}}^7 {x_{(8)}}^3 +33600{x_{(6)}}^7 {x_{(8)}}^2 \\&&
+19200 {x_{(6)}}^7 {x_{(8)}}-518 {x_{(6)}}^6 {x_{(8)}}^8+1640{x_{(6)}}^6 {x_{(8)}}^7 -3668 {x_{(6)}}^6 {x_{(8)}}^6+33600 {x_{(6)}}^6{x_{(8)}}^5 \\&&
-784 {x_{(6)}}^6 {x_{(8)}}^4 +33120 {x_{(6)}}^6 {x_{(8)}}^3 -67520{x_{(6)}}^6 {x_{(8)}}^2-12800 {x_{(6)}}^6 {x_{(8)}} +6400 {x_{(6)}}^6\\&&
+1260 {x_{(6)}}^5{x_{(8)}}^9+240 {x_{(6)}}^5 {x_{(8)}}^8+8400 {x_{(6)}}^5 {x_{(8)}}^7 -26320{x_{(6)}}^5 {x_{(8)}}^6 -28320 {x_{(6)}}^5 {x_{(8)}}^5 \\&&
+12320 {x_{(6)}}^5{x_{(8)}}^4-37280 {x_{(6)}}^5 {x_{(8)}}^3+96000 {x_{(6)}}^5 {x_{(8)}}^2-12800{x_{(6)}}^5 {x_{(8)}}-91 {x_{(6)}}^4 {x_{(8)}}^{10} \\&&
-5220 {x_{(6)}}^4 {x_{(8)}}^9  +4326{x_{(6)}}^4 {x_{(8)}}^8-12120 {x_{(6)}}^4 {x_{(8)}}^7+46164 {x_{(6)}}^4{x_{(8)}}^6-3280 {x_{(6)}}^4 {x_{(8)}}^5 \\&&
+14296 {x_{(6)}}^4 {x_{(8)}}^4-480{x_{(6)}}^4 {x_{(8)}}^3 -67520 {x_{(6)}}^4 {x_{(8)}}^2 +19200 {x_{(6)}}^4 {x_{(8)}}-840{x_{(6)}}^3 {x_{(8)}}^{11}\\&&
+5620 {x_{(6)}}^3 {x_{(8)}}^{10}-2100 {x_{(6)}}^3{x_{(8)}}^9+9160 {x_{(6)}}^3 {x_{(8)}}^8 -15040 {x_{(6)}}^3 {x_{(8)}}^7 -5600{x_{(6)}}^3 {x_{(8)}}^6 \\&&
-27040 {x_{(6)}}^3 {x_{(8)}}^5+18720 {x_{(6)}}^3{x_{(8)}}^4-6880 {x_{(6)}}^3 {x_{(8)}}^3+43200 {x_{(6)}}^3 {x_{(8)}}^2 -19200{x_{(6)}}^3 {x_{(8)}}\\&&
+224 {x_{(6)}}^2 {x_{(8)}}^{12} -2420 {x_{(6)}}^2{x_{(8)}}^{11} +1938 {x_{(6)}}^2 {x_{(8)}}^{10}-14420 {x_{(6)}}^2 {x_{(8)}}^9+22938{x_{(6)}}^2 {x_{(8)}}^8 \\
%\end{eqnarray*}
%\begin{eqnarray*}
&& -31120 {x_{(6)}}^2 {x_{(8)}}^7+53204 {x_{(6)}}^2{x_{(8)}}^6-33360 {x_{(6)}}^2 {x_{(8)}}^5 +20216 {x_{(6)}}^2 {x_{(8)}}^4-12480{x_{(6)}}^2 {x_{(8)}}^3 \\&&
-14320 {x_{(6)}}^2 {x_{(8)}}^2 +12800 {x_{(6)}}^2{x_{(8)}}-3200 {x_{(6)}}^2-420 {x_{(6)}} {x_{(8)}}^{13}+1740 {x_{(6)}}{x_{(8)}}^{12} \\
%\end{eqnarray*}
%\begin{eqnarray*}
&&
 -2820 {x_{(6)}} {x_{(8)}}^{11} +8340 {x_{(6)}} {x_{(8)}}^{10}  -8760{x_{(6)}} {x_{(8)}}^9 +10560 {x_{(6)}} {x_{(8)}}^8-10560 {x_{(6)}} {x_{(8)}}^7 \\
%\end{eqnarray*}
%\begin{eqnarray*} 
 &&
 -3360{x_{(6)}} {x_{(8)}}^6+6240 {x_{(6)}} {x_{(8)}}^5 -12480 {x_{(6)}} {x_{(8)}}^4 +16320{x_{(6)}} {x_{(8)}}^3 -4800 {x_{(6)}} {x_{(8)}}^2  
 \end{eqnarray*}
\begin{eqnarray*} 
 &&
 +77 {x_{(8)}}^{14}-80{x_{(8)}}^{13}+28 {x_{(8)}}^{12}-480 {x_{(8)}}^{11}-223 {x_{(8)}}^{10}-400{x_{(8)}}^9 +590 {x_{(8)}}^8 \\&&+1280 {x_{(8)}}^7 +32 {x_{(8)}}^6+960 {x_{(8)}}^5 -1544{x_{(8)}}^4-1280 {x_{(8)}}^3+1040 {x_{(8)}}^2
\end{eqnarray*} }
and
{\small \begin{eqnarray*}
&& g_2(x_{(6)}, x_{(8)}) = -3 {x_{(6)}} ({x_{(8)}}^2+2)^2 (-20 {x_{(6)}}^6+40 {x_{(6)}}^5{x_{(8)}}-4 {x_{(6)}}^4 {x_{(8)}}^2-60 {x_{(6)}}^4 {x_{(8)}}  \\&&
-20 {x_{(6)}}^3{x_{(8)}}^3 +60 {x_{(6)}}^3 {x_{(8)}}^2+60 {x_{(6)}}^3 {x_{(8)}}+17 {x_{(6)}}^2{x_{(8)}}^4-30 {x_{(6)}}^2 {x_{(8)}}^3-2 {x_{(6)}}^2 {x_{(8)}}^2 \\&&
-40 {x_{(6)}}^2{x_{(8)}}+20 {x_{(6)}}^2 -20 {x_{(6)}} {x_{(8)}}^5+30 {x_{(6)}} {x_{(8)}}^4-50{x_{(6)}} {x_{(8)}}^3+40 {x_{(6)}} {x_{(8)}}^2+7 {x_{(8)}}^6 \\&& 
+7 {x_{(8)}}^4-14{x_{(8)}}^2).
\end{eqnarray*} }

We consider a polynomial ring $R = Q[z, x_{(6)}, x_{(8)}]$ and an ideal $I$ generated by $\{g_1, g_2,$ $ z \, x_{(6)} x_{(8)} - 1\}$ to find non zero solutions for the equations $g_1 = 0,\ g_2 = 0$.  We take a lexicographic order $>$ with $z > x_{(6)} > x_{(8)}$ for a monomial ordering on $R$.  Then by the aid of computer we see that a Gr\"obner basis for the ideal $I$ contains the polynomials
$$
(x_{(8)}- 1)^2 (5 {x_{(8)}}^2 + 2)^2 h(x_{(8)}),
$$
where   $h(x_{(8)})$ is given by
{\footnotesize \begin{eqnarray*}
&& h(x_{(8)}) = 6525496468915200 {x_{(8)}}^{26}-54695097654220800 {x_{(8)}}^{25}+281982445913589120{x_{(8)}}^{24}\\
&& -1138653468769528320 {x_{(8)}}^{23}+3741072893659661028{x_{(8)}}^{22} -10461844097208857304 {x_{(8)}}^{21}\\
&& +25787611597344221492{x_{(8)}}^{20} -56712788589395019672 {x_{(8)}}^{19}+112460183799287026516{x_{(8)}}^{18}\\
&& -203312692454920821816 {x_{(8)}}^{17}+336020119285435562648{x_{(8)}}^{16} -507335167986987659588 {x_{(8)}}^{15}\\
&& +701457087993906599199{x_{(8)}}^{14} -889337244622259846134 {x_{(8)}}^{13}+1029897386253507589663{x_{(8)}}^{12}\\
&& -1083070284962388243040 {x_{(8)}}^{11}+1029965936819584117520{x_{(8)}}^{10} -882550965600444083800 {x_{(8)}}^9\\
&& +678539858188720076700{x_{(8)}}^8 -466171617462687452000 {x_{(8)}}^7+284744896720187390000{x_{(8)}}^6\\
&& -153012868758975000000 {x_{(8)}}^5+71039304710981500000{x_{(8)}}^4-27757371715690000000 {x_{(8)}}^3\\
&& +8664161487275000000{x_{(8)}}^2-1906087342250000000 {x_{(8)}}+216106141875000000,
\end{eqnarray*} }
and 
\begin{equation}\label{astron}
(x_{(8)} - 1)^2(x_{(6)} - w(x_{(8)})) = 0,
\end{equation}
where $w(x_{(8)})$ is a polynomial with rational coefficients.  We solve the equation $h(x_{(8)}) = 0$ numerically  and we obtain two positive solutions, which are given approximately as
$$
x_{(8)} \approx 0.973092, \quad x_{(8)} \approx 1.45884.
$$
By substituting the values of $x_{(8)}$ into (\ref{astron}) we obtain two positive solutions of the system of equations
${g_1 = 0, g_2 = 0}$, approximately as
\begin{itemize}
\item[] $(x_{(6)},\ x_{(8)}) \approx (0.476191,\ 0.973092)$
\item[] $(x_{(6)},\ x_{(8)}) \approx (1.965348,\ 1.45884) $.
\end{itemize}
We substitute these values into the expressiosn of $u_2, u_4$ and $u_5$ and obtain two Einstein metrics on $V_3\bb{C}^{5}\cong \SU(5)/\SU(2)$ which are given as follows:
\begin{itemize}
\item[(1)] $(u_2,\ v_4,\ v_5,\ x_{(6)},\ x_{(8)},\ x_{(7)})\approx (0.390148,\ 1.47889,\ 0.50248,\ 0.476191,\ 0.973092,\ 1)$,
\item[(2)] $(u_2,\ v_4,\ v_5,\ x_{(6)},\ x_{(8)},\ x_{(7)})\approx (0.499212,\ 1.42431,\ 2.06481,\ 1.965348,\ 1.45884,\ 1)$.
\end{itemize}

In the case where $x_{(8)} = 1$ then $c = 0$ and from $g_2 = 0$ we obtain that
$$
9 \left(-20 {x_{(6)}}^6+40 {x_{(6)}}^5-64{x_{(6)}}^4+100{x_{(6)}}^3-35
   {x_{(6)}}^2\right) = 0.
$$
The solutions of the above equation are 
$x_{(6)} =(10 - \sqrt{30})/10$, $x_{(6)} =  (10 + \sqrt{30})/10$.
We substitute into  $u_2, v_4$ and $v_5$ and we obtain two more Einstein metrics of Jensen's type as follows:
\begin{itemize}
\item[(3)] 
$(u_2, v_4,  v_5,  x_{(6)},  x_{(7)},  x_{(8)}) \\ = (1-\sqrt{3/10}, \ 4(55+2\sqrt{30})/175,\  1-\sqrt{3/10},\ (10 - \sqrt{30})/10,\ 1,\ 1),
$
\item[(4)] 
$(u_2, v_4,  v_5,  x_{(6)},  x_{(7)},  x_{(8)}) \\ = (1+\sqrt{3/10}, \ 4(55 - 2\sqrt{30})/175,\  1+\sqrt{3/10},\ (10 + \sqrt{30})/10,\ 1,\ 1)
$.
\end{itemize}

%%%%page 38%%%
\subsection{The Stiefel manifold $V_4\bb{C}^{6}\cong \SU(6)/\SU(2)$}
In this case we have $\ell = m = n = 2$.  To find Einstein metrics we solve the system
$$
r_1 - r_2 = 0, \ r_2 - r_4 = 0, \ r_4 - r_5 = 0, \ r_5- r_6 = 0, \ r_6- r_7 = 0, \ r_7 - r_8 = 0.
$$
We set $x_{(7)} = 1$ and the above system reduces to
{ \begin{eqnarray*}
&& f_1 = {u_1}^2 {u_2} {x_{(6)}}^2 {x_{(8)}}^2+{u_1}^2 {u_2}{x_{(8)}}^2-{u_1} {u_2}^2 {x_{(6)}}^2-{u_1} {u_2}^2{x_{(8)}}^2-{u_1} {x_{(6)}}^2 {x_{(8)}}^2\\ &&+{u_2} {x_{(6)}}^2 {x_{(8)}}^2=0,\\
&&f_2 = 2 {u_2}^2 {x_{(6)}}^2+2 {u_2}^2 {x_{(8)}}^2-3 {u_2} {v_4} {x_{(6)}}^2{x_{(8)}}^2-3 {u_2} {v_4} {x_{(6)}}^2+2 {x_{(6)}}^2 {x_{(8)}}^2 = 0,\\
&&f_3 = 3 {v_4} {x_{(6)}}^2 {x_{(8)}}^4+6 {v_4} {x_{(6)}}^2 {x_{(8)}}^2+3{v_4} {x_{(6)}}^2-4 {v_5} {x_{(6)}}^2 {x_{(8)}}^2-4 {v_5} {x_{(8)}}^4-4{v_5} {x_{(8)}}^2= 0,\\
&&f_4 = 3 {u_1} {x_{(8)}}^3+3 {u_1} {x_{(8)}} +3 {u_2}{x_{(8)}}^3+3 {u_2} {x_{(8)}} +8 {v_5} {x_{(6)}}^2 {x_{(8)}}+10 {u_5}{x_{(8)}}^3+10 {v_5} {x_{(8)}} \\&& -4 {x_{(6)}}^3 {x_{(8)}}^2
 -4 {x_{(6)}}^3+4 {x_{(6)}}{x_{(8)}}^4-24 {x_{(6)}} {x_{(8)}}^3  +8 {x_{(6)}} {x_{(8)}}^2-24 {x_{(6)}} {x_{(8)}}+4{x_{(6)}} = 0,\\
&& f_5 = -6 {u_1} {x_{(8)}}^6-12 {u_1} {x_{(8)}}^4-6 {u_1} {x_{(8)}}^2+6{u_2} {x_{(6)}}^2 {x_{(8)}}^4+12 {u_2} {x_{(6)}}^2 {x_{(8)}}^2+6 {u_2}{x_{(6)}}^2 \\ &&-6 {u_2} {x_{(8)}}^6 -12 {u_2} {x_{(8)}}^4 -6 {u_2} {x_{(8)}}^2+3{v_4} {x_{(6)}}^2 {x_{(8)}}^4+6 {v_4} {x_{(6)}}^2 {x_{(8)}}^2+3 {v_4}{x_{(6)}}^2\\&&
+4 {v_5} {x_{(6)}}^2 {x_{(8)}}^4 -4 {v_5} {x_{(8)}}^6 -8 {v_5}{x_{(8)}}^4 -4 {v_5} {x_{(8)}}^2+16 {x_{(6)}}^3 {x_{(8)}}^5+32 {x_{(6)}}^3{x_{(8)}}^3\\&&
+16 {x_{(6)}}^3 {x_{(8)}}-48 {x_{(6)}}^2 {x_{(8)}}^5-96 {x_{(6)}}^2{x_{(8)}}^3 -48 {x_{(6)}}^2 {x_{(8)}} -16 {x_{(6)}} {x_{(8)}}^7+48 {x_{(6)}}{x_{(8)}}^6\\&&
-32 {x_{(6)}} {x_{(8)}}^5+96 {x_{(6)}} {x_{(8)}}^4-16 {x_{(6)}}{x_{(8)}}^3+48 {x_{(6)}} {x_{(8)}}^2 = 0,\\
&& f_6 = 6 {u_1} {x_{(6)}} {x_{(8)}}^6+12 {u_1} {x_{(6)}} {x_{(8)}}^4+6 {u_1}{x_{(6)}} {x_{(8)}}^2-6 {u_2} {x_{(6)}} {x_{(8)}}^4-12 {u_2} {x_{(6)}}{x_{(8)}}^2\\&&
-6 {u_2} {x_{(6)}} +3 {v_4} {x_{(6)}} {x_{(8)}}^6+3 {v_4}{x_{(6)}} {x_{(8)}}^4-3 {v_4} {x_{(6)}} {x_{(8)}}^2-3 {v_4} {x_{(6)}}-4{v_5} {x_{(6)}} {x_{(8)}}^4\\&&
+4 {v_5} {x_{(6)}} {x_{(8)}}^2-48 {x_{(6)}}{x_{(8)}}^6 +48 {x_{(6)}} {x_{(8)}}^5-96 {x_{(6)}} {x_{(8)}}^4+96 {x_{(6)}}{x_{(8)}}^3-48 {x_{(6)}} {x_{(8)}}^2 \\&&
+48 {x_{(6)}} {x_{(8)}}+16 {x_{(8)}}^7+16 {x_{(8)}}^5 -16 {x_{(8)}}^3-16 {x_{(8)}} = 0. 
\end{eqnarray*} }
Also, from $r_0 = 0$ we obtain that
$$
c = -\frac{{x_{(8)}}^2-1}{\sqrt{3} \left({x_{(8)}}^2+1\right)}.
$$
%%%%page 38

We observe that the polynomials  $f_3, f_4, f_5$ and $f_6$ {\it are linear} with respect to $u_1, u_2, v_4$ and $v_5$ and by solving we obtain the following:
{\small \begin{eqnarray*}
&& u_1 = - 2 ({x_{(6)}}-{x_{(8)}}) \big(3 {x_{(6)}}^5-5 {x_{(6)}}^4 {x_{(8)}}-2 {x_{(6)}}^3{x_{(8)}}^2+10 {x_{(6)}}^3 {x_{(8)}}-2 {x_{(6)}}^2 {x_{(8)}}^3 \\&&
-10 {x_{(6)}}^2{x_{(8)}}    -5 {x_{(6)}} {x_{(8)}}^4+10 {x_{(6)}} {x_{(8)}}^3-10 {x_{(6)}} {x_{(8)}}^2+8{x_{(6)}} {x_{(8)}}-3 {x_{(6)}}+3 {x_{(8)}}^5\\&&
-3 {x_{(8)}}\big)/  \big({x_{(6)}} {x_{(8)}} \left(2 {x_{(6)}}^4-2 {x_{(8)}}^4-5 {x_{(8)}}^2-2\right)\big),\nonumber\\ 
&& u_2 = - 2 ({x_{(6)}}-1) {x_{(8)}} (3 {x_{(6)}}^5-5 {x_{(6)}}^4+10 {x_{(6)}}^3 {x_{(8)}}-2{x_{(6)}}^3-10 {x_{(6)}}^2 {x_{(8)}}^2-2 {x_{(6)}}^2 \nonumber\\&&
-3 {x_{(6)}} {x_{(8)}}^4 +8{x_{(6)}} {x_{(8)}}^3 -10 {x_{(6)}} {x_{(8)}}^2+10 {x_{(6)}} {x_{(8)}}-5 {x_{(6)}}-3{x_{(8)}}^4+3)/
\\
& & 
\big({x_{(6)}} \left(2{x_{(6)}}^4-2 {x_{(8)}}^4-5 {x_{(8)}}^2-2\right) \big), \nonumber\\ 
&& v_4 =  4 {x_{(8)}} \left({x_{(6)}}^2+{x_{(8)}}^2+1\right) \big({x_{(6)}}^4+6 {x_{(6)}}^2{x_{(8)}}-6 {x_{(6)}} {x_{(8)}}^2-6 {x_{(6)}} {x_{(8)}}-{x_{(8)}}^4 \\&& 
+2{x_{(8)}}^2-1\big)/ \big(3 {x_{(6)}} \left({x_{(8)}}^2+1\right) \left(2 {x_{(6)}}^4-2 {x_{(8)}}^4-5 {x_{(8)}}^2-2\right)\big), \nonumber\\ 
&& v_5 =  {x_{(6)}} \left({x_{(8)}}^2+1\right) \big({x_{(6)}}^4+6 {x_{(6)}}^2 {x_{(8)}}-6{x_{(6)}} {x_{(8)}}^2-6 {x_{(6)}} {x_{(8)}}-{x_{(8)}}^4 
+2 {x_{(8)}}^2-1\big)/ \\ & & \big({x_{(8)}} \left(2 {x_{(6)}}^4-2 {x_{(8)}}^4-5 {x_{(8)}}^2-2\right)\big).  
\end{eqnarray*} }

We substitute the above expressions into $f_1$ and $f_2$ and we obtain 
$$
2x_{(8)}(x_{(8)} - 1)g_1(x_{(6)}, x_{(8)}) \ \ \ \mbox{and} \ \ \ 2 {x_{(8)}}^2g_2(x_{(6)}, x_{(8)}),
$$
where $g_1$ and $g_2$ given as follows:
{\small \begin{eqnarray*}
&&g_1(x_{(6)}, x_{(8)}) = 288 {x_{(6)}}^{19}-780 {x_{(8)}} {x_{(6)}}^{18}-780 {x_{(6)}}^{18}+288 {x_{(8)}}^2{x_{(6)}}^{17}+3712 {x_{(8)}} {x_{(6)}}^{17}\\
&& +288 {x_{(6)}}^{17}-4232 {x_{(8)}}^2{x_{(6)}}^{16} -4232 {x_{(8)}} {x_{(6)}}^{16}-864 {x_{(8)}}^4 {x_{(6)}}^{15}+3944{x_{(8)}}^3 {x_{(6)}}^{15}\\&&
+1888 {x_{(8)}}^2 {x_{(6)}}^{15} +3944 {x_{(8)}}{x_{(6)}}^{15} -864 {x_{(6)}}^{15}+2340 {x_{(8)}}^5 {x_{(6)}}^{14}-1084 {x_{(8)}}^4{x_{(6)}}^{14}\\&&
+4532 {x_{(8)}}^3 {x_{(6)}}^{14}+4532 {x_{(8)}}^2 {x_{(6)}}^{14} -1084{x_{(8)}} {x_{(6)}}^{14}+2340 {x_{(6)}}^{14}-864 {x_{(8)}}^6 {x_{(6)}}^{13}\\&&
-7712{x_{(8)}}^5 {x_{(6)}}^{13}-5104 {x_{(8)}}^4 {x_{(6)}}^{13} -37056 {x_{(8)}}^3{x_{(6)}}^{13} -5104 {x_{(8)}}^2 {x_{(6)}}^{13}-7712 {x_{(8)}} {x_{(6)}}^{13}\\&&
-864{x_{(6)}}^{13}+8752 {x_{(8)}}^6 {x_{(6)}}^{12} +16840 {x_{(8)}}^5 {x_{(6)}}^{12}+56608{x_{(8)}}^4 {x_{(6)}}^{12} +56608 {x_{(8)}}^3 {x_{(6)}}^{12}\\&&
+16840 {x_{(8)}}^2{x_{(6)}}^{12} +8752 {x_{(8)}} {x_{(6)}}^{12}+864 {x_{(8)}}^8 {x_{(6)}}^{11}-7888{x_{(8)}}^7 {x_{(6)}}^{11}-13088 {x_{(8)}}^6 {x_{(6)}}^{11} \\&&
-85256 {x_{(8)}}^5{x_{(6)}}^{11} -85280 {x_{(8)}}^4 {x_{(6)}}^{11}-85256 {x_{(8)}}^3{x_{(6)}}^{11}-13088 {x_{(8)}}^2 {x_{(6)}}^{11}-7888 {x_{(8)}} {x_{(6)}}^{11}\\&&
+864{x_{(6)}}^{11} -2340 {x_{(8)}}^9 {x_{(6)}}^{10}+4508 {x_{(8)}}^8 {x_{(6)}}^{10}-400{x_{(8)}}^7 {x_{(6)}}^{10}+82480 {x_{(8)}}^6 {x_{(6)}}^{10}\\&&
+117717 {x_{(8)}}^5{x_{(6)}}^{10} +117717 {x_{(8)}}^4 {x_{(6)}}^{10} +82480 {x_{(8)}}^3 {x_{(6)}}^{10}-400{x_{(8)}}^2 {x_{(6)}}^{10} 
+4508 {x_{(8)}} {x_{(6)}}^{10}\\&& 
-2340 {x_{(6)}}^{10} +864{x_{(8)}}^{10} {x_{(6)}}^9+4288 {x_{(8)}}^9 {x_{(6)}}^9 +896 {x_{(8)}}^8{x_{(6)}}^9 -29600 {x_{(8)}}^7 {x_{(6)}}^9\\&&
-149664 {x_{(8)}}^6 {x_{(6)}}^9 -102552{x_{(8)}}^5 {x_{(6)}}^9-149664 {x_{(8)}}^4 {x_{(6)}}^9-29600 {x_{(8)}}^3{x_{(6)}}^9 +896 {x_{(8)}}^2 {x_{(6)}}^9 
\\&&+4288 {x_{(8)}} {x_{(6)}}^9 +864{x_{(6)}}^9-4808 {x_{(8)}}^{10} {x_{(6)}}^8-12536 {x_{(8)}}^9 {x_{(6)}}^8+4616{x_{(8)}}^8 {x_{(6)}}^8 \\&&
+91112 {x_{(8)}}^7 {x_{(6)}}^8 +120582 {x_{(8)}}^6{x_{(6)}}^8+120582 {x_{(8)}}^5 {x_{(6)}}^8 
+91112 {x_{(8)}}^4 {x_{(6)}}^8
+4616{x_{(8)}}^3 {x_{(6)}}^8\\
&&
-12536 {x_{(8)}}^2 {x_{(6)}}^8 -4808 {x_{(8)}}{x_{(6)}}^8-288{x_{(8)}}^{12} {x_{(6)}}^7+3944 {x_{(8)}}^{11} {x_{(6)}}^7+12064 {x_{(8)}}^{10}{x_{(6)}}^7\\&&
+29328 {x_{(8)}}^9 {x_{(6)}}^7 -66608 {x_{(8)}}^8 {x_{(6)}}^7 -36462{x_{(8)}}^7 {x_{(6)}}^7-152460 {x_{(8)}}^6 {x_{(6)}}^7  -36462 {x_{(8)}}^5{x_{(6)}}^7\\
&&
 -66608 {x_{(8)}}^4 {x_{(6)}}^7 +29328 {x_{(8)}}^3 {x_{(6)}}^7+12064 {x_{(8)}}^2 {x_{(6)}}^7 +3944 {x_{(8)}} {x_{(6)}}^7-288 {x_{(6)}}^7\\&&
 +780 {x_{(8)}}^{13}{x_{(6)}}^6-2644 {x_{(8)}}^{12} {x_{(6)}}^6 -4348 {x_{(8)}}^{11} {x_{(6)}}^6-40284{x_{(8)}}^{10} {x_{(6)}}^6 
 +34975 {x_{(8)}}^9 {x_{(6)}}^6\\
 % \end{eqnarray*}
 %\begin{eqnarray*}
 &&
 +5303 {x_{(8)}}^8{x_{(6)}}^6+79136 {x_{(8)}}^7 {x_{(6)}}^6 +79136 {x_{(8)}}^6 {x_{(6)}}^6+5303{x_{(8)}}^5 {x_{(6)}}^6+34975 {x_{(8)}}^4 {x_{(6)}}^6
\end{eqnarray*}
 \begin{eqnarray*}
 &&
  -40284 {x_{(8)}}^3{x_{(6)}}^6-4348 {x_{(8)}}^2 {x_{(6)}}^6 -2644 {x_{(8)}} {x_{(6)}}^6+780{x_{(6)}}^6-288 {x_{(8)}}^{14} {x_{(6)}}^5\\&&
  -288 {x_{(8)}}^{13} {x_{(6)}}^5+3920{x_{(8)}}^{12} {x_{(6)}}^5 +20576 {x_{(8)}}^{11} {x_{(6)}}^5 +7120 {x_{(8)}}^{10}{x_{(6)}}^5-20496 {x_{(8)}}^9 {x_{(6)}}^5\\&&
  +912 {x_{(8)}}^8 {x_{(6)}}^5-82852{x_{(8)}}^7 {x_{(6)}}^5+912 {x_{(8)}}^6 {x_{(6)}}^5 -20496 {x_{(8)}}^5{x_{(6)}}^5+7120 {x_{(8)}}^4 {x_{(6)}}^5\\&&
  +20576 {x_{(8)}}^3 {x_{(6)}}^5+3920{x_{(8)}}^2 {x_{(6)}}^5-288 {x_{(8)}} {x_{(6)}}^5  
  -288 {x_{(6)}}^5+288 {x_{(8)}}^{14}{x_{(6)}}^4 \\&&
  -72 {x_{(8)}}^{13} {x_{(6)}}^4-15144 {x_{(8)}}^{12} {x_{(6)}}^4-1584{x_{(8)}}^{11} {x_{(6)}}^4 -9264 {x_{(8)}}^{10} {x_{(6)}}^4+11730 {x_{(8)}}^9{x_{(6)}}^4\\&&
  +19986 {x_{(8)}}^8 {x_{(6)}}^4 +19986 {x_{(8)}}^7 {x_{(6)}}^4+11730{x_{(8)}}^6 {x_{(6)}}^4 -9264 {x_{(8)}}^5 {x_{(6)}}^4  
  -1584 {x_{(8)}}^4{x_{(6)}}^4\\&&
  -15144 {x_{(8)}}^3 {x_{(6)}}^4-72 {x_{(8)}}^2 {x_{(6)}}^4 +288 {x_{(8)}}{x_{(6)}}^4 -864 {x_{(8)}}^{14} {x_{(6)}}^3+5904 {x_{(8)}}^{13} {x_{(6)}}^3\\&&
  +6832{x_{(8)}}^{12} {x_{(6)}}^3-6224 {x_{(8)}}^{11} {x_{(6)}}^3 +24160 {x_{(8)}}^{10}{x_{(6)}}^3 -46336 {x_{(8)}}^9 {x_{(6)}}^3+33056 {x_{(8)}}^8 {x_{(6)}}^3\\&&
   -46336{x_{(8)}}^7 {x_{(6)}}^3+24160 {x_{(8)}}^6 {x_{(6)}}^3 -6224 {x_{(8)}}^5{x_{(6)}}^3+6832 {x_{(8)}}^4 {x_{(6)}}^3 +5904 {x_{(8)}}^3 {x_{(6)}}^3\\&&
   -864 {x_{(8)}}^2{x_{(6)}}^3+216 {x_{(8)}}^{15} {x_{(6)}}^2 -648 {x_{(8)}}^{14} {x_{(6)}}^2-5688{x_{(8)}}^{13} {x_{(6)}}^2+5832 {x_{(8)}}^{12} {x_{(6)}}^2\\
&& -12888 {x_{(8)}}^{11}{x_{(6)}}^2+13320 {x_{(8)}}^{10} {x_{(6)}}^2 -144 {x_{(8)}}^9 {x_{(6)}}^2-144{x_{(8)}}^8 {x_{(6)}}^2+13320 {x_{(8)}}^7 {x_{(6)}}^2\\&&
-12888 {x_{(8)}}^6{x_{(6)}}^2 +5832 {x_{(8)}}^5 {x_{(6)}}^2-5688 {x_{(8)}}^4 {x_{(6)}}^2 -648 {x_{(8)}}^3{x_{(6)}}^2+216 {x_{(8)}}^2 {x_{(6)}}^2\\&&
+1728 {x_{(8)}}^{14} {x_{(6)}}-432{x_{(8)}}^{13} {x_{(6)}} +1728 {x_{(8)}}^{12} {x_{(6)}} -3456 {x_{(8)}}^{10}{x_{(6)}}+864 {x_{(8)}}^9 {x_{(6)}}\\&&
-3456 {x_{(8)}}^8 {x_{(6)}} +1728 {x_{(8)}}^6{x_{(6)}}-432 {x_{(8)}}^5 {x_{(6)}}
+1728 {x_{(8)}}^4 {x_{(6)}} -216 {x_{(8)}}^{15}-216{x_{(8)}}^{14}\\&&
-216 {x_{(8)}}^{13}-216 {x_{(8)}}^{12}+432 {x_{(8)}}^{11}+432{x_{(8)}}^{10}+432 {x_{(8)}}^9 +432 {x_{(8)}}^8-216 {x_{(8)}}^7\\&&
-216 {x_{(8)}}^6-216{x_{(8)}}^5-216 {x_{(8)}}^4
\end{eqnarray*} }
and
{\small \begin{eqnarray*}
&& g_2(x_{(6)}, x_{(8)})  = 48 {x_{(6)}}^{14}-224 {x_{(6)}}^{13}+48 {x_{(6)}}^{12} {x_{(8)}}^2+352 {x_{(6)}}^{12} {x_{(8)}}+356 {x_{(6)}}^{12}\\&&
-576 {x_{(6)}}^{11} {x_{(8)}}^2 -1184 {x_{(6)}}^{11}{x_{(8)}}-224 {x_{(6)}}^{11} -96 {x_{(6)}}^{10} {x_{(8)}}^4+576 {x_{(6)}}^{10}{x_{(8)}}^3\\&&
+1836 {x_{(6)}}^{10} {x_{(8)}}^2+1536 {x_{(6)}}^{10} {x_{(8)}} -48{x_{(6)}}^{10}-128 {x_{(6)}}^9 {x_{(8)}}^4-3200 {x_{(6)}}^9 {x_{(8)}}^3\\&&
-1680{x_{(6)}}^9 {x_{(8)}}^2 -1536 {x_{(6)}}^9 {x_{(8)}}+448 {x_{(6)}}^9-96 {x_{(6)}}^8{x_{(8)}}^6 -128 {x_{(6)}}^8 {x_{(8)}}^5+2936 {x_{(6)}}^8 {x_{(8)}}^4\\&&
+3520 {x_{(6)}}^8{x_{(8)}}^3+1440 {x_{(6)}}^8 {x_{(8)}}^2+832 {x_{(6)}}^8 {x_{(8)}} -712{x_{(6)}}^8+576 {x_{(6)}}^7 {x_{(8)}}^6 -2496 {x_{(6)}}^7 {x_{(8)}}^5\\&&
-3280 {x_{(6)}}^7{x_{(8)}}^4-2016 {x_{(6)}}^7 {x_{(8)}}^3-1616 {x_{(6)}}^7 {x_{(8)}}^2+832 {x_{(6)}}^7{x_{(8)}}+448 {x_{(6)}}^7+48 {x_{(6)}}^6 {x_{(8)}}^8 \\&&
-576 {x_{(6)}}^6 {x_{(8)}}^7+1704{x_{(6)}}^6 {x_{(8)}}^6+2880 {x_{(6)}}^6 {x_{(8)}}^5+2096 {x_{(6)}}^6 {x_{(8)}}^4+944{x_{(6)}}^6 {x_{(8)}}^3\\&&
+408 {x_{(6)}}^6 {x_{(8)}}^2 -1536 {x_{(6)}}^6 {x_{(8)}}-48{x_{(6)}}^6 +352 {x_{(6)}}^5 {x_{(8)}}^8-256 {x_{(6)}}^5 {x_{(8)}}^7  -2928 {x_{(6)}}^5{x_{(8)}}^6\\&&
+192 {x_{(6)}}^5 {x_{(8)}}^5-3376 {x_{(6)}}^5 {x_{(8)}}^4 +2320 {x_{(6)}}^5{x_{(8)}}^3-688 {x_{(6)}}^5 {x_{(8)}}^2+1536 {x_{(6)}}^5 {x_{(8)}} -224 {x_{(6)}}^5\\&&
+48{x_{(6)}}^4 {x_{(8)}}^{10}-224 {x_{(6)}}^4 {x_{(8)}}^9+164 {x_{(6)}}^4{x_{(8)}}^8 +1088 {x_{(6)}}^4 {x_{(8)}}^7+876 {x_{(6)}}^4 {x_{(8)}}^6 \\&&
-256 {x_{(6)}}^4{x_{(8)}}^5+1633 {x_{(6)}}^4 {x_{(8)}}^4 -2880 {x_{(6)}}^4 {x_{(8)}}^3+1036{x_{(6)}}^4 {x_{(8)}}^2 -1184 {x_{(6)}}^4 {x_{(8)}} \\&&
+356 {x_{(6)}}^4 +224 {x_{(6)}}^3{x_{(8)}}^9-1104 {x_{(6)}}^3 {x_{(8)}}^8+384 {x_{(6)}}^3 {x_{(8)}}^7-1296 {x_{(6)}}^3{x_{(8)}}^6 +1120 {x_{(6)}}^3 {x_{(8)}}^5 \\&&
-464 {x_{(6)}}^3 {x_{(8)}}^4+1376 {x_{(6)}}^3{x_{(8)}}^3-368 {x_{(6)}}^3 {x_{(8)}}^2+352 {x_{(6)}}^3 {x_{(8)}}-224 {x_{(6)}}^3-84{x_{(6)}}^2 {x_{(8)}}^{10}\\&&
+192 {x_{(6)}}^2 {x_{(8)}}^9 +304 {x_{(6)}}^2{x_{(8)}}^8-400 {x_{(6)}}^2 {x_{(8)}}^7+1064 {x_{(6)}}^2 {x_{(8)}}^6-1504 {x_{(6)}}^2{x_{(8)}}^5\\&&
+944 {x_{(6)}}^2 {x_{(8)}}^4-880 {x_{(6)}}^2 {x_{(8)}}^3 +316 {x_{(6)}}^2{x_{(8)}}^2+48 {x_{(6)}}^2 -192 {x_{(6)}} {x_{(8)}}^9+240 {x_{(6)}} {x_{(8)}}^8\\&&
-240{x_{(6)}} {x_{(8)}}^7+192 {x_{(6)}} {x_{(8)}}^6+192 {x_{(6)}} {x_{(8)}}^5 -240{x_{(6)}} {x_{(8)}}^4 +240 {x_{(6)}} {x_{(8)}}^3-192 {x_{(6)}} {x_{(8)}}^2\\&&
+36{x_{(8)}}^{10}-72 {x_{(8)}}^6+36 {x_{(8)}}^2.
\end{eqnarray*} }

We consider a polynomial ring $R = {\bb Q}[z, x_{(6)}, x_{(8)}]$ and an ideal $I$ generated by $\{g_1, g_2,$ $ z x_{(6)} x_{(8)} - 1\}$ to find non zero solutions for the equations $g_1 = 0,\ g_2 = 0$.  We take a lexicographic order $>$ with $z > x_{(6)} > x_{(8)}$ for a monomial ordering on $R$.  
Then by the aid of computer  we see that a Gr\"obner basis for the ideal $I$ contains the polynomial 
{ \begin{eqnarray*}
(2 {x_{(8)}}^2+5)^2 \left(256 {x_{(8)}}^6-352
   {x_{(8)}}^5+896 {x_{(8)}}^4-961 {x_{(8)}}^3+896 {x_{(8)}}^2-352
   {x_{(8)}}+256\right)^4  h(x_{(8)}) 
   \end{eqnarray*} }
 where   $ h(x_{(8)})=$ 
 {%\scriptsize 
 \tiny
 \begin{eqnarray*}
&&     
40644078463495519177723084800000000 {x_{(8)}}^{58} 
   -636599360907732347401178972160000000 {x_{(8)}}^{57} \\ 
 &&  +5191100419195183352367435153408000000
   {x_{(8)}}^{56}-29753785824771896438059349508096000000
   {x_{(8)}}^{55}\\ 
 && +135192627680445877856223796819845120000
   {x_{(8)}}^{54}-516361347361994237364041410191704064000
   {x_{(8)}}^{53}\\ 
 && +1714176845340411043085057653911857971200
   {x_{(8)}}^{52}-5056951910439083847722148201478778880000
   {x_{(8)}}^{51}\\ 
 && +13463173946418384084581239673092042268160
   {x_{(8)}}^{50}-32702604858836427601562556436132631495680
   {x_{(8)}}^{49}\\ 
 && +73074478080162709767298450382256432673024
   {x_{(8)}}^{48}-151160839532721827123829100525688506355648
   {x_{(8)}}^{47}\\ 
 && +290859738683958218715335374171387713117587
   {x_{(8)}}^{46}-522518193307621597823670610471545296918034
   {x_{(8)}}^{45}\\ 
 && +878838272222037101767064678576069389361339
   {x_{(8)}}^{44}-1386607752438435794483944861009996481373328
   {x_{(8)}}^{43}\\ 
 && +2054687728647824826465309756000191412577616
   {x_{(8)}}^{42}-2860547639552810817086458718197988792478640
   {x_{(8)}}^{41}\\ 
 && +3739378101475752750734169569969105594053280
   {x_{(8)}}^{40}-4581348549386918358670027348735744667090032
   {x_{(8)}}^{39}\\ 
 && +5241820414109727354087818276473527894123404
   {x_{(8)}}^{38}-5565824237345596956736062729200672703813496
   {x_{(8)}}^{37}\\ 
 && +5424363186128466881300155939542628484559548
   {x_{(8)}}^{36}-4754699531043526537144910927195252969271760
   {x_{(8)}}^{35}\\ 
 && +3591716347523905215063410201185994361815744
   {x_{(8)}}^{34}-2079284841096253856056499840520306164822992
   {x_{(8)}}^{33}\\ 
 && +454714233323193018091497888909946789089488
   {x_{(8)}}^{32}+995461754392201718690335133719733779341968
   {x_{(8)}}^{31}\\ 
 && -1998006739366165727563601475806849383386902
   {x_{(8)}}^{30}+2356145206000070543512745518613265350730260
   {x_{(8)}}^{29}\\ 
 && -1998006739366165727563601475806849383386902
   {x_{(8)}}^{28}+995461754392201718690335133719733779341968
   {x_{(8)}}^{27}\\ 
 && +454714233323193018091497888909946789089488
   {x_{(8)}}^{26}-2079284841096253856056499840520306164822992
   {x_{(8)}}^{25}\\ 
 && +3591716347523905215063410201185994361815744
   {x_{(8)}}^{24}-4754699531043526537144910927195252969271760
   {x_{(8)}}^{23}\\ 
 && +5424363186128466881300155939542628484559548
   {x_{(8)}}^{22}-5565824237345596956736062729200672703813496
   {x_{(8)}}^{21} 
 \\
 && +5241820414109727354087818276473527894123404
   {x_{(8)}}^{20}-4581348549386918358670027348735744667090032
   {x_{(8)}}^{19}\\ 
 && +3739378101475752750734169569969105594053280
   {x_{(8)}}^{18}-2860547639552810817086458718197988792478640
   {x_{(8)}}^{17}\\
  % \end{eqnarray*} }
% {\footnotesize \begin{eqnarray*}  
 && +2054687728647824826465309756000191412577616
   {x_{(8)}}^{16}-1386607752438435794483944861009996481373328
   {x_{(8)}}^{15}\\ 
 && +878838272222037101767064678576069389361339
   {x_{(8)}}^{14}-522518193307621597823670610471545296918034
   {x_{(8)}}^{13}\\ 
 && +290859738683958218715335374171387713117587
   {x_{(8)}}^{12}-151160839532721827123829100525688506355648
   {x_{(8)}}^{11}\\ 
 && +73074478080162709767298450382256432673024
   {x_{(8)}}^{10}-32702604858836427601562556436132631495680
   {x_{(8)}}^9\\ 
 && +13463173946418384084581239673092042268160
   {x_{(8)}}^8-5056951910439083847722148201478778880000
   {x_{(8)}}^7\\ 
 && +1714176845340411043085057653911857971200
   {x_{(8)}}^6-516361347361994237364041410191704064000
   {x_{(8)}}^5\\ 
 && +135192627680445877856223796819845120000
   {x_{(8)}}^4-29753785824771896438059349508096000000
   {x_{(8)}}^3\\ 
 && +5191100419195183352367435153408000000
   {x_{(8)}}^2-636599360907732347401178972160000000
   {x_{(8)}}\\ 
 && +40644078463495519177723084800000000. 
\end{eqnarray*} }

We solve the equation $h(x_{(8)}) = 0$ numerically and we obtain four positive solutions which are given approximately as
$$
x_{(8)} \approx 0.67539547,\ x_{(8)} \approx 0.94334874,\ x_{(8)} \approx 1.0600534,\ x_{(8)} \approx 1.4806140.
$$
Also the Gr\"obner basis of the ideal $I$ contains the polynomial
\begin{eqnarray*}&&
(2 {x_{(8)}}^2+5)^2 \left(256 {x_{(8)}}^6-352
   {x_{(8)}}^5+896 {x_{(8)}}^4-961 {x_{(8)}}^3+896 {x_{(8)}}^2-352
   {x_{(8)}}+256\right)^4 \\&& ( x_{(6)} - w( x_{(8)}) )
   \end{eqnarray*} 
   \noindent
where $w( x_{(8)})$ is a polynomial of  $x_{(8)}$ with integer coefficients.  We substitute the above values of $x_{(8)}$ in the above equation and we obtain the solutions
\begin{eqnarray*}
 (x_{(6)}, x_{(8)}) \approx (1.2876390,\ 0.67539547), && 
 (x_{(6)}, x_{(8)}) \approx (0.50583947,\ 0.94334874),\\ 
 (x_{(6)}, x_{(8)}) \approx (0.53621683,\ 1.0600534), &&
 (x_{(6)}, x_{(8)}) \approx (1.9064963,\ 1.4806140). 
\end{eqnarray*}
We substitute the above solutions into $u_1, u_2, u_4$ and $u_5$ and we find three Einstein metrics on $V_4\bb{C}^{6}\cong\SU(6)/\SU(2)$ which are given as follows:
\begin{itemize}
\item[(1)] $(u_1,\ u_2,\ v_4,\ v_5,\ x_{(6)},\ x_{(8)},\ x_{(7)})
\newline
\approx (0.29693405,\ 1.0265896,\ 0.80273874,\ 1.4899863,\ 1.2876390,\ 0.67539547,\ 1)$
\item[(2)] $(u_1,\ u_2,\ v_4,\ v_5,\ x_{(6)},\ x_{(8)},\ x_{(7)})
\newline
\approx(0.60542236\  0.34843563,\ 1.451313,\ 0.52095356,\ 0.50583947,\ 0.94334874,\ 1)$
\item[(3)] $(u_1,\ u_2,\ v_4,\ v_5,\ x_{(6)},\ x_{(8)},\ x_{(7)})
\newline
\approx(0.36936036,\ 0.64178000,\ 1.5384701,\  0.55223857,\ 0.53621683,\ 1.0600534,\ 1)$
\item[(4)] $(u_1,\ u_2,\ v_4,\ v_5,\ x_{(6)},\ x_{(8)},\ x_{(7)})
\newline
\approx(1.5199830,\ 0.43964472,\ 1.1885462,\ 2.2060946,\ 1.9064963,\ 1.4806140,\ 1)$. 
\end{itemize}

%\bigskip

In the case where $x_{(8)} = 1$ then $c = 0$,  and from $g_2(x_{(6)}, x_{(8)}) = 0$  we obtain 
\begin{eqnarray*}
&& (2 {x_{(8)}}-3) (2 {x_{(8)}}-1) (12 {x_{(8)}}^8-32 {x_{(8)}}^7+116 {x_{(8)}}^6-240
   {x_{(8)}}^5+384 {x_{(8)}}^4  \\ & & 
   -576 {x_{(8)}}^3+508 {x_{(8)}}^2-440 {x_{(8)}}+219 ) = 0.
\end{eqnarray*}
The solutions of the above equation are
$$
x_{(6)} = 1/2, \ x_{(6)} = 3/2, \ x_{(6)} \approx 0.806273, \ x_{(6)} \approx 1.54752. 
$$
Then  by  substituting  the above solutions into $u_1, u_2, v_4$ and $v_5$ we  find two Einstein metrics of Jensen's type  and two  more  Einstein metrics on $V_4\bb{C}^{6}\cong\SU(6)/\SU(2)$, which are given as follows:
\begin{itemize}
\item[(1)] $(u_1,\ u_2,\ v_4,\ v_5,\ x_{(6)},\ x_{(8)},\ x_{(7)}) = (1/2,\ 1/2,\ 3/2,\ 1/2,\ 1/2,\ 1,\ 1)$
\item[(2)] $(u_1,\ u_2,\ v_4,\ v_5,\ x_{(6)},\ x_{(8)},\ x_{(7)}) = (3/2,\ 3/2,\  17/18,\ 3/2,\ 3/2,\ 1,\ 1)$
\item[(3)] $(u_1,\ u_2,\ v_4,\ v_5,\ x_{(6)},\ x_{(8)},\ x_{(7)}) 
\newline\approx (0.276881,\ 0.276881,\ 1.43815,\  1.05836,\ 0.806273,\ 1,\ 1)$
\item[(4)] $(u_1,\ u_2,\ v_4,\ v_5,\ x_{(6)},\ x_{(8)},\ x_{(7)})
\newline\approx(0.326422,\ 0.326422,\ 1.17554,\ 1.92172,\ 1.54752,\ 1,\ 1)$.
\end{itemize}

The above computations can be summarized in the following theorem:
\begin{theorem} {\rm 1)}  The complex Stiefel manifold $V_{2}\bb{C}^{4}=\SU(4)/\SU(2)$ admits 
two $\Ad(\s$ $(\U(1)\times\U(1)\times\U(2))$-invariant Einstein metrics of the form {\rm(\ref{metrics2})}, which are of Jensen's type.

\noindent
{\rm 2)}  The complex Stiefel manifold $V_{3}\bb{C}^{5}=\SU(5)/\SU(2)$ admits 
four $\Ad(\s(\U(1)\times\U(2)\times\U(2))$-invariant Einstein metrics of the form {\rm (\ref{metrics2})}, two of which are of Jensen's type.

\noindent
{\rm 3)} The complex Stiefel manifold $V_{4}\bb{C}^{6}=\SU(6)/\SU(2)$ admits 
eight $\Ad(\s(\U(2)\times\U(2)\times\U(2))$-invariant Einstein metrics of the form {\rm (\ref{metrics2})},  two of which are of Jensen's type.
\end{theorem}

\subsection{The Stiefel manifolds $V_{2m}\bb{C}^{2m+n}$}

In the next theorem we prove existence of Einstein metrics, which are not of Jensen's type,  on large families of complex Stiefel manifolds.

\begin{theorem}
 The complex Stiefel manifolds $V_{2m}\bb{C}^{2m + n}$ admit at least two $\Ad(\s(\U(m)\times\U(m)\times\U(n)))$-invariant Einstein metrics, which are not of Jensen's type,  for the following values of $m$ and $n$: 
 \begin{center}
\begin{tabular}{|c|ccc|}
%& & $B_{n, m}(3/2)$   & & $B_{n, m}(x_{(6)})$    \\
 \hline
 $m \geq 8$ &  & $n\geq m/2$ &     \\
     \hline 
 $m =6,  \ 7$ &  & $n\geq 4$ &     \\
  \hline 
  $m =4, \ 5$ & &   $n\geq 3$ &   \\  
 %   \thickline 
%& &$B_{n, m}(4/3)$ & & \\
\hline 
 $m= 2, \ 3$ &  &  $n\geq 2$ &    \\
   \hline 
 \end{tabular} .
 \end{center}
\end{theorem}
  \begin{proof}
We consider the metric on $V_{2m}\bb{C}^{2m+n}$ with $a = 1, b = 0, c = 0$ and $d = 1$ (diagonal metric).  Then for $\ell = m$ we take from $r_0 = 0$ that 
\begin{eqnarray*}
r_0 = \frac{({x_{(8)}}-{x_{(7)}}) ({x_{(8)}}+{x_{(7)}})}{{x_{(8)}}^2
   {x_{(7)}}^2}\frac12\sqrt{\frac{n}{2m+n}}=0.  
\end{eqnarray*}  
We set $x_{(8)} = x_{(7)} = 1$. Then  system (\ref{sist3}) reduces to the system

\begin{equation}\label{systemM1}
(\star)\left\{
\begin{array}{l} 
 f_1 = ({u_1}-{u_2}) (m {u_1} {u_2}-m {x_{(6)}}^2+n {u_1} {u_2} {x_{(6)}}^2)=0, \nonumber\\
 f_2 = m {u_2}^2-2 m {u_2} {v_4} {x_{(6)}}^2+m {x_{(6)}}^2+n {u_2}^2 {x_{(6)}}^2-n {u_2} {v_4} {x_{(6)}}^2=0,  \nonumber\\
 f_3 = 2 m {v_4} {x_{(6)}}^2-2 m {v_5}+n {v_4} {x_{(6)}}^2-n {v_5} {x_{(6)}}^2=0  \nonumber\\
 f_4 = m^2 {u_1}+m^2 {u_2}+2 m^2 {v_5}-4 m^2 {x_{(6)}}+m n {v_5} {x_{(6)}}^2 -m n {x_{(6)}}^3\nonumber\\
 -{u_1}-{u_2}+2 {v_5}=0,  \\
 f_5 = -2 m^2 n {u_1}+2 m^2 n {u_2} {x_{(6)}}^2-2 m^2 n {u_2}+2 m^2 n {x_{(6)}}^3-8 m^2 n {x_{(6)}}^2 \nonumber\\
+8 m^2 n {x_{(6)}} +2 m n^2{x_{(6)}}^3  -4 m n^2 {x_{(6)}}^2+2 m {v_4} {x_{(6)}}^2+2 n {u_1}-2 n {u_2} {x_{(6)}}^2 \nonumber\\
+2 n {u_2} +n {v_4}{x_{(6)}}^2+n {v_5} {x_{(6)}}^2-4 n {v_5}=0,  \nonumber\\
 f_6 = (m-1) (m+1) (u_1-u_2)=0.\nonumber
\end{array}
\right. 
\end{equation}

From  $f_1=0$ and $f_6=0$ we have that $u_1 = u_2$   and by substituting into the system ($\star$)  we obtain
\begin{eqnarray*}
&& f_2 = m {u_1}^2-2 m {u_1} {v_4} {x_{(6)}}^2+m {x_{(6)}}^2+n {u_1}^2 {x_{(6)}}^2-n {u_1} {v_4} {x_{(6)}}^2=0, \\
&& f_3 = 2 m {v_4} {x_{(6)}}^2-2 m {v_5}+n {v_4} {x_{(6)}}^2-n {v_5} {x_{(6)}}^2=0,  \\
&& f_4 = 2 m^2 {u_1}+2 m^2 {v_5}-4 m^2 {x_{(6)}}+m n {v_5} {x_{(6)}}^2-m n {x_{(6)}}^3-2 {u_1}+2 {v_5}=0, \\
&& f_5 = 2 m^2 n {u_1} {x_{(6)}}^2-4 m^2 n {u_1}+2 m^2 n {x_{(6)}}^3-8 m^2 n {x_{(6)}}^2+8 m^2 n {x_{(6)}}+2 m n^2 {x_{(6)}}^3 \\
&&-4 m n^2{x_{(6)}}^2+2 m {v_4} {x_{(6)}}^2-2 n {u_1} {x_{(6)}}^2+4 n {u_1}+n {v_4} {x_{(6)}}^2+n {v_5} {x_{(6)}}^2-4 n{v_5}=0. 
\end{eqnarray*}

We observe that the equations $f_3=0, f_4=0$ and $f_5=0$ are {\it linear} with respect to $u_1, v_4, v_5$. By solving the system of equations $\{ f_3=0, f_4=0, f_5=0 \}$  we obtain that
{ \begin{eqnarray}\label{linear}
&& u_1 =H_1(x_{(6)})\equiv - {x_{(6)}} (2 m^3 n {x_{(6)}}^2-8 m^3 n {x_{(6)}}+8 m^3 n+m^2 n^2 {x_{(6)}}^4 \label{eq_u_1}   -4 m^2 n^2 {x_{(6)}}^3 
\\&& +6 m^2 n^2 {x_{(6)}}^2 -4 m^2n^2 {x_{(6)}} +4 m^2+m n^3 {x_{(6)}}^4-2 m n^3 {x_{(6)}}^3 
+7 m n {x_{(6)}}^2-8 m n {x_{(6)}} \nonumber \\&&+n^2 {x_{(6)}}^4 \nonumber  -4 n^2
   {x_{(6)}})  / 
  \big( {(m^2-1) (2 m n {x_{(6)}}^2-4 m n+n^2 {x_{(6)}}^4-2 n^2 {x_{(6)}}^2-2)}\big)% \equiv H_1(x_{(6)})
    \nonumber\\
   && v_4 = H_4(x_{(6)})\equiv  \frac{(2 m+n {x_{(6)}}^2)(6 m n {x_{(6)}}-8 m n+n^2 {x_{(6)}}^3-4 n^2)}{(2 m+n) (2 m n {x_{(6)}}^2-4 m n+n^2
   {x_{(6)}}^4-2 n^2 {x_{(6)}}^2-2)} %\equiv H_4(x_{(6)})
    \label{eq_v_4}\\ \nonumber\\
&& v_5 = H_5(x_{(6)}) \equiv \frac{{x_{(6)}}^2 (6 m n {x_{(6)}}-8 m n + n^2 {x_{(6)}}^3-4 n^2 )}{2 m n {x_{(6)}}^2-4 m n+n^2 {x_{(6)}}^4-2 n^2{x_{(6)}}^2-2} \label{eq_v_5}. %\equiv H_5(x_{(6)}).
\end{eqnarray} }
%We denote the right hand side of equation (\ref{eq_u_1})  by  $H_1(x_{(6)}) $, the right hand side of equation (\ref{eq_v_4})  by  $H_4(x_{(6)}) $ and the right hand side of equation (\ref{eq_v_5})  by  $H_5(x_{(6)})$. 

From  equations (\ref{eq_v_4}) and (\ref{eq_v_5}),  we see that the value of $v_4$ is positive if and only if the value of $v_5$ is positive. 

Now we substitute the values $u_1$ and $v_4$  into $f_2$ and we see that 
$$
f_2 = \frac{m x_{(6)}^2 A(x_{(6)}) B_{n,m}(x_{(6)})}{\Gamma(x_{(6)})}=0,
$$
where
{\small \begin{eqnarray*}
&& A(x_{(6)}) = 2 m n {x_{(6)}}^2-4 m n {x_{(6)}}+2 m n+n^2 {x_{(6)}}^2-2 n^2 {x_{(6)}}+1, \\
&& B_{n, m}(x_{(6)}) = 16 m^5 n {x_{(6)}}^2-48 m^5 n {x_{(6)}}+40 m^5 n+20 m^4 n^2 {x_{(6)}}^4-64 m^4 n^2 {x_{(6)}}^3\\&&
+68 m^4 n^2 {x_{(6)}}^2 -24 m^4 n^2{x_{(6)}}+20 m^4+8 m^3 n^3 {x_{(6)}}^6-28 m^3 n^3 {x_{(6)}}^5+42 m^3 n^3 {x_{(6)}}^4 \\&&
-32 m^3 n^3 {x_{(6)}}^3+56 m^3 n {x_{(6)}}^2 -32 m^3n {x_{(6)}}-16 m^3 n+m^2 n^4 {x_{(6)}}^8-4 m^2 n^4 {x_{(6)}}^7\\&&
+11 m^2 n^4 {x_{(6)}}^6-14 m^2 n^4 {x_{(6)}}^5+41 m^2 n^2 {x_{(6)}}^4 -32m^2 n^2 {x_{(6)}}^3-8 m^2 n^2 {x_{(6)}}^2-16 m^2 n^2 {x_{(6)}} \\&&
-8 m^2+m n^5 {x_{(6)}}^8-2 m n^5 {x_{(6)}}^7+11 m n^3 {x_{(6)}}^6-8 m n^3{x_{(6)}}^5 -16 m n^3 {x_{(6)}}^3-16 m n {x_{(6)}}^2 \\&&
+16 m n {x_{(6)}}+8 m n+n^4 {x_{(6)}}^8-4 n^4 {x_{(6)}}^5-4 n^2 {x_{(6)}}^4+4 n^2{x_{(6)}}^2+8 n^2 {x_{(6)}}+4,\\
&& \Gamma(x_{(6)}) = (m-1)^2 (m+1)^2 (2 m n {x_{(6)}}^2-4 m n+n^2 {x_{(6)}}^4-2 n^2 {x_{(6)}}^2-2)^2.
\end{eqnarray*} }

\noindent{\bf \underline{Case of $A(x_{(6)}) = 0$}}\\
We consider the polynomial ring $R = \bb{Q}[x_{(6)}, u_1, v_4, v_5]$ and the ideal $I$ generated by the polynomials $\{f_2,\,f_3,\, f_4,\, f_5,\, A(x_{(6)}),\, z\,x_{(6)}\,u_1\,v_4\,v_5-1\}$.  We take a lexicographic ordering $>$ with $z>v_4>v_5>u_1>x_{(6)}$ for a monomial ordering on $R$.  Then, by the aid of computer, we see that a Gr\"obner basis for the ideal $I$ contains the polynomial $A(x_{(6)})$ and
\begin{eqnarray*} & F(x_{(6)}, u_1) = (m-1) m (m+1) (2 m n+1)(32 m^5 n+144 m^4-96 m^3 n^3+336 m^3 n-24 m^2 & \\&
 -176 m^2 n^4+384 m^2 n^2 -96 m n^5+168 m n^3-6 m n-16 n^6+24n^4+1)({u_1}-{x_{(6)}}). &
\end{eqnarray*}
Now we consider the lexicographic order $>$ with $z>v_4>u_1>v_5>x_{(6)}$. Then a Gr\"obner basis for the ideal $I$ contains the polynomial $A(x_{(6)})$ and the polynomial
\begin{eqnarray*}
&& G(x_{(6)}, v_5) = m (2 m n+1) (32 m^5 n+144 m^4-96 m^3 n^3+336 m^3 n-176 m^2 n^4\\
&& +384 m^2 n^2-24 m^2-96 m n^5+168 m n^3 -6 m n-16 n^6+24 n^4+1)({v_5}-{x_{(6)}}). 
\end{eqnarray*}
Therefore from the equations $F(x_{(6)}, u_1) = 0$ and $G(x_{(6)}, v_5) = 0$ we have $x_{(6)} = u_1 = v_5$.  We substitute into $f_i$, $i = 2,3,4,5$ and we have
\begin{eqnarray*}
&& g_1 = -{x_{(6)}}^2 (2 m {v_4} {x_{(6)}}-2 m+n {v_4} {x_{(6)}}-n {x_{(6)}}^2)=0, \\
&& A(x_{(6)}) =   2 m n {x_{(6)}}^2-4 m n {x_{(6)}}+2 m n+n^2 {x_{(6)}}^2-2 n^2 {x_{(6)}}+1=0, \\
&& g_2 = {x_{(6)}} (2 m {v_4} {x_{(6)}}-2 m+n {v_4} {x_{(6)}}-n {x_{(6)}}^2)=0,\\
&& g_3 = {x_{(6)}} (4 m^2 n {x_{(6)}}^2-8 m^2 n {x_{(6)}}+4 m^2 n+2 m n^2 {x_{(6)}}^2-4 m n^2 {x_{(6)}}\\
&&+2 m {v_4} {x_{(6)}}+n {v_4}{x_{(6)}}-n {x_{(6)}}^2)=0.
\end{eqnarray*}

From $A(x_{(6)}) = 0$ we have the solutions for $x_{(6)}$ as follows:
\begin{eqnarray*}
x_{(6)}^1 = \frac{-\sqrt{2 m n^3-2 m n+n^4-n^2}+2 m n+n^2}{2 m n+n^2}\\
x_{(6)}^2 = \frac{\sqrt{2 m n^3-2 m n+n^4-n^2}+2 m n+n^2}{2 m n+n^2}.
\end{eqnarray*}
We substitute into $g_1 = 0$ and we find
\begin{eqnarray*}
v_4^1 = \frac{n \left(-4 m^2+2 \sqrt{\left(n^3-n\right) (2 m+n)}-4 m n-2 n^2+1\right)}{(2 m+n) \left(\sqrt{\left(n^3-n\right) (2 m+n)}-2 m n-n^2\right)} \\ \\
v_4^2 = \frac{n \left(4 m^2+2 \sqrt{\left(n^3-n\right) (2 m+n)}+4 m n+2 n^2-1\right)}{(2 m+n) \left(\sqrt{\left(n^3-n\right) (2 m+n)}+2 m n+n^2\right)}.
\end{eqnarray*}
These metrics correspond to Jensen's type metrics.  
%Next we study the case of $B(x_{(6)})$.

\noindent{\bf \underline{Case of $B_{n, m}(x_{(6)}) = 0$}}\\
For this case we observe that  the value of $B_{n, m}(x_{(6)})$ at $x_{(6)} =0$ is written as
$$B_{n, m}(0) =4 \left(5 m^4-2 m^2+1\right) (2 m n+1) >0.$$
We  also observe that  the value of $B_{n, m}(x_{(6)})$ at $x_{(6)} = 3/2$ is written as
\begin{eqnarray*}
&& B_{n, m}(3/2) = 4 m^5 n+\frac{9 m^4 n^2}{4}+20 m^4-\frac{135 m^3 n^3}{8}+62 m^3 n-\frac{6075 m^2
   n^4}{256}\\ 
   & & +\frac{921 m^2 n^2}{16}-8 m^2-\frac{2187 m n^5}{256}+\frac{675 m
   n^3}{64}-4 m n-\frac{1215 n^4}{256}+\frac{3 n^2}{4}+4. 
\end{eqnarray*}
We see that 
\begin{eqnarray*}&& B_{n, m}(3/2) = 
\left(\frac{135 m}{128}-\frac{43875 m^3}{512}\right)
   \left(n-\frac{m}{2}\right)^3+\left(-\frac{23085 m^2}{512}-\frac{1215}{256}\right)
   \left(n-\frac{m}{2}\right)^4
   \\ & & +\frac{\left(-85775 m^5+512408 m^3-13312 m\right)
   }{4096}\left(n-\frac{m}{2}\right)-\frac{3 \left(23667 m^4-22618 m^2-256\right)
   }{1024}\times \\&&
   \left(n-\frac{m}{2}\right)^2  +\frac{-10625 m^6+544050 m^4-80384
   m^2+32768}{8192}-\frac{2187}{256} m \left(n-\frac{m}{2}\right)^5.
\end{eqnarray*}
We also see that 
\begin{eqnarray*} &&
-10625 m^6+544050 m^4-80384 m^2+32768 \\ & & = 
-10625 (m^2-64)^3-1495950 (m^2-64)^2-61001984 (m^2-64)-561963008
\end{eqnarray*}
and thus $B_{n, m}(3/2)<0$ for $n\geq m/2$ and $m \geq 8$. 

  Some smaller values of $m$ can be examined separately as follows. 

For $ m =7$ we see that 
\begin{eqnarray*} &&B_{n, 7}(3/2) = -\frac{1}{256}\big(15309 (n-4)^5+605070 (n-4)^4+8694540 (n-4)^3 \\
& &
 +53940288 (n-4)^2+126839488 (n-4)+49348608\big), 
   \end{eqnarray*}
   and thus $B_{n, 7}(3/2) < 0$ for $ n \geq 4$. 
   (In fact, we can show that $B_{3, 7}({x_{(6)}}) > 0$ for all ${x_{(6)}}$.)
   
   Similarly, we obtain the following table: 
   \begin{center}
\begin{tabular}{|c|ccc|c|}
%& & $B_{n, m}(3/2)$   & & $B_{n, m}(x_{(6)})$    \\
     \hline 
 $m=6,7$ & &  $B_{n, m}(3/2)$ $ < 0$ \ \  for $n\geq 4$ & & $B_{3, m}(x_{(6)})>0$ for all $x_{(6)}$   \\
  \hline \ 
$m=4,5$& & $B_{n, m}(3/2)$ $ < 0$  \ for $n\geq 3$ & & $B_{2, m}(x_{(6)})>0$ for all $x_{(6)}$ \\  
 %   \thickline 
%& &$B_{n, m}(4/3)$ & & \\
\hline 
 $m=2,3$& &$B_{n, m}(4/3)$ $ < 0$  \  \ for   $n\geq 2$ &  & \\
   \hline 
 \end{tabular}
 \end{center}   
 %  For $ m =6$ we see that \begin{eqnarray*} &&B_{n, 6}(3/2) = -\frac{1}{256}\big(113122 (n-4)^5+482355 (n-4)^4+6535080 (n-4)^3\\& &+39235776 (n-4)^2+95504256   (n-4)+55880960\big),   \end{eqnarray*}  and thus $B_{n, 6}(3/2) < 0$ for $ n \geq 4$. (In fact, we can see that $B_{3, 6}({x_{(6)}}) > 0$ for all ${x_{(6)}}$.)

 %  For $ m =5$ we see that \begin{eqnarray*} &&B_{n, 5}(3/2) = -\frac{1}{256}\big(10935 (n-3)^5+317115 (n-3)^4+3347730 (n-3)^3+15229218 (n-3)^2\\ & & +25627463 (n-3)+4029203\big),   \end{eqnarray*}   and thus $B_{n, 5}(3/2) < 0$ for $ n \geq 3$. (In fact, we can see that $B_{2, 5}({x_{(6)}}) > 0$ for all ${x_{(6)}}$.)

% For $ m =4$ we see that \begin{eqnarray*} &&B_{n, 4}(3/2) = -\frac{1}{256}\big(8748 (n-3)^5+229635 (n-3)^4+2233980 (n-3)^3+9684066 (n-3)^2\\ & & +16984288 (n-3)+6360083\big),   \end{eqnarray*} and thus $B_{n, 4}(3/2) < 0$ for $ n \geq 3$. (In fact, we can see that $B_{2, 4}({x_{(6)}}) > 0$ for all ${x_{(6)}}$.)

% For $ m =3$ we see that \begin{eqnarray*} &&B_{n, 3}(4/3) = -\frac{16}{6561}\big(6144 (n-2)^5+127616 (n-2)^4+966112 (n-2)^3+3143112 (n-2)^2\\& & +3674626 (n-2)+24035\big),   \end{eqnarray*}  and thus $B_{n, 3}(4/3) < 0$ for $ n \geq 2$. 
   
%For $ m =2$ we see that \begin{eqnarray*} &&B_{n, 2}(4/3) = -\frac{4}{6561}\big(16384 (n-2)^5+287744 (n-2)^4+1890816 (n-2)^3+5611736 (n-2)^2 \\ & & +6881500 (n-2)+1884015\big),   \end{eqnarray*}  and thus $B_{n, 2}(4/3) < 0$ for $ n \geq 2$. 

 Finally, the value of $B_{n, m}(x_{(6)})$ at $x_{(6)} = 2$ is 
\begin{eqnarray*}
&& B_{n, m}(2) = 4 (2 m^5 n+m^4 (8 n^2+5)+4 m^3 n (2 n^2+9)+m^2 (84 n^2-2)+m (80 n^3-6 n)\\
&& +32 n^4-8 n^2+1)
\end{eqnarray*}
and we observe that for all positive integers $n ,m$ we have $B_{n, m}(2)>0$. 
 
From the above it follows that there exist at least two positive solutions for $x_{(6)} = \al_1, \beta_1$, where
\begin{center}
\begin{tabular}{|c|ccc|c|}
%& & $B_{n, m}(3/2)$   & & $B_{n, m}(x_{(6)})$    \\
 \hline
 $m \geq 8$ & &  $n\geq m/2$ & &  $0< \al_1 < 3/2$, \  $3/2 < \beta_1 < 2$  \\
     \hline 
 $m=6,7$ & &  $n\geq 4$ & & $0< \al_1 < 3/2$,\   $3/2 < \beta_1 < 2$   \\
  \hline \ 
$m=4,5$& &   $n\geq 3$ & & $0< \al_1 < 4/3$, \  $4/3 < \beta_1 < 2$   \\  
 %   \thickline 
%& &$B_{n, m}(4/3)$ & & \\
\hline 
 $m=2,3$& &   $n\geq 2$ &  &  $0< \al_1 < 4/3$, \  $4/3 < \beta_1 < 2$\\
   \hline 
 \end{tabular} .
 \end{center}

 %$0< \al_1 < 3/2$ and $3/2 < \beta_1 < 2$  for $n\geq m/2$ and $m \geq 8$. 

\smallskip
Next, we substitute into (\ref{linear}) and we take real solutions for $u_1, v_4$ and $v_5$, so we must  prove that these are positive.  
We take the resultant  of the polynomials $B_{n, m}(x_{(6)})$ and the numerator of the rational function  $u_1 - H_1(x_{(6)})$ and  we obtain the polynomial %$h_1(u_1)$:
{\small \begin{eqnarray*}
&& q_1(u_1) = \Big(16 (m^2-1)^6 n^{16} (m^2+m n+1) (2 m n+1) (m^2+2 m n+2)^2 \times\\
& & \big(32 m^5 n+144 m^4
 -96 m^3 n^3  +336 m^3 n-176 m^2 n^4+384 m^2 n^2-24 m^2 \\&&
 -96 m n^5+168 m n^3-6 m n-16 n^6+24 n^4+1\big)\Big) h_1(u_1),
\end{eqnarray*}}
where 
{\small \begin{eqnarray*}
&& h_1(u_1)=(m-1)^2 (m+1)^2 n^4 (3 m+2 n)^4 u_1^{8} -4 (m-1) (m+1) (7 m^2-4) n^4 (2 m+n) \times \\ 
&&(3 m +2 n)^3 u_1^7 + n^3 (3 m+2 n)^2 (4 \left(73 m^4-88 m^2+24\right) n^3+m \left(1197 m^4-1448 m^2+404\right) n^2
\\&&+(1219 m^6-1505 m^4 +396 m^2+16) n+m \left(120 m^6-133 m^4-2
   m^2+24\right))u_1^6  
\\
&&
 -2 n^3 (2 m+n) (3 m+2 n) \big(32 (3 m^2-1)(7 m^2-4) n^3+2 m (1495 m^4-1376 m^2+316) n^2 \\
&& +(3235 m^6-3016 m^4+492 m^2+96) n
+m \left(867 m^6-674 m^4-208 m^2+144\right)\big)u_1^5 
    \end{eqnarray*} 
\begin{eqnarray*} 
&&
+n^2 \big(256 \left(3 m^2-1\right)^2 n^6+64 m \left(353 m^4-244 m^2+47\right) n^5
+8 (10011 m^6-6816 m^4+1138 m^2
\\ &&
+96 ) n^4
+4 m \left(32552 m^6-20249 m^4+1176 m^2+1116\right) n^3  +2 (50219 m^8-24032 m^6
\\&&
-7052 m^4+4176
   m^2+48) n^2+m (32406 m^8-7077 m^6-15356 m^4 +5192 m^2+288) n \\&&
   +m^2 (3060 m^8-1227 m^6-1587 m^4+172 m^2+216)\big)u_1^4 
    -8 n^2 (2 m+n) \big(16 m (37 m^4 -22 m^2 \\&&
     +5) n^4+4 \left(731 m^6-319 m^4+26 m^2+16\right)
   n^3+4 m (1033 m^6-63 m^4-244 m^2 +94) n^2\\&&
   +(2131 m^8 +1265 m^6-1530 m^4+356 m^2+48) n+m (351 m^8+467 m^6-339 m^4-77
   m^2+ \\&&
   72 )\big)u_1^3 + 4 n \big(48 m^2 \left(19 m^4-10 m^2+3\right) n^5+4 m \left(1159 m^6-226 m^4-33
   m^2+60\right) n^4 \\&&
   +4(1915 m^8+750 m^6 -545 m^4+186 m^2+24) n^3+m  (4827 m^8+6834 m^6-2097 m^4-128 m^2\\&&
    +444 ) n^2+ (1093 m^{10} +2998 m^8 +713 m^6-1190 m^4+472 m^2+16 ) n+m  (36 m^{10}+230  m^8 \\&&
    +93 m^6-130 m^4+19 m^2+24 )\big)u_1^2   -8 n (2 m+n) (m^2+2 m n+2) (9 m^8+46 m^6+21 m^4 \\ &&
    -16 m^2+8+2 \left(39 m^4-18 m^2+7\right) m^2 n^2 +\left(77
   m^6+88 m^4-47 m^2+22\right) m n)u_1  + 4 (5 m^4\\&&
   -2 m^2+1) (m^2+m n+1) (2 m n+1) (m^2+2 m n+2)^2.
\end{eqnarray*} }
We observe that the coefficients of the polynomial $h_1(u_1)$ are positive for even degree terms and negative for odd degree terms.  
Thus if the equation $h_1(u_1) = 0$ has real solutions, then these are all positive.  
By the same way we take the resultant for the polynomials $B_{n, m}(x_{(6)})$ and the numerator of the rational function $v_4 - H_4(x_{(6)})$ 
and we obtain the polynomial %$h_4(v_4)$:
 \begin{eqnarray*}
&& q_4(v_4)= -16 n^{20}\big(-32 m^5 n-144 m^4+48 m^3 (2 n^2-7) n+8 m^2 (22 n^4-48 n^2+3)\\
&&+6 m (16 n^5 -28 n^3+n)+16n^6-24 n^4-1\big) h_4(v_4), 
\end{eqnarray*} 
where 
{\small \begin{eqnarray*}
&& h_4(v_4) = (2 m+n)^8 (3 m+2 n)^4 (m^2+m n+1) (2 m n+1) (m^2+2 m n +2)^2 v_4^8 \\
&& -2 n (2 m+n)^8 (3 m+2 n)^3 (m^2+2 m n+2) \big(6 m^6+68 m^5 n+128 m^4 n^2+41 m^4+64 m^3 n^3 \\
&&+172 m^3 n +132 m^2 n^2+54 m^2+84 mn+16\big)v_4^7 +  n (2 m+n)^6 (3 m+2 n)^2 \big(80 m^{11}
 \\&&
+3060 m^{10} n+22010 m^9 n^2 +544 m^9+64352 m^8 n^3+14245 m^8 n+92232 m^7 n^4+75004 m^7 n^2 \\&&
+619m^7+68384 m^6 n^5 +151044 m^6 n^3+20896 m^6 n+25088 m^5 n^6+139312 m^5 n^4 
\\&&
+75512 m^5 n^2 +224 m^5+3584 m^4 n^7 +59328 m^4 n^5+94496 m^4n^3+11436 m^4 n+9408 m^3 n^6 \\&&
+48432 m^3 n^4+24064 m^3 n^2+192 m^3 +8640 m^2 n^5+15776 m^2 n^3+1968 m^2 n+3200 m n^4 
\\&&+1776 m n^2+96 m +384n^3+64 n \big)v_4^6   -2 n^2 (2 m+n)^6 (3 m+2 n) \big(2876 m^{11}+37908 m^{10} n \\
&&
 +177352 m^9 n^2  +9199 m^9+401744 m^8 n^3  +108038 m^8 n+485856 m^7 n^4+404452 m^7n^2 \\&&
 +6988 m^7+319744 m^6 n^5 +667416 m^6 n^3  +96400 m^6 n+107520 m^5 n^6+541664 m^5 n^4 \\&&
 +269168 m^5 n^2+2948 m^5+14336 m^4 n^7 +211520 m^4n^5  +293408 m^4 n^3+34312 m^4 n \\&&
 +31360 m^3 n^6+138304 m^3 n^4+59424 m^3 n^2+2064 m^3 +23040 m^2 n^5  +35264 m^2 n^3\\&&
 +6080 m^2 n+6400 mn^4+3968 m n^2+576 m+512 n^3+384 n \big)v_4^5 + n^2 (2 m+n)^4 \times \\
&&
\big(16080 m^{14}+520608 m^{13} n+4322720 m^{12} n^2+14040 m^{12}+16938600 m^{11} n^3  +967712 m^{11} n\\&&
+37781312 m^{10}n^4+7795248 m^{10} n^2+15289 m^{10}+51970688 m^9 n^5+26578272 m^9 n^3
\end{eqnarray*} }
 {\small \begin{eqnarray*}
&& +524192 m^9 n+45604736 m^8 n^6+48503328 m^8 n^4+4250984 m^8n^2+11480 m^8
\\&&
+25566080 m^7 n^7  +52079872 m^7 n^5 +12309456 m^7 n^3+241944 m^7 n+8849920 m^6 n^8\\&&
+33943424 m^6 n^6+17616272 m^6 n^4
 +1308512m^6 n^2+1544 m^6+1720320 m^5 n^9
 \\&&
 +13190400 m^5 n^7+13936000 m^5 n^5+2594464 m^5 n^3 +133056 m^5 n+143360 m^4 n^{10}\\
 &&
 +2805760 m^4n^8  +6201088 m^4 n^6+2454848 m^4 n^4+368864 m^4 n^2 +1264 m^4+250880 m^3 n^9\\ &&
 +1450240 m^3 n^7+1192960 m^3 n^5 +379328 m^3 n^3+21536 m^3n
 +138240 m^2 n^8  +283648 m^2 n^6\\&&
+172608 m^2 n^4+33664 m^2 n^2+864 m^2+25600 m n^7 +31232 m n^5+17856 m n^3 +1152 m n\\&&
+1024 n^6 +3072 n^4+384n^2\big) v_4^4  -16 n^3 (2 m+n)^4 \big(7808 m^{13}
+111360 m^{12} n+594784 m^{11} n^2 \\
&& +6356 m^{11}+1664096 m^{10} n^3+125340 m^{10} n+2765152 m^9 n^4+642880m^9 n^2 +5036 m^9\\&&
+2863040 m^8 n^5  +1579252 m^8 n^3+48298 m^8 n+1856640 m^7 n^6+2159232 m^7 n^4
\\&&
+219000 m^7 n^2+2819 m^7 +729216 m^6n^7 +1718224 m^6 n^5+454480 m^6 n^3+26488 m^6 n\\&&
+157696 m^5 n^8+786208 m^5 n^6+472096 m^5 n^4 +78116 m^5 n^2 +408 m^5+14336 m^4 n^9\\&&
+190464m^4 n^7+254976 m^4 n^5+92712 m^4 n^3+9592 m^4 n+18816 m^3 n^8
 +67712 m^3 n^6 \\&&
 +49408 m^3 n^4+17184 m^3 n^2+388 m^3+6912 m^2 n^7+10880 m^2n^5+10272 m^2 n^3+1248 m^2 n\\
&& +640 m n^6+1920 m n^4+872 m n^2+144 m+128 n^3+96 n\big) v_4^3 + 
16 n^3 (2 m+n)^2 \big(2240 m^{15}\\&&
+96128 m^{14} n 
+848320 m^{13} n^2+2328 m^{13}+3499648 m^{12} n^3+56304 m^{12} n+8416192 m^{11} n^4
\\&&+500432m^{11} n^2+544 m^{11} +12975200 m^{10} n^5+1956172 m^{10} n^3+38600 m^{10} n +13339312 m^9 n^6
\\&&
+4349536 m^9 n^4+180104 m^9 n^2 +620m^9 +9219520 m^8 n^7+5940704 m^8 n^5+543914 m^8 n^3\\&&
+14624 m^8 n+4215392 m^7 n^8+5109920 m^7 n^6  +976656 m^7 n^4+109640 m^7 n^2+431m^7\\&&
+1217920 m^6 n^9+2760400 m^6 n^7+1006392 m^6 n^5 +272120 m^6 n^3+1808 m^6 n+200704 m^5 n^{10}\\&&
+904704 m^5 n^8+597264 m^5 n^6+311304 m^5n^4+19228 m^5 n^2  -46 m^5+14336 m^4 n^{11}\\&&
+163840 m^4 n^9+199680 m^4 n^7 +180544 m^4 n^5+36856 m^4 n^3+2872 m^4 n +12544 m^3 n^{10}\\&&
+34368m^3 n^8+51584 m^3 n^6+27072 m^3 n^4+5776 m^3 n^2 +5 m^3+2304 m^2 n^9+5760 m^2 n^7\\
&& +8544 m^2 n^5+3616 m^2 n^3+400 m^2 n+960 m n^6+720 m n^4+404 m n^2+24 m +96 n^3+16 n\big)v_4^2\\
&&  -256 n^4 (2 m+n)^2 \big(496 m^{14}+8080 m^{13} n+44736 m^{12} n^2+472 m^{12}
+130952 m^{11} n^3\\&&
+2136 m^{11} n +236456 m^{10} n^4+7056 m^{10}n^2+84 m^{10}+280320 m^9 n^5+20276 m^9 n^3 \\&&
+1968 m^9 n+221432 m^8 n^6 +40056 m^8 n^4+7344 m^8 n^2+56 m^8+114488 m^7 n^7+47760 m^7n^5\\&&
+14298 m^7 n^3+350 m^7 n +36864 m^6 n^8+33170 m^6 n^6+15930 m^6 n^4+2028 m^6 n^2+38 m^6 \\&&
+6656 m^5 n^9+13056 m^5 n^7 +9600 m^5 n^5 +4107m^5 n^3+111 m^5 n+512 m^4 n^{10}+2688 m^4 n^8\\&&
+2880 m^4 n^6+3312 m^4 n^4  +384 m^4 n^2-9 m^4 +224 m^3 n^9 +336 m^3 n^7+1152 m^3 n^5\\&&
+312 m^3n^3+108 m^3 n+144 m^2 n^6 
+72 m^2 n^4+120 m^2 n^2 +2 m^2+30 m n^3+5 m n+2 \big) v_4 \\&&
+64 n^4 (4 m^4+32 m^3 n+32 m^2 n^2+4 m^2+8 m n^3+1) (64 m^3 n^9+768 m^4 n^8+3840 m^5 n^7\\&&
+4 m^2 \left(2629 m^4-24 m^2+12\right) n^6+96 m^3 \left(183 m^4-8 m^2+4\right) n^5 +96 m^4 (197 m^4-24 m^2\\ &&
+12 ) n^4+4 m \left(3316 m^8-810 m^6+417 m^4-12 m^2+3\right) n^3 +48 m^2 \left(8 m^4-2 m^2+1\right) \times \\&&
(15 m^4-2 m^2+1 ) n^2+48 m^3 \left(28 m^8-14 m^6+11 m^4-4 m^2+1\right) n +\left(4 m^4+2 m^2-1\right)^2 \times \\&&
\left(5 m^4-2 m^2+1\right)).
\end{eqnarray*} }

We observe that the coefficients of the polynomial $h_4(v_4)$ are positive for even degree terms and negative for odd degree terms.  Thus if the equation $h_4(v_4) = 0$ has real solutions, then these are all positive. 
 \end{proof}
    
We finally {\it conjecture}  that the Stiefel manifolds $V_2\mathbb{C}^{n+2}=\SU(n+2)/\SU(n)$  admit  precisely two invariant Einstein merics, which are of Jensen's type.  This is the analogue of the real Stiefel manifolds
$V_2\mathbb{R}^{n+2}=\SO(n+2)/\SO(n)$, which had been studied before by other authors (eg. \cite{Ke}).

 \end{document}